\documentclass[12pt,reqno]{article}
\usepackage[a4paper,margin=27mm]{geometry}
\usepackage[T1]{fontenc}
\usepackage[utf8]{inputenc}
\usepackage{lmodern,microtype,amsmath,amssymb,amsthm,mathtools,mathrsfs}
\usepackage{enumitem,booktabs,array,longtable,xcolor}
\usepackage[colorlinks=true,linkcolor=blue!45!black,citecolor=green!35!black,urlcolor=blue!55!black]{hyperref}
\usepackage[nameinlink,capitalise,noabbrev]{cleveref}
\definecolor{equationrefcolor}{RGB}{170,35,35}
\let\originaleqref\eqref
\renewcommand{\eqref}[1]{{\hypersetup{linkcolor=equationrefcolor}\originaleqref{#1}}}
\numberwithin{equation}{section}
\newtheorem{theorem}{Theorem}[section]
\newtheorem{proposition}[theorem]{Proposition}
\newtheorem{lemma}[theorem]{Lemma}
\newtheorem{corollary}[theorem]{Corollary}
\theoremstyle{definition}
\newtheorem{remark}[theorem]{Remark}

\newcommand{\R}{\mathbb R}

\newcommand{\loc}{\mathrm{loc}}

\newcommand{\Sph}{\mathbb S^{N-1}}

\setlist{nosep,leftmargin=2em}
\title{\bf The Hamel–Nadirashvili conjecture for the Fisher–KPP equation}
 \author{\parbox{0.95\textwidth}{\centering
   \normalsize
   Hongjun Guo\textsuperscript{a},
   Kelei Wang\textsuperscript{b},
   Maolin Zhou\textsuperscript{c}\\[0.6em]
   \small
   \textsuperscript{a}School of Mathematical Sciences,
   Key Laboratory of Intelligent Computing and Applications
   (Ministry of Education), Institute for Advanced Study,
   Tongji University, Shanghai, China\\[0.35em]
   \textsuperscript{b}School of Mathematics and Statistics,
   Wuhan University, Wuhan, China\\[0.35em]
   \textsuperscript{c}Chern Institute of Mathematics,
   Nankai University, Tianjin, China
 }}
\date{}
\begin{document}
\maketitle
\begin{abstract}
We resolve the Hamel–Nadirashvili conjecture, which concerns the classification of all \([0,1]\)-valued entire solutions of Fisher–KPP equations with concave \(C^2\) reaction terms in arbitrary space dimension. We prove that every such solution is a \(C_{\mathrm{loc}}^{1,2}\) limit of solutions generated by finitely many planar traveling fronts, possibly together with a spatially homogeneous solution. Our approach reveals an ordered linear structure underlying the nonlinear solution space, leading to a critical–subcritical spectral decomposition. The proof combines linear reconstruction of subcritical approximants with a measure representation (a kind of nonlinear Martin representation) of critical solutions, obtained through small-amplitude asymptotics and comparison with a radially symmetric entire solution.
\end{abstract}
\par
\tableofcontents
\clearpage

\section{Introduction and main results}\label{sec:introduction}

In this paper, we consider entire solutions of the homogeneous Fisher--KPP equation
\begin{equation}\label{eq:kpp}
 u_t=\Delta u+f(u), 
 \qquad\text{for }(t,x)\in\R\times\R^N,
\end{equation}
 with $N\geq1$ and a reaction satisfying
\begin{equation}\label{eq:general-kpp-assumptions}
 \begin{gathered}
 f\in C^2([0,1]),\qquad f(0)=f(1)=0,\qquad
 f(u)>0\quad\text{for }0<u<1,\\
 f''(u)\leq0\quad\text{for }0\leq u\leq1.
 \end{gathered}
\end{equation}
For convenience, we extend $f$ to a concave
$C^1$ function on $\R$ by setting $f(u)=f'(0)u$ for $u<0$ and
$f(u)=f'(1)(u-1)$ for $u>1$.
Concavity and positivity imply
\[
 f'(0)>0,\qquad f'(1)<0,\qquad f(u)\leq f'(0)u
 \quad\text{for }0\leq u\leq1.
\]
A typical example is the logistic nonlinearity, that is, $f(u)=u(1-u)$ for $u\in[0,1]$.
This equation was introduced
in the pioneering works of Fisher \cite{Fisher1937} and
Kolmogorov, Petrovskii, and Piskunov \cite{KPP1937}. Equation \eqref{eq:kpp} also arises in models of population dynamics and other biological phenomena; see \cite{AronsonWeinberger1978,BarenblattZeldovich1971,Fife1979,Fisher1937,Freidlin1995,Murray1989,Rothe1981,Stokes1976}. 
An entire solution is a classical solution defined for $t\in\R$ and $x\in\R^N$. In this work, we only consider entire solutions between $0$ and $1$. The strong maximum principle implies that it is either identically $0$, identically $1$, or strictly between $0$ and $1$.
Let $\mathcal E$ denote this class of entire solutions.
Our purpose is to prove the Hamel--Nadirashvili closure conjecture under \eqref{eq:general-kpp-assumptions}, that is, every entire solution between $0$ and $1$ can be
approximated by Hamel--Nadirashvili solutions generated by traveling fronts and the spatially homogeneous solution.

A typical entire solution of \eqref{eq:kpp} is the traveling wave solution. 
 It is well known that, in dimension $N=1$, \eqref{eq:kpp} admits traveling wave solutions $u(t,x)=\phi_c(x+ct)$ if and only if $c\ge c_*=2\sqrt{f'(0)}$, where $c_*$ is called the minimal, or critical speed, see \cite{AronsonWeinberger1978,Bramson1983,Fife1979,Freidlin1995,HadelerRothe1975,Kanel1961,Rothe1981,Stokes1976}. The profile $\phi_c$ and the speed $c$ of the traveling wave solution satisfy
\begin{equation}\label{eq:critical-front}
  \phi''-c\phi'+f(\phi)=0,
  \qquad
  \phi(-\infty)=0,\quad \phi(+\infty)=1,\quad
  \phi'>0.
\end{equation}
The profile $\phi_c$ is increasing and unique up to
translation. Since \(\phi_c\) connects the two equilibrium states \(0\) and \(1\), we also call $\phi_c(x+ct)$ traveling front. One of the important features of $\phi_c(s)$ is its leading-edge tail near the unstable equilibrium $0$ as $s\rightarrow -\infty$. By the phase-plane analysis, one knows that the leading-edge tail is asymptotically governed by the linearized  equation $\phi''-c\phi'+f'(0)\phi=0$. This ODE has eigenvalues 
\begin{equation*}
  \lambda(c)
  =\frac{c-\sqrt{c^2-4f'(0)}}{2}>0,\qquad  \lambda^+(c)
  =\frac{c+\sqrt{c^2-4f'(0)}}{2}>0.
\end{equation*}
In particular, $\lambda_*:=\lambda(c_*)=\lambda^+(c_*)=c_*/2=\sqrt{f'(0)}$ is the double root of the quadratic equation $\lambda^2-c_*\lambda +f'(0)=0$.  Up to translation, one has the following leading-edge tail of $\phi_c$
\begin{equation}\label{eq:intro-tail}
\hbox{for } c>c_*, \qquad \phi_c(s)\sim e^{\lambda(c) s}\quad \hbox{as } s\rightarrow -\infty,
\end{equation}
and \begin{equation}\label{eq:intro-critical-tail}
 \phi_*(s):=\phi_{c_*}(s)\sim |s|e^{\lambda_*s}
 \quad\text{as }s\to-\infty.
\end{equation}
More precisely, there exist
$B\in\R$ and $\delta>0$ such that
\begin{equation}\label{eq:intro-critical-expansion}
 \phi_*(s)=e^{\lambda_*s}\bigl(-s+B+O(e^{\delta s})\bigr)
 \quad\text{as }s\to-\infty.
\end{equation}
The standard asymptotic estimates \eqref{eq:intro-tail} and
\eqref{eq:intro-critical-tail} can be found in
\cite{AgmonNirenberg1963, BerestyckiNirenberg1992, Bramson1983,CoddingtonLevinson1955, HadelerRothe1975,Kametaka1976,Pazy1967, Uchiyama1978}. 
We also use the logarithmic derivative estimates
\begin{equation}\label{eq:intro-wave-slope}
 0<\frac{\phi_c'(s)}{\phi_c(s)}<\lambda(c)
 \quad\text{for }s\in\R,\qquad
 \frac{\phi_c'(s)}{\phi_c(s)}\rightarrow\lambda(c)
 \quad\text{as }s\to-\infty,
 \qquad c\geq c_*.
\end{equation}
The estimates \eqref{eq:intro-wave-slope} follow from
\cite{HamelRossi2016}.
In higher dimensions, the corresponding planar traveling fronts are $\phi_c(x\cdot \nu+ct)$ for $\nu\in\Sph$. Throughout, the traveling fronts are fixed by \eqref{eq:intro-tail} and \eqref{eq:intro-critical-tail}.

In past decades, the theory of traveling wave solutions has been highly developed. When the coefficients of the equation depend on space and/or time, translation invariance is lost and such standard traveling fronts generally cease to exist, but new analogous notions have been introduced. In spatially and/or
temporally periodic heterogeneous media, the natural analogue is a pulsating traveling front: the front propagates with a well-defined
average speed while its profile is periodically modulated by the medium.
Existence, uniqueness, stability, and spreading-speed properties of such
fronts for KPP-type equations have been extensively studied, see
\cite{BerestyckiHamel2002,BoutillonHamelRoques2026,DengDucrot2023,
HamelRoques2011,LiangZhao2007,LiangZhao2010,Nadin2009,NolenRuddXin2005,
ShigesadaKawasakiTeramoto1986,Weinberger2002}.
This framework has also been extended to almost-periodic, almost-automorphic,
recurrent, random, and uniquely ergodic settings, see
\cite{HuangShen2009,LiangZhou2022,NadinRossi2017,SalakoShen2019,
Shen2001,Shen2004,Shen2010,Shen2011b,Shen2011a}.
For still more general nonperiodic or nonautonomous heterogeneous media, the
notions of generalized transition front provide a flexible description of entire solutions connecting the invaded and
uninvaded states through an interface of uniformly bounded width, without requiring a rigid profile, a constant instantaneous speed, or periodicity, see
\cite{BerestyckiHamel2007,BerestyckiHamel2012,HamelNadin2022,
HamelRossi2015,HamelZhang2025,Zlatos2012}.

Entire solutions describe complete trajectories of the parabolic dynamics, connecting prescribed asymptotic patterns in the past to their evolution.  After a suitable space-time rescaling, they may also arise as local limiting profiles near regions where several interfaces interact. A natural strategy of  constructing 
entire solutions is taking one or several traveling fronts as their asymptotic
building blocks.  For bistable
equations, where $f$ satisfies $f(0)=f(1)=f(\theta)=0$, $f<0$ on $(0,\theta)$ and $f>0$ on $(\theta,1)$ for some $\theta\in (0,1)$, the classical two front constructions describe oppositely oriented
fronts approaching and merging, see
\cite{ChenGuo2005,ChenGuoNinomiya2006,
ChenGuoNinomiyaYao2018,GuoMorita2005,Yagisita2003}.  Recent higher-dimensional results mix finitely many
planar fronts to obtain transition fronts with changing shape, or obtain
lower-dimensional and polyhedral entire solutions as limits of
multidimensional traveling fronts
\cite{GuoWang2025,NinomiyaTaniguchi2025,Taniguchi2025}.

\subsection*{Hamel--Nadirashvili conjecture}

We first recall the work of Hamel and Nadirashvili and their conjecture \cite{HN1999,HN2001}. In the one-dimensional setting of \eqref{eq:kpp}, they \cite{HN1999} used two basic building blocks to construct entire solutions. The first is the spatially homogeneous solution \(\xi_f(t)\), defined by
\begin{equation}\label{eq:homogeneous-normalization}
 \xi_f'(t)=f\bigl(\xi_f(t)\bigr),\qquad
 0<\xi_f(t)<1,\qquad
 \xi_f(t)\sim e^{f'(0)t}\quad\text{as }t\to-\infty.
\end{equation}
This normalization determines $\xi_f$ uniquely. It is increasing and
converges to $1$ as $t\to+\infty$; see \cite{HN2001}.
For $K>0$, the homogeneous solution with leading coefficient $K$ is $\xi_f(t+\ln K/f'(0))$. The second consists of two counter-propagating traveling fronts,
\[
  u_{c,h}^{+}(t,x)=\phi_c(x+ct+h),
  \qquad
  u_{c',h'}^{-}(t,x)=\phi_{c'}(-x+c't+h'),
\]
where \(c,\,c'>c_*\). As $(c,h)$ and $(c',h')$ vary, these solutions form two dimensional manifolds of entire solutions. By mixing such fronts, with or without a time translate of \(\xi_f\), they constructed finite dimensional manifolds of entire solutions.

In arbitrary dimensions, they \cite{HN2001} extended this construction by mixing planar fronts through suitable Radon measures. More precisely, let
\[
  X:=\bigl(\Sph\times[c_*,+\infty)\bigr)\cup\{\infty\}.
\]
The set \(\Sph\times[c_*,+\infty)\) is endowed with its usual topology, while a subset \(A\subset X\) is a neighborhood of \(\infty\) if \(\infty\in A\) and
\[
  \Sph\times[c_0,+\infty)\subset A
\]
for some \(c_0\ge c_*\). Thus, \(X\) is a compact direction-speed space, and the additional point \(\infty\) represents the spatially homogeneous solution, viewed as a front with infinite speed.
Throughout, a positive measure means a nonnegative measure that is not identically zero.
Let \(\mathcal M\) denote the class of  finite positive Radon measures \(\vartheta\) on \(X\) such that the restriction of \(\vartheta\) to the critical speed set is a finite sum of Dirac masses:
\[
  \left.\vartheta\right|_{\Sph\times\{c_*\}}
  =\sum_{i=1}^{k}m_i\delta_{(\nu_i,c_*)},
\]
where \(k\ge0\), \(m_i>0\), and the directions \(\nu_i\in \Sph\) are pairwise distinct. This finite atomic restriction is built into their admissible class and construction.
For every \(\vartheta\in\mathcal M\), there exists an entire solution \(u_\vartheta(t,x)\) of \eqref{eq:kpp} which can be viewed as a mixture of a measurable set of traveling fronts. Moreover, the map
\[
 \vartheta\mapsto u_\vartheta
\] is one-to-one on $\mathcal M$. Its restriction to
\[
 \widehat{\mathcal M}
 :=\{\vartheta\in\mathcal M:
       \vartheta(\Sph\times\{c_*\})=0\}
\]
is continuous in the topology of \cite[Theorem~1.2]{HN2001}.
Their construction yields an infinite-dimensional family of entire solutions. Finite sums of Dirac masses recover finite mixtures of planar fronts, whereas more general measures produce multidimensional solutions, including nonplanar traveling fronts and radially symmetric entire solutions.
They also proved the following partial uniqueness result. Let \(u(t,x)\) be an entire solution with $0<u<1$. If there exists \(\varepsilon>0\) such that
\begin{equation}\label{eq:intro-hn-past-decay}
  \max_{|x|\le(c_*+\varepsilon)|t|}u(t,x)\rightarrow0
  \qquad\text{as }t\to-\infty,
\end{equation}
then there exists a unique measure \(\vartheta\in\mathcal M\) such that
\[
  u(t,x)=u_\vartheta(t,x)
  \qquad\text{for all }(t,x)\in\mathbb R\times\mathbb R^N,
\]
and $\operatorname{supp}\vartheta
  \subset
  \Sph\times[c_*+\varepsilon,+\infty)\cup\{\infty\}$.

Finally, they conjectured that the full class \(\mathcal E\) of entire solutions satisfying \(0\le u\le1\) is the closure of their measure-generated family:
\begin{equation}\label{eq:intro-hn-conjecture}
  \mathcal E
  =
  \overline{\bigl\{u_\vartheta:\vartheta\in\mathcal M\bigr\}}^{\,\mathcal T},
\end{equation}
where \(\mathcal T\) denotes $C^{1,2}_{\loc}$ topology (locally uniform convergence of the functions, their first two spatial derivatives, and their first time derivative). If this conjecture holds, every entire solution in $\mathcal E$ would be
obtainable, in the topology $\mathcal T$, from mixtures of traveling fronts. Thus, these
finite and infinite speed fronts would constitute a nonlinear generating family for \eqref{eq:kpp}, analogous to the exponential functions in Widder's
representation of nonnegative entire solutions of the linear heat equation
$\partial_t v=\Delta v$ on $\mathbb R^N\times\mathbb R$, see \cite{Widder1975} and Lin and Zhang \cite{LinZhang2019} for a Martin boundary formulation and extensions to complete noncompact manifolds.

\subsection*{The closure theorem}

We prove \eqref{eq:intro-hn-conjecture} under the concavity assumptions \eqref{eq:general-kpp-assumptions}. In fact, every entire solution can be approximated in $C^{1,2}_{\loc}(\R\times\R^N)$ by Hamel--Nadirashvili entire solutions, generated by finitely many planar traveling fronts and possibly one spatially homogeneous solution. The number of fronts and their parameters may vary along the approximating sequence. To state this precisely, let $m\ge0$ be an integer and consider
the planar fronts
\[
 v_i(t,x)=\phi_{c_i}(x\cdot\nu_i+c_it+h_i),
 \qquad i=1,\cdots,m,
\]
where $\nu_i\in\Sph$, $c_i\geq c_*$, and $h_i\in\R$.
We may also include the spatially homogeneous solution
\[
 v_0(t)=\xi_f(t+h_0),\qquad h_0\in\R.
\]
At least one component must be present: if $m=0$, the homogeneous
solution $v_0$ is required; if $m\ge1$, it is optional.
The pairs $(c_i,\nu_i)$, $i=1,\cdots,m$, are required
to be pairwise distinct.
By \cite[Theorem~1.1]{HN2001}, there exists an entire solution $u$
satisfying
\begin{equation}\label{eq:intro-finite-interval}
 \max_i v_i(t,x)\leq u(t,x)
 \leq\min\{\textstyle\sum_i v_i(t,x),1\}
 \qquad\text{for }(t,x)\in\R\times\R^N.
\end{equation}
We call this the \emph{finite HN solution}.
Uniqueness of the solution within these bounds is established in
Theorem~\ref{thm:packet}. The identification of our finite approximants
with these HN solutions, including their wave phases, is given in
the proof of Theorem~\ref{thm:main-closure}.
Equivalently, these finite HN solutions are the solutions
$u_\vartheta$ associated with nonzero finite sums of Dirac masses
on $X$, see \cite{HN2001}.

\begin{theorem}
\label{thm:main-closure}
For every entire solution $u$ of \eqref{eq:kpp} with $0\leq u\leq1$,
there exists a sequence of finite positive measures
\[
 \vartheta_k
 =\sum_{i=1}^{m_k}a_{k,i}\delta_{(\nu_{k,i},c_{k,i})}
   +a_{k,0}\delta_\infty\in\mathcal M,
 \qquad k=1,2,\cdots,
\]
where $m_k\geq0$ is an integer, $a_{k,0}\geq0$, and
\[
 a_{k,i}>0,\qquad \nu_{k,i}\in\Sph,\qquad
 c_{k,i}\in[c_*,+\infty)
 \qquad\text{for }i=1,\cdots,m_k,
\]
such that  
\[
 u_{\vartheta_k}\rightarrow u
 \quad\text{in }C^{1,2}_{\mathrm{loc}}(\R\times\R^N)
 \quad\text{as }k\to\infty.
\]
\end{theorem}

The term $a_{k,0}\delta_\infty$ accounts for the optional homogeneous
solution. The constant solutions $0$ and $1$
are included as limits of time translates of $\xi_f$.
The conclusion confirms Hamel--Nadirashvili's conjecture for every reaction satisfying \eqref{eq:general-kpp-assumptions}.

\subsection*{Entire solutions from positive Borel measures on $\Sph\times\{c_*\}$}

A key step in the proof of Theorem~\ref{thm:main-closure} is to
construct and identify the critical part of an entire solution.
This leads to a further result of independent interest: we remove
the finite atomic restriction in the construction of Hamel and
Nadirashvili \cite{HN1999,HN2001} on the critical set
$\Sph\times\{c_*\}$ and construct canonical entire solutions
for arbitrary finite positive Borel measures on this set.
Since $\Sph$ and $X$ are compact metric spaces, every finite
Borel measure on either space is a Radon measure. Thus our use
of Borel measures is consistent with the terminology of
Hamel and Nadirashvili.

Let $\mu$ be a finite positive Borel measure on $\Sph$ and set
$M:=\mu(\Sph)>0$. Define its Laplace transform and log--Laplace phase by
\begin{equation}\label{eq:critical-measure-phase}
 L_\mu(x):=\int_{\Sph}e^{\lambda_*x\cdot\nu}\,d\mu(\nu),
 \qquad z_\mu(t,x):=c_*t+\lambda_*^{-1}\ln L_\mu(x).
\end{equation}
Multiplying $\mu$ by a positive constant translates this phase $z_\mu$ in time.

A natural first attempt is to apply the construction strategy of Hamel
and Nadirashvili directly. Consider the normalized $\mu$-mixture of
critical traveling fronts
\[
 P_\mu(t,x):=\frac1M\int_{\Sph}
 \phi_*\left(x\cdot\nu+c_*t+\frac{\ln M}{\lambda_*}\right)\,d\mu(\nu).
\]
Since $d\mu/M$ has unit mass, concavity of $f$ and Jensen's inequality give
\[
 (P_\mu)_t-\Delta P_\mu
 =\frac1M\int_{\Sph}
 f\left(\phi_*\left(x\cdot\nu+c_*t
 +\frac{\ln M}{\lambda_*}\right)\right)\,d\mu(\nu)
 \leq f(P_\mu),
\]
so $P_\mu$ is a subsolution. The corresponding mixture of critical
exponentials,
\[
 Q_\mu(t,x):=\frac1M\int_{\Sph}
 e^{\lambda_*(x\cdot\nu+c_*t+\lambda_*^{-1}\ln M)}\,d\mu(\nu)
 =e^{\lambda_*c_*t}L_\mu(x)=e^{\lambda_*z_\mu(t,x)},
\]
satisfies $(Q_\mu)_t-\Delta Q_\mu=f'(0)Q_\mu$. Since
$f(q)\leq f'(0)q$ for $0\leq q\leq1$, the truncation
$\min\{1,Q_\mu\}$ is a supersolution of \eqref{eq:kpp}.
One might therefore try to construct an entire solution between
$P_\mu$ and $\min\{1,Q_\mu\}$. However, these barriers are not ordered.
Indeed, at $x=0$,
\[
 P_\mu(t,0)=\phi_*\left(c_*t+\frac{\ln M}{\lambda_*}\right),
 \qquad Q_\mu(t,0)=e^{\lambda_*c_*t+\ln M},
\]
and \eqref{eq:intro-critical-tail} yields
\[
 \frac{P_\mu(t,0)}{Q_\mu(t,0)}\sim c_*|t|\rightarrow+\infty
 \qquad\text{as }t\to-\infty.
\]
Thus $Q_\mu$ cannot dominate $P_\mu$ on the leading edge. The
obstruction is the polynomial factor in the critical tail of $\phi_*$.
Our construction instead starts from $\phi_*(z_\mu)$, which is a global
subsolution for every such measure, and uses a radially symmetric
entire solution to obtain an upper bound strictly below one.

The following theorem gives the critical-sphere construction and
the injectivity and continuity of its measure parametrization.
\begin{theorem}\label{thm:borel-main}
For every finite positive Borel measure $\mu$ on $\Sph$, there exists
a canonical classical entire solution $u_\mu$ of \eqref{eq:kpp} with
$\partial_tu_\mu>0$ and
\begin{equation}\label{eq:borel-main-squeeze}
 0<\phi_*(z_\mu(t,x))\leq u_\mu(t,x)<1
 \qquad\text{for all }(t,x)\in\R\times\R^N.
\end{equation}
Moreover,
\begin{equation}\label{eq:borel-main-extinction}
 \begin{aligned}
 u_\mu(t,\cdot)&\rightarrow0 &&\text{as }t\to-\infty,\\
 u_\mu(t,\cdot)&\rightarrow1 &&\text{as }t\to+\infty,
 \end{aligned}
 \qquad\text{in }C^2_{\loc}(\R^N).
\end{equation}
For every $\alpha\in(0,1)$, there exists $C_\alpha>0$ such that
\begin{equation}\label{eq:borel-main-level-strip}
 \{x\in\R^N:u_\mu(t,x)=\alpha\}
 \subset\{x\in\R^N:|z_\mu(t,x)|\leq C_\alpha\}
 \qquad\text{for every }t\in\R.
\end{equation}
Furthermore, the map $\mu\mapsto u_\mu$ is one-to-one and continuous from
the weak topology\footnote{Here $\mu_n\rightharpoonup\mu$ means
$\int_{\Sph}\varphi\,d\mu_n\to\int_{\Sph}\varphi\,d\mu$
for every $\varphi\in C(\Sph)$.}
of finite positive measures on $\Sph$ to
$C^{1,2}_{\loc}(\R\times\R^N)$.
\end{theorem}

The solution $u_\mu$ is the least entire solution lying above
$\phi_*(z_\mu)$. Theorem~\ref{thm:canonical-critical-solutions} gives
the construction and a radial upper bound, from which
\eqref{eq:borel-main-level-strip} follows. This phase-strip inclusion
does not, by itself, imply a uniform $O(1)$ Hausdorff-distance bound
between the level sets of $u_\mu$ and the zero level set of $z_\mu$,
since $|\nabla z_\mu|$ need not have a positive lower bound.
Injectivity is proved in Subsection~\ref{sec:critical-canonical-identification},
where Proposition~\ref{prop:faithfulness-weak-continuity} also establishes
weak continuity.
Comparison in the construction also gives
$u_\mu\leq u_\nu$ whenever $\mu\leq\nu$.

Constructing this family alone does not classify all entire
solutions outside the partial uniqueness regime
\eqref{eq:intro-hn-past-decay}. In fact, the failure of
\eqref{eq:intro-hn-past-decay} does not force a solution to be generated
solely by a measure on the critical sphere. An example in
\cite{GuoWangZhou} is generated by a finite positive measure
$\vartheta$ on $X$ with
$\vartheta(\Sph\times\{c_*\})=0$, yet fails
\eqref{eq:intro-hn-past-decay}. Thus the critical-sphere construction
complements Theorem~\ref{thm:main-closure}, whose approximation
statement applies to all entire solutions in $\mathcal E$.

\begin{remark}\label{rem:mixed-family-closure}
The closure can also be described using mixtures of the original
HN solutions with no critical-sphere mass and the critical solutions
of Theorem~\ref{thm:borel-main}. With the convention $u_0=0$,
\[
 \mathcal E
 =
 \overline{\left\{
 u_\vartheta\oplus u_\mu:
 \begin{array}{l}
 \vartheta\in\widehat{\mathcal M}\cup\{0\},\\
 \mu\text{ is a finite nonnegative Borel measure on }\Sph
 \end{array}
 \right\}}^{\,\mathcal T}.
\]
Here $\oplus$ denotes the nonlinear superposition constructed in
Section~\ref{sec:pde-dilations}.
The condition $\vartheta\in\widehat{\mathcal M}$ means
$\vartheta(\Sph\times\{c_*\})=0$.
Every element of the displayed family belongs to $\mathcal E$,
and the family contains all finite HN solutions by
the proof of Theorem~\ref{thm:main-closure}.
The equality therefore follows from Theorem~\ref{thm:main-closure}
and the closedness of $\mathcal E$ in $\mathcal T$.
\end{remark}

\subsection*{Notations and organization}

To compare nonlinear solutions of \eqref{eq:kpp} with solutions of the linearized
equation, we introduce the following notations. For a bounded measurable initial function $h:\R^N\to\R$, let $v$ and $w$ solve $v_t=\Delta v$ and $w_t=\Delta w+f'(0)w$ for $t>0$, $x\in\R^N$, with $v(0,x)=w(0,x)=h(x)$. Define the evolution operators
\[
 \bigl(P_t[h]\bigr)(x):=v(t,x),\qquad
 \bigl(L_t[h]\bigr)(x):=w(t,x)
 \qquad\text{for }t\geq0,\ x\in\R^N.
\]
Then $P_t[h]=e^{t\Delta}h$ and $L_t[h]=e^{f'(0)t}P_t[h]$.
For bounded nonnegative initial data $h$, similarly define
$\bigl(S_t[h]\bigr)(x):=u(t,x)$ for $t\geq0$ and $x\in\R^N$,
where $u$ solves
\begin{equation}\label{eq:cauchy-semigroup-definition}
\begin{cases}
 u_t=\Delta u+f(u), & t>0,\ x\in\R^N,\\
 u(0,x)=h(x), & x\in\R^N.
\end{cases}
\end{equation}
Here $f$ is understood through the concave extension fixed after
\eqref{eq:general-kpp-assumptions}.
The nonlinear semigroup satisfies $S_0[h]=h$ and
$S_{t+s}[h]=S_t[S_s[h]]$ for $t,s\geq0$.
Thus $L_t[h]$ and $S_t[h]$ are the solutions at time $t$ of the
linearized and nonlinear equations, respectively, with the same
initial datum $h$.

We also recall the representation used below: every nonnegative
entire solution $H$ of $H_t=\Delta H+f'(0)H$ admits
\begin{equation}\label{eq:widder-input}
 H(t,x)=\int_{\R^N}
 e^{\lambda_*x\cdot\xi+f'(0)(1+|\xi|^2)t}\,d\mu(\xi),
 \qquad \mu(\R^N)=H(0,0),
\end{equation}
where $\mu$ is a unique finite nonnegative measure. This follows by
applying the nonnegative ancient heat
representation to $e^{-f'(0)t}H$ on each past half-line and identifying
the resulting measures by uniqueness of the Laplace transform;
see \cite[Theorem~1.1]{LinZhang2019}.

We organize the paper as follows.
The differential estimates in
Subsection~\ref{subsec:entire-differential-estimates} allow us to define nonlinear
dilation $D_a$, for $a>0$, and superposition $\oplus$ by Cauchy evolution from times tending to
$-\infty$. Dilation and superposition provide nonlinear analogues of positive
scalar multiplication and addition, respectively, while preserving
the class of entire solutions.
In Subsections~\ref{subsec:nonlinear-constructions}
and~\ref{subsec:asymptotic-uniqueness-laws}, we establish
the algebraic properties of these operations, analogous to those
of scalar multiplication and addition in a linear space.
We also construct subtraction $\ominus$: for entire solutions
$0\leq U\leq V<1$, the solution $V\ominus U$ is uniquely determined by $U\oplus(V\ominus U)=V$.
In Subsection~\ref{subsec:ordered-solution-space}, these operations
identify the solution class with the positive cone of an ordered
vector space. For each nonzero solution $U$, we equip
the subspace $E_U$ consisting of elements bounded in order by scalar
multiples of $U$ with a complete order-unit norm and represent
it as a space of continuous functions. We relate norm convergence of
positive elements to local parabolic convergence of entire solutions
and describe space--time translations as bounded linear operators
on $E_U$.

Building on these estimates and nonlinear operations,
Section~\ref{sec:general-finite-lift} associates a solution
of the linearized equation $u_t=\Delta u+f'(0)u$ with each entire
solution $U$ of \eqref{eq:kpp} satisfying $0<U<1$ and
\begin{equation}\label{eq:intro-linear-finiteness}
 \sup_{T>0}\bigl(L_T[U(-T,\cdot)]\bigr)(0)<\infty.
\end{equation}
We take $U(-T,\cdot)$ as the initial datum at time $-T$ and evolve
it according to the linearized equation up to time $0$.
The condition states that the resulting value at the spatial origin
remains bounded as $T\to\infty$. We prove that this bound yields
a finite entire solution $H$ of $H_t=\Delta H+f'(0)H$, given by
\[
 H(t,\cdot)=\lim_{T\to\infty}L_{t+T}[U(-T,\cdot)].
\]
Moreover, the original nonlinear solution $U$ can be recovered from
$H$ by
\[
 U(t,\cdot)=\lim_{T\to\infty}
 S_{t+T}\bigl[\min\{H(-T,\cdot),1\}\bigr].
\]
Thus $H$ determines $U$. Representing $H$ by a finite measure
and approximating that measure by finite atomic measures then
gives approximation of $U$ by finite HN solutions.

Small-amplitude analysis in the same section shows
that the strict logarithmic gradient bound
\begin{equation}\label{eq:intro-strict-gradient-bound}
 |\nabla\ln U|\leq\lambda_*r
 \qquad\text{on }\R\times\R^N,\quad\text{for some }0\leq r<1,
\end{equation}
is sufficient for \eqref{eq:intro-linear-finiteness}.
Under this bound, we prove
\[
 \frac{D_aU}{a}\rightarrow H
 \quad\text{as }a\to0,
\]
where $H$ is finite and equals the linear limit above.

Section~\ref{sec:degenerate-spectral-decomposition} then determines
the relation between the time and space translation parameters.
In the representation constructed in Section~\ref{sec:pde-dilations},
translation by $(s,h)$ acts as multiplication by
$e^{s\lambda+\lambda_*h\cdot\xi}$. We prove that
\[
 \lambda=f'(0)(1+|\xi|^2).
\]
Using this relation, we decompose an arbitrary entire solution into
a critical part corresponding to $|\xi|=1$ and an increasing limit
of solutions satisfying \eqref{eq:intro-strict-gradient-bound}.

To identify the critical part,
Section~\ref{sec:critical-measure-solutions}
constructs a radially symmetric solution and uses it to build
critical solutions from finite measures on the sphere.
This section proves Theorem~\ref{thm:borel-main}, including
the injectivity and continuity of the measure parametrization.
We then show that the critical part obtained in
Section~\ref{sec:degenerate-spectral-decomposition} belongs to
this constructed family. For a nonzero critical solution $U_*$,
the small-amplitude limit of $D_aU_*/[a\ln(1/a)]$ yields a
nonzero measure on the critical sphere. We prove that $U_*$ is
exactly the solution constructed from the corresponding measure.
Section~\ref{sec:final-native-closure} completes the proof of
Theorem~\ref{thm:main-closure}. We first fix a solution with a strict
logarithmic gradient bound in the decomposition and approximate
both its representing measure and the critical measure by finite
atomic measures. We prove convergence of the corresponding nonlinear
superpositions and identify them as finite HN solutions.
A diagonal sequence then converges to the original entire solution
in $C^{1,2}_{\loc}(\R\times\R^N)$.

\section{The ordered structure of the space of entire solutions}
\label{sec:general-kpp-differential}
\label{sec:ancient-estimates}
\label{sec:pde-dilations}

For a homogeneous linear equation, scalar multiples, sums, and differences
of solutions are again solutions. These pointwise operations do not, in
general, preserve solutions of the nonlinear equation \eqref{eq:kpp}.
We first establish differential estimates for entire solutions and then
use Cauchy evolution to construct nonlinear analogues of these operations
on the class of entire solutions satisfying $0\leq U<1$.
Dilation $D_aU$, for $a>0$, plays the role of positive scalar
multiplication, while superposition $U\oplus V$ plays the role of addition.
For $U\leq V$, subtraction produces the unique entire solution $W$ such
that $U\oplus W=V$.
We prove the corresponding algebraic laws and identify this solution
class with the positive cone of an ordered vector space $E$.
For each nonzero solution $U$, the subspace $E_U$ consisting of
elements bounded in order by scalar multiples of $U$ carries a
complete order-unit norm and admits a representation as a space
of continuous functions. We show that norm convergence of positive
elements implies convergence of the corresponding entire solutions in
$C^{1,2}_{\loc}(\R\times\R^N)$.
Finally, we study space--time translations as bounded linear order
isomorphisms of $E_U$ and describe their multiplication-operator
representation. The relation between their time and space parameters,
and the resulting critical--subcritical decomposition, are established
in Section~\ref{sec:degenerate-spectral-decomposition}.

\subsection{Differential estimates}
\label{subsec:entire-differential-estimates}

Let $U$ be a classical entire solution of \eqref{eq:kpp} with $0<U<1$ on $\R\times\R^N$. In this subsection, we establish Li--Yau-type differential estimates for
$\ln U$, yielding logarithmic gradient bounds and strict monotonicity in time.
For $0<u\leq1$, concavity of $f$ gives
\begin{equation}\label{eq:gkpp-reaction-remainders}
 \begin{aligned}
 f'(0)u-f(u)&=\int_0^u(u-v)(-f''(v))\,dv\geq0,\\
 f(u)-uf'(u)&=\int_0^u v(-f''(v))\,dv\geq0.
 \end{aligned}
\end{equation}
In particular, $0\leq f(u)/u\leq f'(0)$ and $f'(u)\leq f(u)/u$.
Since $-f''$ is continuous and nonnegative,
\begin{equation}\label{eq:gkpp-k-zero}
 f'(u)=\frac{f(u)}u\quad\Longleftrightarrow\quad
 f(v)=f'(0)v\ \hbox{for all }0\leq v\leq u.
\end{equation}
Consequently $f(u)/u<f'(0)$ implies $f'(u)<f(u)/u$.

The following estimates control the spatial and temporal variation of
$\ln U$ and imply time monotonicity and the limits at both ends of time.

\begin{theorem}
\label{thm:gkpp-ancient}
Every classical entire solution $U$ of \eqref{eq:kpp} with
$0<U<1$ on $\R\times\R^N$ satisfies
\begin{equation}\label{eq:gkpp-ancient-bounds}
 |\nabla\ln U|^2\leq\partial_t\ln U,\qquad
 |\nabla\ln U|\leq\lambda_*,\qquad
 0<\partial_t\ln U\leq2f'(0).
\end{equation}
If $|\nabla\ln U|\leq\lambda_*\ell$ for some $0\leq\ell\leq1$,
then $\partial_t\ln U\leq f'(0)(1+\ell^2)$ and
\begin{equation}\label{eq:gkpp-gap-time}
 0<U_t\leq(1+\ell^2)f(U).
\end{equation}
In particular,
\begin{equation}\label{eq:gkpp-phase-deficit}
 0<U_t\leq2f(U).
\end{equation}
Moreover, $U(t,\cdot)\to0$ as $t\to-\infty$ and
$U(t,\cdot)\to1$ as $t\to+\infty$, locally uniformly in $\R^N$.
\end{theorem}

\begin{proof}
For convenience, we write
\[
 h(t,x)=\ln U,\qquad p(t,x)=|\nabla h|^2,\qquad r(t,x)=h_t,\qquad
 \mathcal L=\partial_t-\Delta-2\nabla h\cdot\nabla.
\]
For any $t_1<t_2$, $R>0$, and $x_0\in\R^N$, continuity and positivity of $U$ imply a positive lower bound on $[t_1,t_2]\times\overline{B_R(x_0)}$.
A direct calculation gives, on $\R\times\R^N$,
\begin{equation}\label{eq:gkpp-log-equation}
 h_t=\Delta h+p+\frac{f(U)}{U},
\end{equation}
 and
 \begin{align}
 \mathcal Lp=2\left(f'(U)-\frac{f(U)}U\right)p-2|\nabla^2h|^2, \qquad
 \mathcal Lr=\left(f'(U)-\frac{f(U)}U\right)r.
 \label{eq:gkpp-p-r}
\end{align}
Here, $|\nabla^2h|$ denotes the Euclidean norm of the Hessian matrix.

\emph{Step 1: the inequality $p\leq r$.}
Fix $\alpha>1$ and set $F(t,x)=p(t,x)-\alpha r(t,x)$.
Using \eqref{eq:gkpp-p-r} gives
\begin{align*}
 \mathcal LF
 &=\left(f'(U)-\frac{f(U)}U\right)(2p-\alpha r)-2|\nabla^2h|^2\\
 &=\left(f'(U)-\frac{f(U)}U\right)(p+F)-2|\nabla^2h|^2.
\end{align*}
Also, \eqref{eq:gkpp-log-equation} yields
\[
 \Delta h=r-p-\frac{f(U)}U
 =-\frac{F+(\alpha-1)p+\alpha f(U)/U}\alpha.
\]
On $\{F>0\}$, the Cauchy-Schwarz inequality gives
\[
 |\nabla^2h|^2\geq\frac{(\Delta h)^2}{N}
 \geq\frac{(F+(\alpha-1)p)^2}{N\alpha^2},
\]
and we obtain
\begin{equation}\label{eq:gkpp-F-coercive}
 \mathcal LF\leq-\frac{2}{N\alpha^2}(F+(\alpha-1)p)^2.
\end{equation}

We construct a supersolution for $F$.  Since neither $F$ nor the
drift $2\nabla h$ is known to be bounded on the whole space, we first
regularize the initial data away from zero.  Fix $T_0<t_1$. For
$0<\varepsilon<1$, let $U^\varepsilon(t,x)$ solve \eqref{eq:kpp} on $[T_0,t_1]\times\mathbb R^N$ with
\begin{equation}\label{eq:ancient-regularized-data}
 U^\varepsilon(T_0,x)
 =\varepsilon+(1-\varepsilon)U(T_0,x).
\end{equation}
The comparison principle gives
$\varepsilon\leq U^\varepsilon\leq1$ and $U\leq U^\varepsilon$.
Write
\[
 h^\varepsilon=\ln U^\varepsilon,\qquad
 p^\varepsilon=|\nabla h^\varepsilon|^2,\qquad
 r^\varepsilon=\partial_t h^\varepsilon,\qquad
 F^\varepsilon=p^\varepsilon-\alpha r^\varepsilon,\]
 and
 \[
 \mathcal L^\varepsilon=\partial_t-\Delta
                  -2\nabla h^\varepsilon\cdot\nabla.
\]
Interior estimates for $U$ and Cauchy estimates for $U^\varepsilon$
give, on $Q=[T_0,t_1]\times\mathbb R^N$,
\begin{equation}\label{eq:ancient-regularized-derivatives}
 \|\nabla U^\varepsilon\|_{L^\infty(Q)}
 +\|U_t^\varepsilon\|_{L^\infty(Q)}\leq K,
\end{equation}
with $K$ independent of $\varepsilon$.  Since
$U^\varepsilon\geq\varepsilon$, it follows that
\[
 \|\nabla h^\varepsilon\|_{L^\infty(Q)}\leq\frac K\varepsilon,
 \qquad
 \|F^\varepsilon\|_{L^\infty(Q)}
 \leq\frac{K^2}{\varepsilon^2}+\frac{\alpha K}{\varepsilon}.
\]
Thus the drift $2\nabla h^\varepsilon$ and $F^\varepsilon$ are
bounded for each fixed $\varepsilon$.

The calculation for \eqref{eq:gkpp-F-coercive} applies
to $U^\varepsilon$ and it gives that on $\{F^\varepsilon>0\}$,
\[
 \mathcal L^\varepsilon F^\varepsilon
 \leq-\frac{2}{N\alpha^2}(F^\varepsilon+(\alpha-1)p^\varepsilon)^2
 \leq-\frac{2}{N\alpha^2}(F^\varepsilon)^2.
\]
Choose
\[
 M_\varepsilon
 =\max\left\{1,\sup_{x\in\mathbb R^N}F^\varepsilon(T_0,x)\right\},
 \qquad
 \overline F^\varepsilon(t)
 =\frac{N\alpha^2M_\varepsilon}{N\alpha^2+2M_\varepsilon(t-T_0)}.
\]
This positive function $\overline F^\varepsilon(t)$ satisfies
\[
 (\overline F^\varepsilon)'=-\frac{2}{N\alpha^2}(\overline F^\varepsilon)^2,
 \qquad
 F^\varepsilon(T_0,x)\leq\overline F^\varepsilon(T_0).
\]
The comparison principle
gives that for $T_0<t\leq t_1$,
\begin{equation}\label{F-e}
 F^\varepsilon(t,x)
 \leq\overline F^\varepsilon(t)
 =\frac{N\alpha^2}{N\alpha^2M_\varepsilon^{-1}+2(t-T_0)}
 \leq\frac{N\alpha^2}{2(t-T_0)}.
\end{equation}

The difference $D^\varepsilon=U^\varepsilon-U$
satisfies $D^\varepsilon\geq0$ by comparison. Concavity gives
\[
 (\partial_t-\Delta)D^\varepsilon
 =\left(\int_0^1f'\bigl(U+\theta D^\varepsilon\bigr)\,d\theta\right)
 D^\varepsilon
 \leq f'(0)D^\varepsilon.
\]
Together with $0\leq D^\varepsilon(T_0,x)\leq\varepsilon$, this yields
$0\leq D^\varepsilon(t,x)\leq\varepsilon e^{f'(0)(t-T_0)}$.
Thus, for fixed $T_0<t_1$, interior parabolic regularity gives
\begin{equation}\label{eq:ancient-regularized-convergence}
 (U^\varepsilon,\nabla U^\varepsilon,U_t^\varepsilon)(t_1,\cdot)
 \rightarrow (U,\nabla U,U_t)(t_1,\cdot)
 \quad\text{in }C_{\loc}(\mathbb R^N) \hbox{ as }\varepsilon \rightarrow 0.
\end{equation}
Since $U>0$, $F^\varepsilon(t_1,x)\rightarrow F(t_1,x)$ as $\varepsilon\rightarrow0$.
Letting $\varepsilon\rightarrow0$ in \eqref{F-e} gives
\[
 p(t_1,x)-\alpha r(t_1,x)
 \leq\frac{N\alpha^2}{2(t_1-T_0)}.
\]
Let $T_0\to-\infty$ and then let
$\alpha\rightarrow1$. This proves $p\leq r$ everywhere.
In particular, since $p\geq0$, $U_t=Ur\geq0$.
Since $U_t$ solves
$(\partial_t-\Delta)U_t=f'(U)U_t$, either
$U_t\equiv0$ or $U_t>0$ by the strong maximum principle and Cauchy uniqueness.
If $U_t\equiv0$, $p\leq r$ would give $\nabla U=0$, which implies that $U$ is a constant. The
equation \eqref{eq:kpp} would then give $f(U)=0$, contrary to $0<U<1$.
Thus $U_t>0$.

\emph{Step 2: the bound $p\leq f'(0)$.}
Fix $T_0<t_1$ and use the regularization
\eqref{eq:ancient-regularized-data} and the notation of Step 1 on
$[T_0,t_1]\times\mathbb R^N$. Define
\[
 Q^\varepsilon
 =U^\varepsilon\bigl((t-T_0)p^\varepsilon+h^\varepsilon
                         -f'(0)(t-T_0)\bigr).
\]
Then $Q^\varepsilon(T_0,\cdot)=U^\varepsilon h^\varepsilon\leq0$.
Using $f(U^\varepsilon)/U^\varepsilon\leq f'(0)$,
$\mathcal L^\varepsilon h^\varepsilon=f(U^\varepsilon)/U^\varepsilon-p^\varepsilon$,
and $\mathcal L^\varepsilon p^\varepsilon\leq0$, it follows that
\begin{align*}
 (\partial_t-\Delta)Q^\varepsilon
 &=\frac{f(U^\varepsilon)}{U^\varepsilon}Q^\varepsilon
   +U^\varepsilon\left((t-T_0)\mathcal L^\varepsilon p^\varepsilon
                        +\frac{f(U^\varepsilon)}{U^\varepsilon}-f'(0)\right)\\
 &\leq\frac{f(U^\varepsilon)}{U^\varepsilon}Q^\varepsilon.
\end{align*}
By \eqref{eq:ancient-regularized-derivatives},
$\varepsilon\leq U^\varepsilon\leq1$, and $h^\varepsilon\leq0$, one has
\[
 Q^\varepsilon
 \leq\frac{(t_1-T_0)K^2}{\varepsilon}
 \quad\hbox{on }[T_0,t_1]\times\mathbb R^N.
\]
The maximum principle therefore gives $Q^\varepsilon\leq0$,
and hence
\[
 (t-T_0)p^\varepsilon\leq-h^\varepsilon+f'(0)(t-T_0).
\]
At $t=t_1$, let $\varepsilon\rightarrow0$. Using
\eqref{eq:ancient-regularized-convergence} and $U(t_1,x)>0$, we obtain
\[
 p(t_1,x)\leq f'(0)-\frac{h(t_1,x)}{t_1-T_0}.
\]
Letting $T_0\to-\infty$ proves
$p\leq f'(0)$ and hence $|\nabla h|\leq\lambda_*$.

\emph{Step 3: the bound $r\leq f'(0)(1+\ell^2)$.}
We first show that $\sup r$ is finite. Fix $(t_*,x_*)$ and set
\[
 V(t,x)=\frac{U(t,x)}{U(t_*,x_*)}.
\]
On $(t_*-4,t_*]\times B_2(x_*)$, the inequality $U_t>0$
from Step 1 and the estimate $|\nabla\ln U|\leq\lambda_*$ from Step 2 give
\[
 U(t,x)\leq U(t_*,x)
 \leq e^{\lambda_*|x-x_*|}U(t_*,x_*)
 \leq e^{2\lambda_*}U(t_*,x_*).
\]
Thus, $0<V\leq e^{2\lambda_*}$ on $(t_*-4,t_*]\times B_2(x_*)$ and $V$ satisfies
$V_t-\Delta V=(f(U)/U)V$. Since $u\mapsto f(u)/u$ has a $C^1$ extension
to $[0,1]$ and $U$ is an entire solution with $0<U<1$, interior
parabolic estimates applied to the coefficient $f(U)/U$ and then to $V$ give
\[
 |V_t(t_*,x_*)|+|\nabla^2V(t_*,x_*)|
 \leq Ce^{2\lambda_*},
\]
where $C$ is independent of $(t_*,x_*)$ and $U(t_*,x_*)$. This means
\[
 \frac{|U_t(t_*,x_*)|+|\nabla^2U(t_*,x_*)|}{U(t_*,x_*)}\leq Ce^{2\lambda_*}.
\]
Finally,
$\nabla^2h=\nabla^2U/U-\nabla h\otimes\nabla h$ and
$|\nabla h|^2\leq f'(0)$ give
\begin{equation}\label{eq:gkpp-H-bounded}
 |\nabla^2h|\leq Ce^{2\lambda_*}+f'(0)<\infty,
 \qquad 0<r=h_t\leq Ce^{2\lambda_*}<\infty.
\end{equation}

Assume $|\nabla h|\leq\lambda_*\ell$, where $0\leq\ell\leq1$,
and $M:=\sup_{\mathbb R\times\mathbb R^N}r>f'(0)(1+\ell^2)$.
From \eqref{eq:gkpp-H-bounded}, one has $M\leq Ce^{2\lambda_*}<\infty$.
Take a sequence $(t_n,x_n)$ such that
$r(t_n,x_n)\to M$ and define
\[
 w_n(t,x)=\frac{U(t+t_n,x+x_n)}{U(t_n,x_n)}.
\]
Then, $w_n(0,0)=1$ and it follows from $|\nabla h|\leq\lambda_*\ell$ and \eqref{eq:gkpp-H-bounded} that
\begin{align*}
 |\ln w_n(t,x)|
 &=|h(t_n+t,x_n+x)-h(t_n,x_n)|\\
 &\leq\lambda_*\ell|x|+Ce^{2\lambda_*}|t|.
\end{align*}
It implies
\begin{equation}\label{bound-wn}
 e^{-\lambda_*\ell|x|-R_0|t|}\leq w_n(t,x)
 \leq e^{\lambda_*\ell|x|+R_0|t|}.
\end{equation}
where $R_0:=Ce^{2\lambda_*}$.
Set $a_n=U(t_n,x_n)$. The equation for $w_n(t,x)$ reads as
\[
 (w_n)_t-\Delta w_n=\frac{f(a_nw_n)}{a_n}.
\]
By the bound \eqref{bound-wn}, $0<a_n<1$ and parabolic estimates,
the sequence $w_n$, up to an extraction of a subsequence, converges
locally to a positive entire function $w$, with $w(0,0)=1$, satisfying
\[
 w_t=\Delta w+F_a(w),\qquad
 F_a(v)=\begin{cases}f(av)/a,&a>0,\\f'(0)v,&a=0.\end{cases}
\]
Here $a$ is a limit of $a_n$, and $a_nw_n<1$ implies $aw\leq1$.
For $a=0$ this
passage follows from $f(a_nv)/a_n\to f'(0)v$, uniformly on bounded
$v$-intervals together with the first derivative. Write
$\widetilde h=\ln w$, $\widetilde r=\widetilde h_t$, and
$\widetilde{\mathcal L}=\partial_t-\Delta-2\nabla\widetilde h\cdot\nabla$.
Then,
\[
 |\nabla\widetilde h|\leq\lambda_*\ell,\qquad
 \widetilde r\leq M,\qquad \widetilde r(0,0)=M,
\]
and
\begin{align*}
 \widetilde h_t&=\Delta\widetilde h+|\nabla\widetilde h|^2+\frac{f(aw)}{aw},\\
 \widetilde{\mathcal L}\widetilde r
 &=\left(f'(aw)-\frac{f(aw)}{aw}\right)\widetilde r.
\end{align*}
Clearly, $M$ is a supersolution for $\widetilde r$. The strong maximum
principle gives $\widetilde r(t,x)=M$ on $(-\infty,0]\times\mathbb R^N$.
Then, for a fixed negative time $t$, one has
\[
 \Delta\widetilde h
 =M-|\nabla\widetilde h|^2-\frac{f(aw)}{aw}
 \geq M-f'(0)(1+\ell^2)>0.
\]
This contradicts the gradient bound by the divergence
theorem. More precisely, for every $L>0$, with $\nu$ the outward unit normal,
\[
 \bigl(M-f'(0)(1+\ell^2)\bigr)|B_L|
 \leq\int_{B_L}\Delta\widetilde h\,dx
 =\int_{\partial B_L}\nabla\widetilde h\cdot\nu\,dS
 \leq\lambda_*\ell|\partial B_L|.
\]
Since $|\partial B_L|/|B_L|=N/L$, letting $L\to\infty$ contradicts
$M>f'(0)(1+\ell^2)$.
Hence $M\leq f'(0)(1+\ell^2)$.
Taking $\ell=1$, as allowed by Step 2, gives $0<r\leq2f'(0)$.

\emph{Step 4: the bound $U_t\leq(1+\ell^2)f(U)$.}
Assume $|\nabla h|\leq\lambda_*\ell$, where $0\leq\ell\leq1$, and put
$\lambda=1+\ell^2$. Step 3 gives $0<r\leq f'(0)\lambda$.
Differentiating $f(U)/U$ gives the exact identity
\begin{equation}\label{eq:gkpp-Lb}
 \mathcal L\left(\frac{f(U)}U\right)
 =\left(f'(U)-\frac{f(U)}U\right)\frac{f(U)}U-Uf''(U)p.
\end{equation}
For $d=\lambda f(U)/U-r$, equations \eqref{eq:gkpp-p-r} and
\eqref{eq:gkpp-Lb} imply
\begin{equation}\label{eq:gkpp-phase-identity}
 \left(\mathcal L-f'(U)+\frac{f(U)}U\right)d
 =-\lambda Uf''(U)p\geq0.
\end{equation}
Suppose that $m=\inf d<0$.  Since $r\leq f'(0)\lambda$ and
$f(u)/u\to f'(0)$ as $u\searrow0$, a minimizing sequence has
$U\geq a_0>0$.  Translating and passing to a subsequence gives a
positive limiting solution for which $d$ attains the global minimum
$m$ at the origin.  If its value there were one, the strong maximum
principle would make the limit the constant one and $d=0$ there.
Thus this value lies in $(0,1)$.  At this minimum,
$\lambda f(U)/U-r=m<0$ and $r\leq f'(0)\lambda$ imply $f(U)/U<f'(0)$; hence
$f'(U)-f(U)/U<0$ by \eqref{eq:gkpp-k-zero}. Formula
\eqref{eq:gkpp-phase-identity} then gives
\[
 \mathcal Ld=\left(f'(U)-\frac{f(U)}U\right)m-\lambda Uf''(U)p>0.
\]
At a space--time minimum the left side is nonpositive, a
contradiction. Consequently $r\leq\lambda f(U)/U$, proving
\eqref{eq:gkpp-gap-time} and \eqref{eq:gkpp-phase-deficit}.

Finally, since $U_t>0$ and $0<U<1$, the limits
\[
 U_\pm(x):=\lim_{t\to\pm\infty}U(t,x)
\]
exist and satisfy
\[
 0\leq U_-(x)\leq U(0,x)\leq U_+(x)\leq1.
\]
For any sequence $t_n\to\pm\infty$, interior parabolic
estimates give, after passing to a subsequence,
\[
 U(t+t_n,x)\to U_\pm(x)
 \quad\text{in }C^{1,2}_{\loc}(\R\times\R^N).
\]
Passing to the limit in the equation and in
$|\nabla U|^2\leq UU_t$, we obtain
\[
 \Delta U_\pm+f(U_\pm)=0,
 \qquad |\nabla U_\pm|^2\leq0.
\]
Thus $U_\pm$ are constants, each equal to zero or one.
Since $U_-(x)\leq U(0,x)<1$ and $U_+(x)\geq U(0,x)>0$,
one has $U_-\equiv0$ and $U_+\equiv1$.
The same compactness argument yields $U(t,\cdot)\to0$ as $t\to-\infty$
and $U(t,\cdot)\to1$ as $t\to+\infty$,
locally uniformly in $\R^N$.
\end{proof}

\subsection{Nonlinear constructions and amplitude estimates}
\label{subsec:nonlinear-constructions}
\label{subsec:critical-wave-amplitude}

We construct dilation and superposition from the Cauchy evolution
and establish the amplitude estimates needed for their algebraic
laws and small-amplitude limits.

All Cauchy problems below use the concave
extension fixed after \eqref{eq:general-kpp-assumptions}.
Given bounded nonnegative initial data $h_T$ prescribed at time $-T$,
we call $S_{t+T}[h_T]$, for $t>-T$, a \emph{pullback approximation}.
A limit of these solutions as $T\to\infty$, with $t$ fixed,
is called a \emph{pullback limit}.
Since $f'(1)<0$, comparison with the scalar equation bounds
solutions with bounded nonnegative initial data, and its solutions
started above one decrease to one. Thus any pullback limit from
uniformly bounded initial data takes values in $[0,1]$.

Let $\mathcal E$ denote all entire solutions in $[0,1]$ and
$\mathcal C=\{0\}\cup\{U\in\mathcal E:0<U<1\}$.
The preceding theorem gives, for every entire solution $U$ of
\eqref{eq:kpp} with $0<U<1$,
\begin{equation}\label{eq:cone-ancient-input}
 |\nabla\ln U|\leq\lambda_*,\qquad
 0<U_t\leq2f(U)\leq2f'(0)U,\qquad
 U(t,\cdot)\rightarrow0\quad(t\to-\infty).
\end{equation}
All unqualified inequalities between solutions are pointwise.

We first record the comparison, concavity, and stability properties
of the Cauchy operators that will control the nonlinear constructions.

\begin{lemma}
\label{lem:pde-semigroup-inequalities}
For bounded nonnegative initial data $h(x),\,k(x)$ and $t\geq0$,
the following properties hold:
\[
 S_t[h+k]\leq S_t[h]+S_t[k],
 \]
 and
 \begin{equation}\label{sub-concavity}
 S_t[\theta h+(1-\theta)k]\geq\theta S_t[h]+(1-\theta)S_t[k],\quad \hbox{for }0\leq\theta\leq1.
\end{equation}
In particular,
\begin{equation}\label{scaling-inequality}
 S_t[ah]
 \geq aS_t[h],\quad \hbox{for }0<a\leq1,\qquad
  S_t[ah] \leq aS_t[h],\quad \hbox{for }a\geq1.
  \end{equation}
Moreover
\begin{equation}\label{eq:cone-stability}
 |S_t[h]-S_t[k]|\leq S_t[|h-k|]
 \leq e^{f'(0)t}e^{t\Delta}|h-k|.
\end{equation}
\end{lemma}
\begin{proof}
For $u,v\geq0$, concavity gives
\[
 f(u)+f(v)-f(u+v)
 =\int_0^u\bigl(f'(s)-f'(s+v)\bigr)\,ds\geq0,
\]
so the sum of two solutions is a supersolution. On the other hand, concavity also gives
\[
 \theta f(u)+(1-\theta)f(v)
 -f(\theta u+(1-\theta)v)\leq0.
\]
The first two inequalities follow from the comparison principle. Since $S_t[0]=0$, the scaling inequalities follow by applying \eqref{sub-concavity} to $h$ and $0$ with
$\theta=a$ for $0<a\leq1$ and to $ah$ and $0$ with $\theta=1/a$ for $a\geq1$.
For $u=S_t[h]$ and $v=S_t[k]$, concavity gives
\[
 (u-v)_t-\Delta(u-v)=f(u)-f(v)\leq f(u-v)
 \qquad\text{where }u>v.
\]
Thus $\max\{u-v,0\}$ is a subsolution with initial data
$\max\{h-k,0\}$. Comparison in both directions gives
$|S_t[h]-S_t[k]|\leq S_t[|h-k|]$. The remaining inequality in
\eqref{eq:cone-stability} follows from $f(u)\leq f'(0)u$ for $u\geq0$.
It also proves continuity for uniformly bounded, locally
uniformly convergent initial data, by dominated convergence in the
Gaussian integral at each positive elapsed time.
\end{proof}

For $U,V\in\mathcal C$ and an amplitude parameter $a>0$, define the nonlinear dilation
$D_aU$ and superposition $U\oplus V$ by
\begin{equation}\label{eq:basic-pullbacks}
 D_aU(t)=\lim_{T\to\infty}S_{t+T}[aU(-T)],\qquad
 U\oplus V=\lim_{T\to\infty}S_{t+T}[U(-T)+V(-T)].
\end{equation}
Here and below $T>-t$. The first sequence increases in $T$ for
$a\leq1$ and decreases for $a\geq1$; the second decreases.
Indeed compare the evolution from $-T_2$ to $-T_1$, for $T_2>T_1$,
using the preceding lemma, then evolve to $t$.
Scalar comparison and interior regularity give entire limits with
convergence in $C^{1,2}_{\loc}(\R\times\R^N)$, and
\begin{equation}\label{eq:cone-basic-bounds}
 \begin{gathered}
 aU\leq D_aU\leq U\quad\text{for }0<a\leq1,\qquad
 U\leq D_aU\leq aU\quad\text{for }a\geq1,\\
 \max(U,V)\leq U\oplus V\leq\min(U+V,1).
 \end{gathered}
\end{equation}
The past limits and the strong maximum principle put all these
solutions in $\mathcal C$. Comparison shows that $U\oplus V$ is
the greatest entire solution below the ordinary sum $U+V$.
The operations preserve order and commute with translations and
rotations. The map $a\mapsto D_aU(t,x)$ is nondecreasing and concave.

The spatial bound proved above must also hold after changing the
amplitude $a$. The following lemma shows that dilation preserves any
logarithmic gradient bound.

\begin{lemma}\label{lem:gap-preservation}
Let $U$ be a classical entire solution of \eqref{eq:kpp}
with $0<U<1$. If $|\nabla\ln U|\leq\lambda_*r$ for some $r\geq0$, then the same bound
holds for $D_aU$ and the corresponding Cauchy solutions
$S_{t+T}[aU(-T,\cdot)]$, for every $a>0$, $T>0$, and $t>-T$.
\end{lemma}
\begin{proof}
Integrating the assumed bound along a line segment at time $-T$ gives
\[
 U(-T,x+y)\leq e^{\lambda_*r|y|}U(-T,x)
 \qquad\text{for }x,y\in\R^N.
\]
For $a>0$ and $t>-T$, translation invariance, comparison, and
\eqref{scaling-inequality} therefore yield
\[
 \begin{aligned}
 \bigl(S_{t+T}[aU(-T,\cdot)]\bigr)(x+y)
 &=\bigl(S_{t+T}[aU(-T,\cdot+y)]\bigr)(x)\\
 &\leq\bigl(S_{t+T}[e^{\lambda_*r|y|}aU(-T,\cdot)]\bigr)(x)\\
 &\leq e^{\lambda_*r|y|}
       \bigl(S_{t+T}[aU(-T,\cdot)]\bigr)(x).
 \end{aligned}
\]
Taking logarithms and interchanging $x$ and $x+y$ gives the claimed
gradient bound for these Cauchy solutions. Passing to the limit
$T\to\infty$ in the spatial comparison gives the same bound for $D_aU$.
\end{proof}

We next determine the endpoint limits of the amplitude parameter $a$.
These limits will allow us to normalize a nonzero solution by its
value at a fixed space--time point.

\begin{lemma}\label{lem:cone-small-amplitudes}
For every entire solution $U$ of \eqref{eq:kpp} with $0<U<1$,
\[
 \lim_{a\searrow0}D_aU=0,\qquad\lim_{a\to\infty}D_aU=1
 \quad\text{in }C^{1,2}_{\loc}.
\]
\end{lemma}
\begin{proof}
Integration of $\partial_t\ln U\leq2f'(0)$ from
$t+\frac{\ln a}{2f'(0)}$ to $t$, for $a\leq1$, gives
$aU(t,x)\leq U(t+\frac{\ln a}{2f'(0)},x)$.
Cauchy comparison therefore gives
$D_aU\leq U(t+\frac{\ln a}{2f'(0)},x)$.
For $a\geq1$ integration on the reversed interval gives
$D_aU\geq U(t+\frac{\ln a}{2f'(0)},x)$. Apply the past and future limits
of Theorem~\ref{thm:gkpp-ancient}.
\end{proof}

For the subsequent amplitude comparisons, define the critical-profile
slope as a function of the solution value:
\[
 q(u)=\phi_*'(\phi_*^{-1}(u))\quad\text{for }0<u<1.
\]
Then $qq'=c_*q-f$, $0<q(u)<\lambda_*u$,
$q(u)/u\to\lambda_*$ at zero,
and $\lim_{u\nearrow1}q(u)=0$.
These properties follow from \eqref{eq:intro-wave-slope}
and the traveling-wave equation \eqref{eq:critical-front}.

The spatial gradient bound also yields a second-order differential
inequality for the amplitude map $a\mapsto D_aU(t,x)$. The following inequality is the starting point for
the small-amplitude limits studied later.

\begin{lemma}\label{lem:finite-amplitude}
If $|\nabla\ln U|\leq\lambda_*r$ for some $0\leq r\leq1$ and
$G(a)=D_aU(t,x)$, then for every $\gamma\in[r^2,1]$,
\begin{equation}\label{eq:trace-F-inequality}
 \gamma a^2G''\leq aG'-\frac{f(G)}{f'(0)}
\end{equation}
in the distributional sense on $(0,\infty)$.
\end{lemma}

\begin{proof}
The extended $f$ is $C^{1,1}(\R)$. Let $\rho_\varepsilon$ be a standard
nonnegative mollifier of unit integral and set
\[
 f_\varepsilon=\rho_\varepsilon*f-(\rho_\varepsilon*f)(0).
\]
Then $f_\varepsilon\to f$ in $C^1_{\loc}$,
$f_\varepsilon(0)=0$, $f_\varepsilon'\leq f'(0)$, and
$f_\varepsilon''\leq0$. For finite $T$, let $u$ solve
$u_t=\Delta u+f_\varepsilon(u)$ with $u(-T)=aU(-T)$, and set
$v=a\partial_a u$,
$w=a^2\partial_a^2u$ and
$\mathscr L_u=\partial_t-\Delta-f_\varepsilon'(u)$.
Differentiating the Cauchy problem twice with respect to the
amplitude parameter $a$ gives
\[
 \mathscr L_uv=0,\qquad
 \mathscr L_uw=f_\varepsilon''(u)v^2,\qquad
 \mathscr L_uf_\varepsilon(u)=-f_\varepsilon''(u)|\nabla u|^2.
\]
Initially $v=u$, $w=0$.
Each directional derivative of $u$ satisfies the same linear equation
as $v$. The initial bound $|\nabla u|\leq\lambda_*rv$ thus propagates.
Consequently
\[
 \mathscr L_u\left\{\gamma w+\frac{f_\varepsilon(u)}{f'(0)}-v\right\}
   =f_\varepsilon''(u)\left(\gamma v^2-\frac{|\nabla u|^2}{f'(0)}\right)\leq0.
\]
Its initial value is $f_\varepsilon(aU)/f'(0)-aU\leq0$, so comparison proves
$\gamma a^2u_{aa}\leq au_a-f_\varepsilon(u)/f'(0)$.
For a nonnegative test function $\psi\in C_c^\infty((0,\infty))$,
\[
 \int_0^\infty\left(\gamma u(a^2\psi)''+u(a\psi)'
       +\frac{f_\varepsilon(u)}{f'(0)}\psi\right)\,da\leq0.
\]
For fixed $T$, the bound
$0\leq u(t)\leq a e^{f'(0)(t+T)}\|U(-T)\|_\infty$ and Cauchy stability
allow $\varepsilon\searrow0$, giving the same inequality for
$S_{t+T}[aU(-T)]$ with $f$ in place of $f_\varepsilon$.
These finite pullbacks are concave and locally bounded in $a$, so their
convergence to $G$ is locally uniform on $(0,\infty)$.
Letting $T\to\infty$ proves the distributional inequality.
\end{proof}

The following proposition compares nonlinear dilation with a phase
shift of the critical traveling front and gives its asymptotic
behavior.

\begin{proposition}\label{prop:inverse-wave}\label{prop:inverse-wave-bounds}
Every entire solution $U$ of \eqref{eq:kpp} with $0<U<1$ on
$\R\times\R^N$ satisfies, for every $(t,x)\in\R\times\R^N$,
\begin{equation}\label{eq:inverse-wave-upper}
 D_aU(t,x)\leq\phi_*(\lambda_*^{-1}\ln a+\phi_*^{-1}(U(t,x)))
 \qquad\text{for }0<a\leq1,
\end{equation}
and
\begin{equation}\label{eq:inverse-wave-lower}
 D_aU(t,x)\geq\phi_*(\lambda_*^{-1}\ln a+\phi_*^{-1}(U(t,x)))
 \qquad\text{for }a\geq1.
\end{equation}
Moreover,
\begin{equation}\label{eq:amplitude-low-value}
 D_aU(t,x)=aU(t,x)+o(U(t,x))\quad\text{as }U(t,x)\searrow0,
\end{equation}
uniformly with respect to the entire solution $0<U<1$ and
$(t,x)\in\R\times\R^N$, and locally uniformly in $a\in(0,\infty)$.
\end{proposition}

\begin{proof}
Fix $(t,x)$ and put $G(a)=D_aU(t,x)$, $u(s)=G(e^s)$, and
$w(s)=e^sG'_+(e^s)$, where $G'_+$ denotes the right derivative of $G$.
Since $G$ is concave, $u$ is locally absolutely continuous,
$w$ has locally bounded variation, and $u'=w$ almost everywhere.
The change of variables $a=e^s$
in \eqref{eq:trace-F-inequality}, with $\gamma=1$, gives
\[
 Dw\leq\bigl(2w-f(u)/f'(0)\bigr)\,ds
\]
in the sense of distributions.
The endpoints are $u(-\infty)=0$ and $u(+\infty)=1$.
Since $q(u)$ is locally absolutely continuous,
$d=w-q(u)/\lambda_*$ satisfies
\[
 Dd\leq\frac{f(u)}{\lambda_*q(u)}d\,ds.
\]
Since $d$ is locally of bounded variation, the integrating-factor
argument applies to this distributional inequality.
If $d(s_0)=-\delta<0$ for its right-continuous representative, then
\[
 d(s)\leq-\delta\exp\left(\int_{s_0}^s
 \frac{f(u(\sigma))}{\lambda_*q(u(\sigma))}\,d\sigma\right)
 \leq-\delta\qquad\text{for }s\geq s_0.
\]
Consequently $0\leq w(s)\leq q(u(s))/\lambda_*-\delta\rightarrow-\delta$
as $s\to+\infty$, a contradiction.
Thus $u'\geq q(u)/\lambda_*$ almost everywhere, and
\[
 \frac{d}{ds}\phi_*^{-1}(u(s))
 =\frac{u'(s)}{q(u(s))}\geq\frac1{\lambda_*}
 \qquad\text{for almost every }s.
\]
For $0<a\leq1$, integration over $[\ln a,0]$ yields
\[
 \phi_*^{-1}(U(t,x))-\phi_*^{-1}(D_aU(t,x))\geq-\lambda_*^{-1}\ln a.
\]
For $a\geq1$, integration over $[0,\ln a]$ yields
\[
 \phi_*^{-1}(D_aU(t,x))-\phi_*^{-1}(U(t,x))\geq\lambda_*^{-1}\ln a.
\]
Applying the increasing function $\phi_*$ proves
\eqref{eq:inverse-wave-upper} and \eqref{eq:inverse-wave-lower}.
Finally
\[
 \frac{\phi_*(z+\lambda_*^{-1}\ln a)}{a\phi_*(z)}\rightarrow1
 \quad(z\to-\infty)
\]
uniformly for $a$ in a fixed compact positive interval. Squeezing
between these estimates and the elementary bounds
\eqref{eq:cone-basic-bounds} proves~\eqref{eq:amplitude-low-value}.
\end{proof}

The preceding amplitude estimates imply that superposition agrees
with ordinary addition to first order when the solutions are small.

 \begin{lemma}\label{lem:addition-low-value}
Uniformly over all $U,V\in\mathcal C$,
\[
 U\oplus V=U+V+o(U+V)\quad\text{as }U+V\searrow0.
\]
\end{lemma}
\begin{proof}
Concavity of each Cauchy operator gives
\[
 \tfrac12(D_2U+D_2V)\leq U\oplus V\leq U+V.
\]
Applying~\eqref{eq:amplitude-low-value} gives the conclusion.
\end{proof}

\subsection{Asymptotic uniqueness and algebraic laws}
\label{subsec:asymptotic-uniqueness-laws}

We first prove that suitable comparisons in the remote past determine
the order of entire solutions. We use this result to establish the
algebraic laws of dilation and superposition, construct subtraction,
and prove the amplitude regularity needed below.

Write
\begin{equation}\label{eq:reaction-defects}
 F_1(u)=f'(0)u-f(u),\qquad F_2(u)=f(u)-uf'(u),
\end{equation}
both nonnegative on $[0,1]$.
For a classical entire solution $U$ of \eqref{eq:kpp} with $0<U<1$, write
$\mathcal P_U=\partial_t-\Delta-f'(U)$ and let $P_U(t,s)$ denote
the positive evolution for $\mathcal P_Uv=0$. Since
$(\partial_t-\Delta-f'(U))U=F_2(U)\geq0$,
\begin{equation}\label{eq:linearized-submarkov}
 P_U(t,s)[U(s)]\leq U(t).
\end{equation}

To prove the comparison and uniqueness results below, we consider
two entire solutions $X,Y$ of \eqref{eq:kpp} with $0<X,Y<1$.
We assume that $(X-Y)_+=o(Y)$ uniformly as $Y\searrow0$,
and bound this difference by $Y\omega(Y)$, where
$\omega\geq0$ is continuous on $[0,1]$ and $\omega(0)=0$.
The following lemma provides the averaged linearized estimate
needed to deduce $X\leq Y$.

\begin{lemma}\label{lem:cesaro-error}
For every entire solution $U$ of \eqref{eq:kpp} with $0<U<1$,
every continuous $\omega:[0,1]\to\R$ with $\omega(0)=0$,
and every fixed $(t,x)$,
\begin{equation}\label{eq:cesaro-error}
 \lim_{T\to\infty}\frac1T\int_{t-T}^t
 P_U(t,s)[U(s)|\omega(U(s))|](x)\,ds=0.
\end{equation}
\end{lemma}
\begin{proof}
We first show that, for every fixed $\delta\in(0,1)$,
\[
 \lim_{T\to\infty}\frac1T\int_{t-T}^t
 P_U(t,s)[U(s)\mathbf1_{\{U(s)\geq\delta\}}](x)\,ds=0.
\]
To construct an auxiliary function controlling this integral, we
need that $F_1/F_2$ is bounded on $[c,1]$ for every $0<c<1$.
Here and below, we use the convention
\begin{equation}\label{eq:reaction-defect-quotient}
 \left(\frac{F_1}{F_2}\right)(u):=
 \begin{cases}
 F_1(u)/F_2(u),&F_2(u)>0,\\
 0,&F_2(u)=0.
 \end{cases}
\end{equation}
This convention is consistent with $F_1(u)=0$ at such points,
as follows from the integral identities
\eqref{eq:gkpp-reaction-remainders}, which read
\[
 F_1(u)=\int_0^u(u-v)(-f''(v))\,dv,\qquad
 F_2(u)=\int_0^u v(-f''(v))\,dv.
\]
If $F_2(c)>0$, then $F_2(u)\geq F_2(c)$ for $u\geq c$, so the
quotient is bounded. Otherwise, let $\theta\geq c$ be the endpoint
of the interval on which $f(u)=f'(0)u$.
We have $\theta<1$ since $f(1)=0<f'(0)$. Both $F_1$ and $F_2$
vanish on $[0,\theta]$, while for $u>\theta$ the integral formulas give
\[
 F_1(u)\leq(u-\theta)\int_\theta^u(-f''(v))\,dv,
 \qquad
 F_2(u)\geq\theta\int_\theta^u(-f''(v))\,dv>0.
\]
Thus $F_1(u)/F_2(u)\leq(u-\theta)/\theta$, proving the required bound.

Keep $\delta$ fixed, set $d=\ln\delta$, $\ell=1-d$, and choose
$R>-d$. Define a bounded nonnegative $C^1$ function on $(-\infty,0]$ by
\[
 \Psi_R'(h)=
 \begin{cases}
 0,&h\leq-2R,\\
 \ell(h+2R)/R,&-2R<h<-R,\\
 \ell,&-R\leq h<d,\\
 1-h,&d\leq h\leq0,
 \end{cases}
\]
\[
 \Psi_R(h)=A_R+\int_{-2R}^h\Psi_R'(z)\,dz,\qquad
 A_R=1+\ell\sup_{\substack{e^{-R}\leq u\leq1\\F_2(u)>0}}
                   \frac{F_1(u)}{F_2(u)}.
\]
The supremum defining $A_R$ is finite by the boundedness of the
quotient \eqref{eq:reaction-defect-quotient} on $[e^{-R},1]$.
Since $\Psi_R\geq A_R>0$ and $0\leq\Psi'_R\leq\ell$, this choice of
$A_R$ ensures
$F_2(e^h)\Psi_R(h)\geq F_1(e^h)\Psi'_R(h)$ for $h\geq-R$.
Where $F_2(e^h)=0$, the same inequality holds because $F_1(e^h)=0$.
Put $h=\ln U$ and $p=|\nabla h|^2\leq f'(0)$, using
Theorem~\ref{thm:gkpp-ancient}. For a function of $h$, write
\[
 \mathcal A_p\Psi=p\Psi''+\left(p-\frac{f(U)}U\right)\Psi'
   -\frac{F_2(U)}U\Psi
  =p\Psi''+(p-f'(0))\Psi'
   +\frac{F_1(U)}U\Psi'-\frac{F_2(U)}U\Psi.
\]
On $h\geq d$ the first two terms equal
$-f'(0)+(f'(0)-p)h\leq-f'(0)$, and the last two have nonpositive sum.
On $-R\leq h<d$, we have $\Psi_R''=0$,
$(p-f'(0))\Psi_R'\leq0$, and again the last two terms have
nonpositive sum. On $-2R<h<-R$, we use
$\Psi_R''=\ell/R$, $0\leq\Psi_R'\leq\ell$, and $F_2\Psi_R\geq0$
to bound the positive terms by
\[
 \varepsilon_R=f'(0)\ell/R+
             \ell\sup_{0<u\leq e^{-R}}F_1(u)/u\rightarrow0.
\]
Here $\varepsilon_R\to0$ as $R\to\infty$ with $\delta$ fixed,
because $F_1(u)/u\to0$ as $u\searrow0$.
For $h<-2R$, both derivatives of $\Psi_R$ vanish and
$\mathcal A_p\Psi_R=-F_2(U)\Psi_R/U\leq0$.
Combining these estimates gives, almost everywhere,
\[
 \mathcal A_p\Psi_R\leq-f'(0)\mathbf1_{\{U\geq\delta\}}+\varepsilon_R.
\]
For $F_R=U\Psi_R(\ln U)$ the product rule gives
\[
 (\partial_t-\Delta-f'(U))F_R
 =-U\mathcal A_p\Psi_R
 \geq f'(0)U\mathbf1_{\{U\geq\delta\}}-\varepsilon_RU.
\]
The differential inequality holds in the sense of distributions.
Indeed, $u\mapsto u\Psi_R(\ln u)$, extended by zero at zero,
is continuously differentiable with Lipschitz derivative on $[0,1]$:
it equals $A_Ru$ for $u\leq e^{-2R}$ and is piecewise twice
continuously differentiable with bounded second derivative elsewhere.
Thus the junctions introduce no singular terms in the second derivative.
Variation of constants and positivity of $P_U$ give
\[
 \begin{aligned}
 f'(0)\int_{t-T}^t
 P_U(t,s)[U(s)\mathbf1_{\{U(s)\geq\delta\}}](x)\,ds
 &\leq F_R(t,x)-\bigl(P_U(t,t-T)[F_R(t-T)]\bigr)(x)\\
 &\quad+\varepsilon_R\int_{t-T}^t P_U(t,s)[U(s)](x)\,ds\\
 &\leq F_R(t,x)+\varepsilon_RT U(t,x).
 \end{aligned}
\]
The last inequality uses $F_R\geq0$ and
\eqref{eq:linearized-submarkov}. Dividing by $f'(0)T$ yields
\[
 \frac1T\int_{t-T}^tP_U(t,s)[U(s)\mathbf1_{\{U(s)\geq\delta\}}](x)\,ds
 \leq\frac{F_R(t,x)}{f'(0)T}
       +\frac{\varepsilon_R}{f'(0)}U(t,x).
\]
For fixed $R$, $F_R(t,x)$ is finite, so its contribution vanishes
as $T\to\infty$. Letting $R\to\infty$ next proves the claim for
$\{U\geq\delta\}$.

Finally, split $U|\omega(U)|$ over $\{U<\delta\}$ and
$\{U\geq\delta\}$. By positivity and \eqref{eq:linearized-submarkov},
the averaged contribution of the first set is at most
$U(t,x)\sup_{[0,\delta]}|\omega|$. The contribution of the second
set is at most $\|\omega\|_{L^\infty([0,1])}$ times the integral
average just shown to vanish. Hence
\[
 \limsup_{T\to\infty}\frac1T\int_{t-T}^t
 P_U(t,s)[U(s)|\omega(U(s))|](x)\,ds
 \leq U(t,x)\sup_{[0,\delta]}|\omega|.
\]
Letting $\delta\searrow0$ proves the lemma, since $\omega$ is
continuous and $\omega(0)=0$.
\end{proof}

The following theorem shows that a sufficiently small error relative
to an entire solution near zero determines the order of the solutions
and, in particular, yields uniqueness.

\begin{theorem}\label{thm:asymptotic-uniqueness}
If entire solutions $X,Y$ of \eqref{eq:kpp} with $0<X,Y<1$ satisfy
$(X-Y)_+\leq Y\omega(Y)$, where $\omega\geq0$ is continuous and $\omega(0)=0$,
then $X\leq Y$.
In particular, suppose a positive function $Z$ satisfies
$cZ\leq X,Y\leq CZ$ for constants $c,C>0$, and
$|X-Y|=o(Z)$ uniformly as $Z\searrow0$. Then $X=Y$.
\end{theorem}
\begin{proof}
Concavity and Kato's inequality give
\[
 (\partial_t-\Delta)(X-Y)_+\leq f'(Y)(X-Y)_+.
\]
Hence $(X-Y)_+(t)\leq P_Y(t,s)[Y(s)\omega(Y(s))]$ for every $s<t$.
Average over $s\in[t-T,t]$ and apply Lemma~\ref{lem:cesaro-error}.
For the second assertion, $Z\leq Y/c$ and the uniform
small-$Z$ assumption give
\[
 \frac{|X-Y|}{Y}\leq\frac1c\frac{|X-Y|}{Z}\rightarrow0
 \qquad\text{uniformly as }Y\searrow0.
\]
Moreover, $X\leq CZ\leq(C/c)Y$, so
$|X-Y|/Y\leq C/c+1$ everywhere. We may therefore choose a continuous
function $\omega\geq0$, with $\omega(0)=0$, such that
$|X-Y|\leq Y\omega(Y)$. The first assertion yields $X\leq Y$.
Interchanging $X$ and $Y$ gives $Y\leq X$, and hence $X=Y$.
\end{proof}

Theorem~\ref{thm:asymptotic-uniqueness} allows us to establish the
algebraic laws of nonlinear dilation and superposition, analogous
to those of scalar multiplication and addition.

\begin{proposition}\label{prop:pde-full-dilation}
For $a,b>0$ and $U,V,W\in\mathcal C$,
\begin{align}
 D_aD_bU&=D_{ab}U,&D_1U&=U,\label{eq:pde-full-group}\\
 (U\oplus V)\oplus W&=U\oplus(V\oplus W),&
 U\oplus V&=V\oplus U,\label{eq:pde-superposition-assoc}\\
 D_a(U\oplus V)&=D_aU\oplus D_aV,&
 D_aU\oplus D_bU&=D_{a+b}U.\label{eq:pde-superposition-laws}
\end{align}
Thus $D_{1/a}$ is the order-preserving inverse of $D_a$.
\end{proposition}

\begin{proof}
We apply the second assertion of Theorem~\ref{thm:asymptotic-uniqueness}
to each pair of entire solutions below. The bounds
\eqref{eq:cone-basic-bounds} make both members of each pair comparable
to the indicated positive function, with constants depending only on
$a$ and $b$.

For the dilation group law, take $Z=U$. By
\eqref{eq:amplitude-low-value},
\[
 D_a(D_bU)=aD_bU+o(D_bU)=abU+o(U),\qquad
 D_{ab}U=abU+o(U)
 \quad\text{as }U\searrow0.
\]
Theorem~\ref{thm:asymptotic-uniqueness} therefore gives
$D_aD_bU=D_{ab}U$. The identity $D_1U=U$ follows directly from
the Cauchy evolution of the entire solution $U$.

For associativity, take $Z=U+V+W$. Since
$\max(U,V)\geq(U+V)/2$, the elementary bounds give
\[
 \tfrac14Z\leq (U\oplus V)\oplus W,\,
 U\oplus(V\oplus W)\leq Z.
\]
Applying Lemma~\ref{lem:addition-low-value} twice yields
\[
 (U\oplus V)\oplus W=Z+o(Z),\qquad
 U\oplus(V\oplus W)=Z+o(Z)
 \quad\text{as }Z\searrow0.
\]
The same theorem proves associativity. Commutativity follows directly
from the definition \eqref{eq:basic-pullbacks}.

For distributivity, take $Z=U+V$. Combining
\eqref{eq:amplitude-low-value} with
Lemma~\ref{lem:addition-low-value}, we obtain
\[
 D_a(U\oplus V)=aZ+o(Z),\qquad
 D_aU\oplus D_aV=aZ+o(Z)
 \quad\text{as }Z\searrow0.
\]
Theorem~\ref{thm:asymptotic-uniqueness} gives
$D_a(U\oplus V)=D_aU\oplus D_aV$.
Finally, taking $Z=U$, the same estimates give
\[
 D_aU\oplus D_bU=(a+b)U+o(U),\qquad
 D_{a+b}U=(a+b)U+o(U)
 \quad\text{as }U\searrow0.
\]
Another application of Theorem~\ref{thm:asymptotic-uniqueness}
proves the last identity.
If any of $U,V,W$ is zero, the same argument applies whenever the
comparison function is nonzero; the remaining cases follow from
$D_a0=0$ and $U\oplus0=U$.
The group law gives $D_{1/a}D_a=D_aD_{1/a}=D_1$, and order preservation
was established in the construction of $D_a$.
\end{proof}

We next prove an order cancellation law: adding the same solution by
nonlinear superposition preserves and reflects the order.

\begin{proposition}\label{prop:cone-cancellation}
For $U,\,V,\,W\in\mathcal C$, it holds
\[
 U\oplus W\leq V\oplus W \quad\Leftrightarrow\quad U\leq V.
\]
In particular, $U\oplus W=V\oplus W$ implies $U=V$.
\end{proposition}
\begin{proof}
The implication from right to left follows from order preservation.
Conversely, assume $U\oplus W\leq V\oplus W$. For integers $n\geq1$
and $0\leq j<n$, associativity, commutativity, and order preservation give
\[
 \underbrace{U\oplus\cdots\oplus U}_{n-j}
 \oplus\underbrace{V\oplus\cdots\oplus V}_{j}\oplus W
 \leq
 \underbrace{U\oplus\cdots\oplus U}_{n-j-1}
 \oplus\underbrace{V\oplus\cdots\oplus V}_{j+1}\oplus W.
\]
Iterating over $j$ yields
\[
 \underbrace{U\oplus\cdots\oplus U}_{n}\oplus W
 \leq\underbrace{V\oplus\cdots\oplus V}_{n}\oplus W.
\]
Applying the order-preserving map $D_{1/n}$ and using
Proposition~\ref{prop:pde-full-dilation},
we obtain
\[
 U\oplus D_{1/n}W\leq V\oplus D_{1/n}W.
\]
The elementary bounds for superposition imply
\[
 U\leq U\oplus D_{1/n}W
 \leq V\oplus D_{1/n}W\leq V+D_{1/n}W.
\]
Letting $n\to\infty$ proves $U\leq V$.
\end{proof}

The dilation group law also makes the amplitude parameter a strict
parametrization of the solution value at each fixed space--time point.

\begin{corollary}\label{cor:cone-normalization}
For an entire solution $U\in (0,1)$ and any $(t,x)$, the map
$a\mapsto D_aU(t,x)$ is strictly increasing from $(0,\infty)$ onto
$(0,1)$.
Also, for $0<a<b$,
\begin{equation}\label{eq:amplitude-lipschitz}
 0\leq D_bU-D_aU
 \leq (b/a-1)D_aU\leq(b-a)/a.
\end{equation}
\end{corollary}

\begin{proof}
Concavity and order preservation give continuity and monotonicity in $a$.
Fix $0<a<b$. If
$D_aU(t_0,x_0)=D_bU(t_0,x_0)$, the strong maximum principle and
Cauchy uniqueness give $D_aU\equiv D_bU$. By
Proposition~\ref{prop:pde-full-dilation} and
Lemma~\ref{lem:cone-small-amplitudes},
\[
 U=D_{a/b}U=D_{(a/b)^n}U\rightarrow0
 \qquad\hbox{for }n\to\infty,
\]
contradicting $U>0$. Thus $D_aU(t,x)<D_bU(t,x)$ for every $(t,x)$.

Lemma~\ref{lem:cone-small-amplitudes} gives
\[
 \lim_{a\searrow0}D_aU(t,x)=0,\qquad
 \lim_{a\to\infty}D_aU(t,x)=1.
\]
Together with continuity and strict monotonicity, these limits prove
the assertion. The Lipschitz estimate follows from the group law and
$D_{b/a}Z\leq(b/a)Z$.
\end{proof}

We now show that nonlinear dilation commutes with monotone limits
of entire solutions.

\begin{proposition}\label{prop:cone-monotone-dilation}
Suppose \(U_n\in\mathcal C\) converges pointwise monotonically to an entire solution \(U\in\mathcal C\). Then, for every \(a>0\),
\[
D_aU_n\rightarrow D_aU
\quad\text{in }C^{1,2}_{\mathrm{loc}}(\mathbb R\times\mathbb R^N)
\quad\text{as }n\to\infty.
\]
\end{proposition}
\begin{proof}
Fix $a>0$. Order preservation and $0\leq D_aU_n\leq1$ give a
pointwise monotone limit $V=\lim_{n\to\infty}D_aU_n$.
By interior parabolic compactness,
$V\in\mathcal E$ and
\[
 D_aU_n\rightarrow V
 \quad\text{in }C^{1,2}_{\mathrm{loc}}(\R\times\R^N).
\]
For a decreasing sequence $V\leq D_aU_1\in\mathcal C$, and for an
increasing sequence $V\leq D_aU\in\mathcal C$; hence $V\in\mathcal C$.
It remains to identify $V$. If $U_n\searrow U$, then
\[
 D_aU\leq V\leq D_aU_n.
\]
Applying the order-preserving map $D_{1/a}$ and using
Proposition~\ref{prop:pde-full-dilation} gives
\[
 U\leq D_{1/a}V\leq U_n.
\]
Letting $n\to\infty$ yields $D_{1/a}V=U$.
If $U_n\nearrow U$, the same argument gives
\[
 D_aU_n\leq V\leq D_aU
 \quad\Rightarrow\quad
 U_n\leq D_{1/a}V\leq U,
\]
and again $D_{1/a}V=U$ as $n\to\infty$.
In both cases, Proposition~\ref{prop:pde-full-dilation} yields $
 V=D_a(D_{1/a}V)=D_aU$,
which proves the assertion for every $a>0$.
\end{proof}

We next extend monotone continuity from dilation to superposition
for decreasing sequences.

\begin{lemma}
\label{lem:cone-decreasing-superposition}
If $U_n\searrow U$ and $V_n\searrow V$ in $\mathcal C$, then
\[
 U_n\oplus V_n\searrow U\oplus V
 \quad\text{in }C^{1,2}_{\mathrm{loc}}(\R\times\R^N)
 \quad\text{as }n\to\infty.
\]
\end{lemma}
\begin{proof}
For fixed $(t,x)$, the quantity
$\bigl(S_{t+T}[U_n(-T,\cdot)+V_n(-T,\cdot)]\bigr)(x)$ is nonincreasing in both $n$ and
$T>\max\{0,-t\}$. Thus the two monotone limits commute, and
finite-time semigroup continuity gives
\begin{align*}
 \lim_{n\to\infty}(U_n\oplus V_n)(t,x)
 &=\lim_{n\to\infty}\lim_{T\to\infty}
 \bigl(S_{t+T}[U_n(-T,\cdot)+V_n(-T,\cdot)]\bigr)(x)\\
 &=\lim_{T\to\infty}\lim_{n\to\infty}
 \bigl(S_{t+T}[U_n(-T,\cdot)+V_n(-T,\cdot)]\bigr)(x)\\
 &=\lim_{T\to\infty}\bigl(S_{t+T}[U(-T,\cdot)+V(-T,\cdot)]\bigr)(x)\\
 &=(U\oplus V)(t,x).
\end{align*}
Parabolic compactness gives convergence in
$C^{1,2}_{\mathrm{loc}}(\R\times\R^N)$.
\end{proof}

For general locally convergent sequences, we record the one-sided
limit inequalities needed in the construction of subtraction.

\begin{lemma}
\label{lem:cone-semicontinuity}
Suppose $U_n,V_n\in\mathcal C$ converge locally uniformly to $U,V\in\mathcal C$.
At each $(t,x)$,
\[
 D_aU\leq\liminf_nD_aU_n\quad\hbox{for }0<a<1,\qquad
 D_aU\geq\limsup_nD_aU_n\quad\hbox{for }a>1,
\]
\[
 U\oplus V\geq\limsup_n(U_n\oplus V_n).
\]
\end{lemma}
\begin{proof}
Fix $(t,x)$ and $T>\max\{0,-t\}$. Finite-time semigroup continuity
and monotonicity give
\begin{align*}
 \liminf_{n\to\infty}(D_aU_n)(t,x)
 &\geq\lim_{n\to\infty}\bigl(S_{t+T}[aU_n(-T,\cdot)]\bigr)(x)\\
 &=\bigl(S_{t+T}[aU(-T,\cdot)]\bigr)(x)
 \qquad\hbox{for }0<a<1,\\
 \limsup_{n\to\infty}(D_aU_n)(t,x)
 &\leq\lim_{n\to\infty}\bigl(S_{t+T}[aU_n(-T,\cdot)]\bigr)(x)\\
 &=\bigl(S_{t+T}[aU(-T,\cdot)]\bigr)(x)
 \qquad\hbox{for }a>1,\\
 \limsup_{n\to\infty}(U_n\oplus V_n)(t,x)
 &\leq\lim_{n\to\infty}\bigl(S_{t+T}[U_n(-T,\cdot)+V_n(-T,\cdot)]\bigr)(x)\\
 &=\bigl(S_{t+T}[U(-T,\cdot)+V(-T,\cdot)]\bigr)(x).
\end{align*}
Letting $T\to\infty$ in these three inequalities proves the assertions.
\end{proof}

The cancellation law gives uniqueness of a nonlinear difference.
The following theorem establishes its existence whenever the two
solutions are ordered.

\begin{theorem}\label{thm:cone-subtraction}
If $U,V\in\mathcal C$ and $U\leq V$, there is a unique $W\in\mathcal C$
with $U\oplus W=V$. It is the increasing pullback from $V-U$.
\end{theorem}

\begin{proof}
The cases $U=0$ and $U=V$ are immediate. Otherwise $0<U<V$
by the strong maximum principle.
For $b>0$ the ordinary function $F_b=(1+b)V-bU$ is a bounded
supersolution: write $V=(F_b+bU)/(1+b)$ and apply concavity.
Its decreasing pullback $A_b$ satisfies
\[
 V\leq A_b\leq F_b\leq(1+b)V,\qquad
 A_b+D_bU\geq D_{1+b}V.
\]
The last inequality follows from subadditivity since
$F_b+bU=(1+b)V$. Therefore
\[
 0\leq F_b-A_b
 \leq (1+b)V-D_{1+b}V+D_bU-bU=o(V)
 \quad(V\searrow0).
\]
For $b\geq1$ the last term is nonpositive; for $b<1$ use
$U\leq V$ and the uniform amplitude error modulus.

Fix $s>b$. The solutions
$D_sU\oplus A_b$ and $D_{s-b}U\oplus D_{1+b}V$
are bounded between $V$ and a fixed multiple of $V$, and both equal
$(s-b)U+(1+b)V+o(V)$ as $V\searrow0$.
Asymptotic uniqueness gives
\begin{equation}\label{eq:affine-cancellation}
 D_sU\oplus A_b=D_{s-b}U\oplus D_{1+b}V.
\end{equation}
To let $b\nearrow s$, note that for $b<c$,
\[
 F_b\leq F_c\leq(1+c-b)F_b,\qquad
 A_b\leq A_c\leq(1+c-b)A_b.
\]
The last inequalities follow by Cauchy comparison and scaling.
They also hold for $D_sU\oplus A_b$, with the same multiplicative
upper factor, by applying comparison to its initial sums.
Thus its limit as $b\nearrow s$ is $D_sU\oplus A_s$.
The right side of \eqref{eq:affine-cancellation} tends to
$D_{1+s}V$, squeezed between that term with parameter $1+b$ and its
ordinary sum with $D_{s-b}U$. Hence
\[
 D_sU\oplus A_s=D_{1+s}V.
\]
Put $\alpha_s=s/(1+s)$ and $W_s=D_{1/(1+s)}A_s$.
Then $D_{\alpha_s}U\oplus W_s=V$.
Cancellation shows that $W_s$ decreases as $s$ increases.
Its limit $W\leq V$ belongs to $\mathcal C$.
Since $V\leq U\oplus W_s$, decreasing continuity gives
$V\leq U\oplus W$. For fixed $a<1$ and sufficiently large $s$,
$D_aU\oplus W\leq V$ and therefore $D_a(U\oplus W)\leq V$.
Let $a\nearrow1$ to obtain the opposite inequality.

Finally $q=V-U$ is a subsolution, since
$f(V)-f(U)\leq f(V-U)$. Its increasing pullback $\widetilde W$ is its least
entire majorant, with $q\leq \widetilde W\leq V$.
From $V=U\oplus W\leq U+W$ one has $q\leq W$, so $\widetilde W\leq W$.
Also $V\leq U+\widetilde W$, and the greatest-minorant property of $\oplus$
gives $V\leq U\oplus \widetilde W\leq U\oplus W=V$.
Cancellation yields $\widetilde W=W$, and uniqueness follows likewise.
\end{proof}

Nonlinear subtraction allows us to pass to increasing limits in
superposition, as required for the decompositions below.

\begin{corollary}\label{cor:cone-increasing-superposition}
If $U_n\nearrow U$ in $\mathcal C$ and $V\in\mathcal C$, then
\[
 U_n\oplus V\nearrow U\oplus V
 \quad\text{in }C^{1,2}_{\mathrm{loc}}(\R\times\R^N)
 \quad\text{as }n\to\infty.
\]
\end{corollary}
\begin{proof}
Since $U_n\oplus V$ is increasing in $n$ and bounded by $1$,
interior parabolic estimates give an entire solution $Z$ such that
\[
 U_n\oplus V\rightarrow Z
 \qquad\hbox{in }C^{1,2}_{\mathrm{loc}}(\R\times\R^N)
 \quad\hbox{as }n\to\infty.
\]
 Since $0\leq Z\leq U\oplus V\in\mathcal C$, it belongs
to $\mathcal C$. Since $Z\geq V$, Theorem~\ref{thm:cone-subtraction}
gives a unique $X\in\mathcal C$ such that $Z=V\oplus X$. Order cancellation yields $X\geq U_n$
for every $n$, so $X\geq U$. Thus $Z\geq U\oplus V$, and hence
$Z=U\oplus V$, which proves the assertion.
\end{proof}

To relate the ordered-space operations to ordinary values of
solutions, we compare nonlinear convex combinations with their
pointwise averages.

\begin{lemma}
\label{lem:faithfulness-evaluation}
For $Z_1,\cdots,Z_m\in\mathcal C$ and positive numbers
$\theta_j$ with $\sum_j\theta_j=1$,
\begin{equation}\label{eq:faithfulness-physical-jensen}
 \left(\bigoplus_{j=1}^m D_{\theta_j}Z_j\right)(t,x)
\geq\sum_{j=1}^m\theta_j Z_j(t,x).
\end{equation}
\end{lemma}

 \begin{proof}
First take two terms, with weights $\theta$ and $1-\theta$.
The ordinary convex combination
$F=\theta Z_1+(1-\theta)Z_2$ is a subsolution and satisfies
$F\leq D_\theta Z_1+D_{1-\theta}Z_2$.
For $0<\theta<1$, comparison gives
\[
 F(t)\leq S_{t+T}[F(-T)]
 \leq S_{t+T}[D_\theta Z_1(-T)+D_{1-\theta}Z_2(-T)].
\]
Let $T\to\infty$ to obtain the asserted two-term inequality.
The endpoints are immediate. Induction, using associativity and
distributivity of dilation, proves the finite version.
\end{proof}

To justify the later small-amplitude analysis, we establish
first-order regularity in the amplitude parameter and convergence
of the corresponding derivatives.

\begin{lemma}\label{lem:amplitude-C1-degenerate}
For every entire solution $U$ of \eqref{eq:kpp} with $0<U<1$,
the map $a\mapsto D_aU$ is
$C^1$ on $(0,\infty)$, with values locally in
$C^{1,2}(\mathbb R\times\mathbb R^N)$. If $W=D_aU$, then
\begin{equation}\label{eq:amplitude-C1-bounds}
 \frac{q(W)}{\lambda_*a}\leq\partial_aD_aU\leq\frac Wa,
 \qquad
 (\partial_t-\Delta-f'(W))\partial_aD_aU=0.
\end{equation}
If $D_aU_n\to D_aU$ pointwise for every $a>0$, then
\[
 \partial_aD_aU_n\rightarrow\partial_aD_aU
 \quad\text{in }C^{1,2}_{\mathrm{loc}}(\mathbb R\times\mathbb R^N),
\]
locally uniformly in $a\in(0,\infty)$. The same derivative convergence
holds for the finite pullbacks $S_{t+T}[aU(-T)]$ as $T\to\infty$.
In particular,
\begin{equation}\label{eq:trace-generator-equations}
 V[U]:=\left.\partial_aD_aU\right|_{a=1}
 =\lim_{T\to\infty}P_U(t,-T)[U(-T)],
 \qquad \mathcal P_UV[U]=0.
\end{equation}
\end{lemma}

\begin{proof}
At a fixed space--time point, $a\mapsto D_aU$ is concave and increasing.
Fix $a>0$, set $W=D_aU$, and denote its left and right amplitude
derivatives by $v_-$ and $v_+$, respectively. The group law gives
$D_{a+h}U=D_{1+h/a}W$. Applying
\eqref{eq:inverse-wave-upper} and
\eqref{eq:inverse-wave-lower} to $W$, then letting $h\to0$ from
the two sides, gives
\[
 \frac{q(W)}{\lambda_*a}\leq v_+\leq v_-\leq\frac Wa.
\]
The last bound follows also from concavity and $D_0U=0$.

For $0<|h|<a/2$, the difference quotient
\[
 v_h=\frac{D_{a+h}U-D_aU}{h}
\]
is nonnegative, bounded by $2/a$, and satisfies
\[
 (\partial_t-\Delta)v_h=c_hv_h,\qquad
 c_h=\int_0^1f'\bigl(W+\theta(D_{a+h}U-W)\bigr)\,d\theta.
\]
Local convergence $c_h\to f'(W)$ and interior parabolic estimates
show that $v_\pm$ are bounded entire solutions of
$\mathcal P_Wv_\pm=0$, where
$\mathcal P_W=\partial_t-\Delta-f'(W)$.
Consequently $\eta=v_--v_+\geq0$ satisfies
\[
 \mathcal P_W\eta=0,\qquad
 0\leq\eta\leq\frac Wa E(W),\qquad
 E(u)=1-\frac{q(u)}{\lambda_*u}\quad\text{for }0<u<1.
\]
Extend $E$ continuously to $[0,1]$ by $E(0)=0$ and $E(1)=1$.
If $P_W(t,s)$ is the positive evolution of $\mathcal P_W$, then
for every $T>0$,
\[
 0\leq\eta(t)
 =\frac1T\int_{t-T}^tP_W(t,s)[\eta(s)]\,ds
 \leq\frac1{aT}\int_{t-T}^t
 P_W(t,s)\bigl[W(s)E(W(s))\bigr]\,ds.
\]
Lemma~\ref{lem:cesaro-error} makes the last expression tend to zero.
Thus $v_-=v_+$, proving differentiability and
\eqref{eq:amplitude-C1-bounds}.

For differentiable concave functions $G_n$, the secant inequalities give
\[
 \frac{G_n(a+\delta)-G_n(a)}\delta
 \leq G_n'(a)
 \leq\frac{G_n(a)-G_n(a-\delta)}\delta
 \qquad\text{for }0<\delta<a.
\]
If $G_n\to G$ locally uniformly and $G$ is differentiable, letting
$n\to\infty$ and then $\delta\to0$ yields $G_n'(a)\to G'(a)$.
Since $G_n'$ is nonincreasing and $G'$ is continuous, this convergence
is uniform on compact subintervals of $(0,\infty)$.
Apply this to $G_n(a)=D_aU_n(t,x)$. The bound
$0\leq\partial_aD_aU_n\leq1/a$, the linear equations
\eqref{eq:amplitude-C1-bounds}, and interior parabolic estimates give
\[
 \sup_{a\in I}
 \|\partial_aD_aU_n-\partial_aD_aU\|_{C^{1,2}(K)}\rightarrow0
 \qquad\text{as }n\to\infty
\]
for every compact interval $I\subset(0,\infty)$ and compact cylinder
$K\subset\R\times\R^N$. Taking
$G_n(a)=D_{a+a_n-a_0}U(t,x)$ for $a_n\to a_0>0$ also proves
continuity of $a\mapsto\partial_aD_aU$.

For $u_{a,T}(t,\cdot)=S_{t+T}[aU(-T,\cdot)]$, we similarly have
\[
 0\leq\partial_au_{a,T}\leq\frac{u_{a,T}}a
 \leq\frac{\max\{a,1\}}a,\qquad
 (\partial_t-\Delta-f'(u_{a,T}))\partial_au_{a,T}=0.
\]
The same argument gives $\partial_au_{a,T}\to\partial_aD_aU$
in $C^{1,2}_{\loc}$, locally uniformly in $a>0$, as $T\to\infty$.
At $a=1$, $u_{1,T}=U$ and
\[
 \left.\partial_au_{a,T}(t,\cdot)\right|_{a=1}
 =P_U(t,-T)[U(-T,\cdot)],
\]
which proves \eqref{eq:trace-generator-equations}.
\end{proof}

\subsection{The ordered solution space and translations}
\label{subsec:ordered-solution-space}

Using the algebraic laws and subtraction from
Subsection~\ref{subsec:asymptotic-uniqueness-laws}, we identify
$\mathcal C$ with the positive cone of an ordered vector space $E$.
We equip each $E_U$ with a complete order-unit norm, represent it
by continuous functions, and relate norm convergence to local
parabolic convergence. We then describe space--time translations
on $E_U$ as multiplication operators in this representation.

Let $E$ consist of the formal differences $U\ominus V$, with
$U,\,V\in\mathcal C$, defined by
\[
 U\ominus V=U'\ominus V'\quad\Leftrightarrow\quad
 U\oplus V'=U'\oplus V.
\]
This rule defines an equivalence relation on the pairs $(U,V)$:
reflexivity and symmetry are immediate, while transitivity follows
from Proposition~\ref{prop:cone-cancellation}. Here $U\ominus V$
denotes the equivalence class represented by $(U,V)$. Define
\begin{align}
 (U\ominus V)\oplus(U'\ominus V')&=(U\oplus U')\ominus(V\oplus V'),\notag\\
 a\odot(U\ominus V)&=D_aU\ominus D_aV\quad\hbox{for } a>0,\label{eq:ordered-space-scalar-multiplication}\\
 (U\ominus V)\ominus(U'\ominus V')&=(U\oplus V')\ominus(V\oplus U'),\notag\\
 0\odot(U\ominus V)&=0\ominus0.\notag
\end{align}
For negative scalars use $(-a)\odot x=0\ominus(a\odot x)$, for $a>0$ and $x\in E$.
Since $\mathcal C$ is closed under $\oplus$ and $D_a$ for $a>0$,
these operations take values in $E$. The commutativity, associativity,
and distributive laws in
Proposition~\ref{prop:pde-full-dilation}, and cancellation in
Proposition~\ref{prop:cone-cancellation} ensure that these operations
are independent of the representatives and satisfy the vector-space
axioms. Thus $(E,\oplus,\odot)$ is a real vector space, with zero
$0\ominus0$ and additive inverse $V\ominus U$ for $U\ominus V$.
Moreover, $U\mapsto U\ominus0$ is injective, since
$U\ominus0=V\ominus0$ implies $U\oplus0=V\oplus0$, hence $U=V$.
Define its order by
\begin{equation}\label{eq:ordered-space-order}
 U\ominus V\leq U'\ominus V'
 \quad\Leftrightarrow\quad U\oplus V'\leq U'\oplus V.
\end{equation}
For $U,\,V,\,\widetilde U,\,\widetilde V,\,U',\,V'\in\mathcal C$ with
$U\ominus V=\widetilde U\ominus\widetilde V$,
Proposition~\ref{prop:cone-cancellation} gives
\[
 U\oplus V'\leq U'\oplus V
 \quad\Leftrightarrow\quad
 \widetilde U\oplus V'\leq U'\oplus\widetilde V.
\]
The same argument applies when the representation of $U'\ominus V'$
is changed. Thus the comparison depends only on the two elements of
$E$, not on the pairs chosen to represent them. This order makes $E$
an ordered vector space. We identify
$U\in\mathcal C$ with $U\ominus0\in E$. On this copy of $\mathcal C$,
the order is the original order, vector addition is $\oplus$,
and $a\odot U=D_aU$ for $a>0$. We use $\oplus$, $\ominus$, and $\odot$ for addition, subtraction,
and scalar multiplication in $E$, respectively. These operations extend
those on $\mathcal C$. 

By Theorem~\ref{thm:cone-subtraction}, every positive element
$U\ominus V$ is represented by the unique $W\in\mathcal C$ satisfying
$V\oplus W=U$. Thus the positive cone of $E$ is precisely $\mathcal C$.

An ordered vector space is called a \emph{vector lattice} if every
pair of elements $x,y$ has a least upper bound $x\vee y$ and a
greatest lower bound $x\wedge y$ with respect to its order.
We next establish this property for $E$, together with the
Archimedean property stated below.

\begin{proposition}\label{prop:cone-vector-lattice}
The space $E$ is a vector lattice, and its order is Archimedean:
if $n\odot x\leq y$ for every positive integer $n$, then $x\leq0$.
\end{proposition}

\begin{proof}
For $U,V\in\mathcal C$, form the ordinary pointwise minimum
$m(t,x)=\min\{U(t,x),V(t,x)\}$. Order preservation gives
\[
 S_h[m(t)]\leq\min\{S_h[U(t)],S_h[V(t)]\}=m(t+h).
\]
For $R>T$, applying this inequality from time $-R$ to $-T$ and then
evolving to time $t$ shows
\[
 S_{t+R}[m(-R)]\leq S_{t+T}[m(-T)].
\]
Thus the pullbacks decrease. Their limit $M$ is entire by parabolic
compactness and satisfies $0\leq M\leq U,V$. In particular
$M\in\mathcal C$. If $Z\in\mathcal C$ satisfies $Z\leq U,V$,
then comparison from $-T$ gives
$Z(t)\leq S_{t+T}[m(-T)]$, hence $Z\leq M$.
Therefore $M$ is the greatest positive lower bound of $U,V$;
write it as $U\wedge V$.

For $W\in\mathcal C$, let $R=(U\oplus W)\wedge(V\oplus W)$. Since $W\leq R$, natural
subtraction writes $R=W\oplus X$ with $X\in\mathcal C$.
Order cancellation gives $X\leq U,V$, and therefore
$R\leq W\oplus(U\wedge V)$. Conversely,
$W\oplus(U\wedge V)$ is a positive common lower bound of $U\oplus W$ and
$V\oplus W$. Hence
\begin{equation}\label{eq:cone-meet-translation}
 (U\oplus W)\wedge(V\oplus W)=(U\wedge V)\oplus W.
\end{equation}
This meet is greatest even among signed lower bounds. To see this,
let $z\in E$ satisfy $z\leq U,V$, and choose $W\in E_+$ with
$z\oplus W\geq0$; such a $W$ exists by the construction of $E$ as
differences of positive elements. Then
\[
 z\oplus W\leq(U\oplus W)\wedge(V\oplus W)=(U\wedge V)\oplus W,
\]
so $z\leq U\wedge V$.

For arbitrary $x,y\in E$, choose $W\geq0$ such that $x\oplus W$ and
$y\oplus W$ are positive and define
\[
 x\wedge y=((x\oplus W)\wedge(y\oplus W))\ominus W.
\]
The preceding greatest-lower-bound property shows directly that this
is the greatest lower bound of $x,y$, so it is independent of $W$.
Joins are given by $x\vee y=0\ominus((0\ominus x)\wedge(0\ominus y))$. Thus $E$ is a
vector lattice.

To prove the Archimedean property, suppose that $n\odot x\leq y$ for all
positive integers $n$. Taking positive parts in the vector lattice
gives $n\odot x^+\leq y^+$. Identify $x^+=U$ and $y^+=V$ with elements
of $\mathcal C$. Then
\[
 U\leq D_{1/n}V.
\]
Lemma~\ref{lem:cone-small-amplitudes} gives $U=0$ on letting
$n\to\infty$. Consequently $x\leq0$, as required.
\end{proof}

Fix a nonzero $U\in\mathcal C$ and define its principal ideal by
\[
 E_U=\{x\in E:|x|\leq c\odot U\text{ for some }c\geq0\},
 \qquad
 \|x\|_U=\inf\{c\geq0:|x|\leq c\odot U\}.
\]
Here $|x|=x\vee(0\ominus x)$ is the lattice absolute value.
The element $U$ is an \emph{order unit} of $E_U$: for every
$x\in E_U$, there exists $c>0$ such that
$(-c)\odot U\leq x\leq c\odot U$.
The corresponding norm $\|\cdot\|_U$ is called the order-unit norm.

The following proposition proves completeness of the order-unit norm
and represents $E_U$ by continuous functions, preserving the norm
and the order operations.

\begin{proposition}\label{prop:cone-principal-ideal}
The space $E_U$, equipped with the order-unit norm $\|\cdot\|_U$,
is complete. Moreover, there exist a compact Hausdorff space $K_U$
and a linear bijection $J:E_U\rightarrow C(K_U)$ such that
\[
 J(U)=1,\qquad \|Jx\|_\infty=\|x\|_U,\qquad
 J(x\vee y)=\max\{Jx,Jy\}
 \quad\text{for }x,y\in E_U.
\]
\end{proposition}
\begin{proof}
For $x,y\in E_U$ and $\alpha\in\R$, the lattice inequalities
\[
 |\alpha\odot x|=|\alpha|\odot|x|,\qquad
 |x\oplus y|\leq|x|\oplus|y|
\]
give homogeneity and the triangle inequality. By
Proposition~\ref{prop:cone-vector-lattice},
\[
 \|x\|_U\leq c
 \ \Longrightarrow\
 |x|\leq(c+n^{-1})\odot U\quad\text{for every }n\geq1
 \ \Longrightarrow\
 |x|\leq c\odot U.
\]
Thus $\|\cdot\|_U$ is a norm, its defining infimum is attained, and
\[
 \begin{gathered}
 |x|\leq|y|\ \Longrightarrow\ \|x\|_U\leq\|y\|_U,\\
 \|x\vee y\|_U=\max\{\|x\|_U,\|y\|_U\}
 \qquad\text{for }x,y\geq0.
 \end{gathered}
\]
The second identity follows from
$x\vee y\leq c\odot U\iff x,y\leq c\odot U$.

If $0\leq A_n\nearrow$ and $A_n\leq c\odot U$, with $c>0$,
interior parabolic compactness gives
\[
 A_n\rightarrow A\quad\text{in }C^{1,2}_{\loc}(\R\times\R^N),
 \qquad 0\leq A\leq D_cU<1,\qquad A\in\mathcal C\cap E_U.
\]
For every upper bound $B\in E$,
\[
 B\geq A_n\geq0\quad\text{for all }n
 \ \Longrightarrow\ B\in\mathcal C,\quad B\geq A.
\]
Hence $A=\sup_n A_n$ in $E$. The case $c=0$ is immediate.

Let $\{x_n\}_{n\geq1}\subset E_U$ be norm Cauchy. Choose a subsequence
and put $d_j=x_{n_{j+1}}\ominus x_{n_j}$ so that
\[
 |d_j|\leq2^{-j}\odot U,\qquad
 P_m:=\bigoplus_{j=1}^m d_j^+,\qquad
 N_m:=\bigoplus_{j=1}^m d_j^-.
\]
Then $0\leq P_m,N_m\nearrow\leq U$, so the preceding argument gives
$P=\sup_mP_m$ and $N=\sup_mN_m$ in $E_U$. For $k\geq m$,
\[
 0\leq P_k\ominus P_m,\ N_k\ominus N_m
 \leq\left(\sum_{j=m+1}^k2^{-j}\right)\odot U
 \leq2^{-m}\odot U.
\]
Taking suprema yields
\[
 0\leq P\ominus P_m,\ N\ominus N_m\leq2^{-m}\odot U.
\]
Since $x_{n_{m+1}}=x_{n_1}\oplus P_m\ominus N_m$,
\[
 \|x_{n_{m+1}}\ominus(x_{n_1}\oplus P\ominus N)\|_U
 \leq2^{1-m}\rightarrow0\qquad\text{as }m\to\infty.
\]
The Cauchy property therefore gives
$x_n\to x_{n_1}\oplus P\ominus N$ in $E_U$, proving completeness.

Finally,
\[
 \{x\in E_U:\|x\|_U\leq1\}=[(-1)\odot U,U].
\]
Kakutani's representation theorem \cite{Kakutani1941}, applied
with the norm properties established above, gives $K_U$ and $J$
with the asserted properties.
\end{proof}

The norm constructed above must also control convergence of the
physical solutions. The next lemma connects the order-unit norm
to local parabolic convergence.

\begin{lemma}
\label{lem:cone-unit-norm-local}
If $W,V\in E_U\cap\mathcal C$ and
$\|W\ominus V\|_U\leq\varepsilon$ with $\varepsilon>0$, then their
ordinary pointwise difference satisfies
\[
 |W(t,x)-V(t,x)|\leq D_\varepsilon U(t,x).
\]
In particular, for $V_n,V\in E_U\cap\mathcal C$,
\[
 \|V_n\ominus V\|_U\rightarrow0
 \quad\Longrightarrow\quad
 V_n\rightarrow V\quad\text{in }C^{1,2}_{\loc}(\R\times\R^N)
 \qquad\text{as }n\to\infty.
\]
\end{lemma}

\begin{proof}
The norm inequality gives
$W\leq V\oplus D_\varepsilon U$ and $V\leq W\oplus D_\varepsilon U$ in $E$.
In terms of entire solutions these become
\[
 W\leq V\oplus D_\varepsilon U\leq V+D_\varepsilon U,
 \qquad
 V\leq W\oplus D_\varepsilon U\leq W+D_\varepsilon U,
\]
where the final sums are ordinary sums of functions. This proves
the pointwise bound. Lemma~\ref{lem:cone-small-amplitudes} and
parabolic compactness prove the convergence assertion.
\end{proof}

We now study space--time translations on the ordered solution space
and on its principal ideals.
For $s\in\mathbb R$ and $h\in\mathbb R^N$, define
\[
 (T_{s,h}V)(t,x)=V(t+s,x+h)\qquad\text{for }V\in\mathcal C.
\]
Translations preserve superposition and dilation, so they extend to
linear order isomorphisms of $E$, with inverse $T_{-s,-h}$.

The differential estimates now give order bounds for translations.
These bounds show that translations preserve the subspaces $E_U$
and will lead to their multiplication-operator representation.

\begin{lemma}\label{lem:cone-translation-bounds}
The following operator inequalities hold on $E_+=\mathcal C$:
\[
 e^{-\lambda_*|h|}I\leq T_{0,h}\leq e^{\lambda_*|h|}I,
\]
and, for $s\geq0$,
\[
 I\leq T_{s,0}\leq e^{2f'(0)s}I.
\]
Every $T_{s,h}$ preserves every principal ideal $E_U$.
\end{lemma}
\begin{proof}
The bounds \eqref{eq:cone-ancient-input} imply the ordinary
pointwise inequalities
\[
 e^{-\lambda_*|h|}V(t,x)\leq V(t,x+h)\leq e^{\lambda_*|h|}V(t,x),
\]
\[
 V(t,x)\leq V(t+s,x)\leq e^{2f'(0)s}V(t,x)\qquad\text{for }s\geq0.
\]
To convert these into inequalities for vector-space scalar
multiplication, observe that if entire solutions $Z,V$ satisfy
$Z\leq aV$ as ordinary functions, then comparison from $-T$ yields
\[
 Z(t)\leq S_{t+T}[aV(-T)],
\]
and hence $Z\leq D_aV$. The reversed ordinary inequality gives
$Z\geq D_aV$ by the same argument. Applying this observation to
the preceding translation bounds proves the two operator inequalities.
For negative $s$, inversion gives
$e^{2f'(0)s}I\leq T_{s,0}\leq I$. In particular, for all $s,h$,
\[
 e^{-2f'(0)|s|-\lambda_*|h|}I\leq T_{s,h}
 \leq e^{2f'(0)|s|+\lambda_*|h|}I.
\]
If $V\in E_U$ and $|V|\leq c\odot U$, then, because $T_{s,h}$ is a lattice
isomorphism,
\[
 |T_{s,h}V|=T_{s,h}|V|
 \leq c\odot T_{s,h}U
 \leq\bigl(ce^{2f'(0)|s|+\lambda_*|h|}\bigr)\odot U.
\]
Thus $T_{s,h}$ maps $E_U$ into itself; its inverse does also.
\end{proof}

In this representation, space--time translations become multiplication
by positive continuous functions. We identify their exponential form.

\begin{proposition}\label{prop:cone-translation-multipliers}
For the representation $J:E_U\to C(K_U)$, there are continuous
functions
\[
 \lambda:K_U\to[0,2f'(0)],\qquad
 \xi:K_U\to\overline B_1\subset\mathbb R^N
\]
such that
\begin{equation}\label{eq:cone-joint-multipliers}
 J(T_{s,h}V)(k)
 =e^{s\lambda(k)+\lambda_*h\cdot\xi(k)}J(V)(k)
\end{equation}
for every $V\in E_U$, $k\in K_U$, and $(s,h)\in\mathbb R^{1+N}$.
Moreover,
\[
 \|T_{s,h}-I\|_{E_U\to E_U}
 \leq e^{2f'(0)|s|+\lambda_*|h|}-1.
\]
\end{proposition}
\begin{proof}
We first prove an elementary fact about positive operators. Let
$T:C(K_U)\to C(K_U)$ be linear and satisfy $0\leq T\leq\beta I$.
For real $f\in C(K_U)$,
\[
 |Tf|\leq T|f|\leq\beta|f|.
\]
Fix $k\in K_U$. If $f(k)=0$, this inequality gives $(Tf)(k)=0$.
Apply this to $f-f(k)1$ and obtain
\[
 (Tf)(k)=g(k)f(k),\qquad g=T1\in C(K_U).
\]
Thus $T$ is multiplication by the continuous function $g$.
If also $T\geq\alpha I$, then $\alpha\leq g\leq\beta$.

Transport each translation to $C(K_U)$ and apply this fact using
Lemma~\ref{lem:cone-translation-bounds}. Its multiplier
$g_{s,h}$ is strictly positive, continuous in $k$, and satisfies
\[
 g_{s+r,h+\ell}=g_{s,h}g_{r,\ell},\qquad
 |\ln g_{s,h}(k)|\leq2f'(0)|s|+\lambda_*|h|.
\]
For a fixed $k$, the logarithm is therefore an additive function of
$(s,h)$ that is continuous at the origin. It is a real linear
functional, so
\[
 \ln g_{s,h}(k)=s\lambda(k)+\lambda_*h\cdot\xi(k).
\]
One may define
$\lambda(k)=\ln g_{1,0}(k)$ and
$\xi_i(k)=\lambda_*^{-1}\ln g_{0,e_i}(k)$, which shows their continuity in $k$.
The inequalities $1\leq g_{s,0}\leq e^{2f'(0)s}$ for $s\geq0$ give
$0\leq\lambda\leq2f'(0)$. The inequalities
$e^{-\lambda_*|h|}\leq g_{0,h}\leq e^{\lambda_*|h|}$ give
$|\lambda_*h\cdot\xi(k)|\leq\lambda_*|h|$ for all $h$, hence $|\xi(k)|\leq1$.
This proves \eqref{eq:cone-joint-multipliers}.
Finally, the norm of a multiplication operator on $C(K_U)$ is the
supremum norm of its multiplier. The bound on the logarithm therefore
gives the stated operator-norm estimate directly.
\end{proof}

\section{Linear representations and small-amplitude limits}
\label{sec:general-finite-lift}
\label{sec:generator-critical-trace}

This section relates nonlinear entire solutions to nonnegative
entire solutions of the linearized equation and their representing
measures. We first assume the finiteness condition
\eqref{eq:general-finite-lift-hypothesis}. Under this condition,
linear evolution of the past data $U(-T,\cdot)$ has a finite limit
$H$ as $T\to\infty$. We represent $H$ by a finite measure, recover
$U$ by nonlinear Cauchy evolution, and approximate $U$ by
Hamel--Nadirashvili solutions associated with finite atomic measures.

We then establish a sufficient condition for this reconstruction.
If $|\nabla\ln U|\leq\lambda_*r$ for some $0\leq r<1$, the
small-amplitude quotient $D_aU/a$ has a finite limit represented
by a measure on $\overline B_r$. We prove that it equals the
linear limit obtained from the past data. Thus the strict gradient
bound implies the finiteness condition needed for reconstruction
and finite atomic approximation.

\subsection{Linear reconstruction and finite-front approximation}
\label{subsec:linear-reconstruction}

We prove reconstruction and finite HN approximation under
\eqref{eq:general-finite-lift-hypothesis}. The preliminary estimates
treat capped critical exponential fields and finite families of fronts;
they will be used to control the reconstruction and identify its
finite atomic approximants.

Throughout the section, $f$ satisfies \eqref{eq:general-kpp-assumptions}
with $f'(0)>0$, and $S^f_\tau=S_\tau$ denotes its Cauchy evolution.
We suppress the superscript $f$ in the notation $V_h=V_h^f$ for the
capped pullback constructed below.

Concavity implies $0\leq f(u)\leq f'(0)u$ and
$f'(u)\leq f'(0)$. For $F_1$ defined in
\eqref{eq:reaction-defects}, Taylor's formula gives
\begin{equation}\label{eq:general-quadratic-error}
 0\leq F_1(u)\leq M u^2\quad\text{for }0\leq u\leq1,
 \qquad M=\tfrac12\|f''\|_{L^\infty([0,1])}.
\end{equation}
 In particular $M\geq f'(0)$, since $f'(0)-f(1)=f'(0)$. Subtracting two Cauchy equations and applying
the maximum principle gives
\begin{equation}\label{eq:general-cauchy-stability}
 |S^f_\tau[a]-S^f_\tau[b]|\leq L_\tau[|a-b|],
 \qquad 0\leq a,b\leq1.
\end{equation}
The concave extension fixed after \eqref{eq:general-kpp-assumptions}
is used if a sum exceeds one. Lemma~\ref{lem:pde-semigroup-inequalities}
and comparison give
\begin{equation}\label{eq:general-subadditivity}
 S^f_\tau[\min\{a+b,1\}]\leq S^f_\tau[a]+S^f_\tau[b].
\end{equation}

We first consider linear solutions generated by finite measures on the critical sphere.
Let $\mu$ be a finite positive Borel measure on $\mathbb S^{N-1}$ and set
\begin{equation}\label{def-h}
 H_{\mu}(t,x)
 =e^{2f'(0)t}\int_{\mathbb S^{N-1}}
   e^{\lambda_*\nu\cdot x}\,d\mu(\nu)
 =\int_{\mathbb S^{N-1}}
   e^{\lambda_*\nu\cdot x+2f'(0)t}\,d\mu(\nu).
\end{equation}
For each $\nu\in\mathbb S^{N-1}$, the integrand is a planar
exponential traveling wave of the linearized equation
$u_t=\Delta u+f'(0)u$, propagating with speed $c_*$ in the
direction $-\nu$.
Thus $H_{\mu}$ is a superposition of linear traveling waves at the
critical speed and is itself an entire solution of the
linearized equation.

Let $L_t=e^{f'(0)t}P_t$ denote the semigroup
associated with the linearized equation
$u_t=\Delta u+f'(0)u$. The semigroup identity gives
$\bigl(L_{t+n}[aH_{\mu}(-n,\cdot)]\bigr)(x)=aH_{\mu}(t,x)$ for $t+n>0$.
The following calculation shows that capping the past data at one
retains exactly half of the critical linear field in the limit.
We will use this loss to prove that the corresponding nonlinear
pullback limit is zero.

\begin{lemma}\label{lem:critical-gaussian-half}
For every $a>0$ and
$(t,x)\in\R\times\R^N$, it holds
\[
 \bigl(L_{t+n}[\min\{aH_{\mu}(-n,\cdot),1\}]\bigr)(x)\rightarrow\frac a2H_{\mu}(t,x)
 \qquad\hbox{as $n\in\mathbb{N}_+$ and }n\to +\infty.
\] 
\end{lemma}
\begin{proof}
The result is immediate if $\mu=0$, so assume $M=\mu(\mathbb S^{N-1})>0$.
We first verify that for $\mu$-almost every $\nu$,
\begin{equation}\label{eq:critical-local-mass}
 \frac{\ln\mu(B(\nu,n^{-1/2}))}{\sqrt n}\rightarrow0.
\end{equation}
Equivalently, there exists a Borel set \(E\subset\mathbb S^{N-1}\)
with \(\mu(E)=0\) such that \eqref{eq:critical-local-mass} holds
for every \(\nu\in\mathbb S^{N-1}\setminus E\).
Balls here are intersected with the sphere.

For each $n\geq1$, intersect the sphere with a grid of half-open cubes
of side length $1/(2\sqrt{Nn})$. The nonempty intersections form a
Borel partition $Q_{n,1},\cdots,Q_{n,K_n}$ satisfying
\[
 K_n\leq C_Nn^{N/2},\qquad
 \operatorname{diam}Q_{n,j}\leq\frac1{2\sqrt n}<n^{-1/2}.
\]

Fix $\varepsilon>0$ and define the set of points with small local mass by
\[
 E_{n,\varepsilon}
 :=\{\nu\in\mathbb S^{N-1}:\mu(B(\nu,n^{-1/2}))
                              <e^{-\varepsilon\sqrt n}\}.
\]
If $Q_{n,j}\cap E_{n,\varepsilon}\neq\varnothing$, choose
$\nu\in Q_{n,j}\cap E_{n,\varepsilon}$. The diameter bound gives
\[
 Q_{n,j}\subset B(\nu,n^{-1/2}),\qquad
 \mu(Q_{n,j})\leq\mu(B(\nu,n^{-1/2}))<e^{-\varepsilon\sqrt n}.
\]
Summing over the partition cells that meet $E_{n,\varepsilon}$, we obtain
\[
 \mu(E_{n,\varepsilon})
 \leq\sum_{j:\,Q_{n,j}\cap E_{n,\varepsilon}\neq\varnothing}
       \mu(Q_{n,j})
 \leq C_Nn^{N/2}e^{-\varepsilon\sqrt n}.
\]
In particular, $\sum_{n=1}^{\infty}\mu(E_{n,\varepsilon})<\infty$.

The set of points belonging to infinitely many $E_{n,\varepsilon}$ is
\[
 E_\varepsilon:=\bigcap_{m=1}^{\infty}\bigcup_{n=m}^{\infty}
 E_{n,\varepsilon}.
\]
For every $m\geq1$, countable subadditivity gives
\[
 \mu(E_\varepsilon)\leq\sum_{n=m}^{\infty}\mu(E_{n,\varepsilon})
 \rightarrow0\qquad(m\to\infty),
\]
so $\mu(E_\varepsilon)=0$. Set $E:=\bigcup_{k=1}^{\infty}E_{1/k}$.
Then $\mu(E)=0$. If $\nu\notin E$, for every $k\geq1$ there is
$n_0=n_0(\nu,k)$ such that $\nu\notin E_{n,1/k}$ for all $n\geq n_0$.
Using also $\mu(B(\nu,n^{-1/2}))\leq M$, we obtain
\[
 -\frac1k
 \leq\frac{\ln\mu(B(\nu,n^{-1/2}))}{\sqrt n}
 \leq\frac{\ln M}{\sqrt n}
 \qquad\text{for }n\geq n_0.
\]
Letting first $n\to\infty$ and then $k\to\infty$ proves
\eqref{eq:critical-local-mass} for every $\nu\notin E$.

Fix $a>0$ and $t\in\R$. For fixed \(\nu\in\mathbb S^{N-1}\setminus E\), fixed $x,z\in\mathbb R^N$, and
$y_n=x+2\lambda_*(t+n)\nu+\sqrt{t+n}\,z$, one has
\begin{equation}\label{eq:critical-log-asymptotic}
 \frac{\ln(aH_{\mu}(-n,y_n))}{\sqrt{t+n}}
 \rightarrow\lambda_*\nu\cdot z
 \qquad\hbox{as }n\to\infty.
\end{equation}
For the upper bound, $H_{\mu}(-n,y_n)\leq M e^{-2f'(0)n+\lambda_*|y_n|}$ and
$|y_n|=2\lambda_*(t+n)+\sqrt{t+n}\,\nu\cdot z+O_{x,z}(1)$.
For the lower bound, restrict the  integral defining $H_{\mu}$ to
$|\zeta-\nu|<n^{-1/2}$. For all sufficiently large $n$, one has
$0<t+n\leq2n$, and hence
\begin{align*}
 2f'(0)(t+n)\nu\cdot\zeta
 &=2f'(0)(t+n)-f'(0)(t+n)|\nu-\zeta|^2
   \geq2f'(0)(t+n)-2f'(0),\\
 \lambda_*\sqrt{t+n}\,\zeta\cdot z
 &\geq\lambda_*\sqrt{t+n}\,\nu\cdot z-\sqrt2\lambda_*|z|,
 \qquad \lambda_*\zeta\cdot x\geq-\lambda_*|x|.
\end{align*}
This gives
\[
 H_{\mu}(-n,y_n)\geq\mu(B(\nu,n^{-1/2}))
 e^{2f'(0)t+\lambda_*\sqrt{t+n}\,\nu\cdot z-2f'(0)-\sqrt2\lambda_*|z|-\lambda_*|x|},
\]
which, with \eqref{eq:critical-local-mass} and $(t+n)/n\to1$, proves
\eqref{eq:critical-log-asymptotic}.

Set $q_n(y)=\min\{1,1/(aH_{\mu}(-n,y))\}$. Then, \(\min\{aH_{\mu}(-n,y),1\}=aH_{\mu}(-n,y)q_n(y).\)
A direct calculation gives
\begin{align*}
 \bigl(L_{t+n}[\min\{aH_{\mu}(-n,\cdot),1\}]\bigr)(x)
 &=ae^{2f'(0)t}\int_{\mathbb S^{N-1}}e^{\lambda_*\nu\cdot x}\\
 &\quad\times\left(\int_{\mathbb R^N}(4\pi)^{-N/2}e^{-|z|^2/4}
 q_n(y_n)\,dz\right)d\mu(\nu).
\end{align*}
By \eqref{eq:critical-log-asymptotic}, \(q_n(y_n)\) tends to $1$ if
$\nu\cdot z<0$ and to $0$ if $\nu\cdot z>0$. The hyperplane
$\nu\cdot z=0$ has Lebesgue measure zero. Since $0\leq q_n\leq1$,
dominated convergence gives, for $\mu$-almost every $\nu$,
\begin{align*}
 I_n(\nu)&:=\int_{\mathbb R^N}(4\pi)^{-N/2}e^{-|z|^2/4}
 q_n(y_n)\,dz\\
 &\rightarrow\int_{\{z:\,\nu\cdot z<0\}}
 (4\pi)^{-N/2}e^{-|z|^2/4}\,dz=\frac12.
\end{align*}
Moreover,
\[
 0\leq e^{\lambda_*\nu\cdot x}I_n(\nu)
 \leq e^{\lambda_*|x|},\qquad \mu(\mathbb S^{N-1})<\infty.
\]
Applying dominated convergence to the outer integral therefore yields
\[
 \bigl(L_{t+n}[\min\{aH_{\mu}(-n,\cdot),1\}]\bigr)(x)
 \rightarrow\frac a2e^{2f'(0)t}
 \int_{\mathbb S^{N-1}}e^{\lambda_*\nu\cdot x}\,d\mu(\nu)
 =\frac a2H_{\mu}(t,x).
\]
as $n\rightarrow +\infty$.
\end{proof}

The preceding Gaussian estimate rules out a nonzero nonlinear entire
solution bounded above by a critical exponential field.

\begin{proposition}\label{prop:critical-field-vanishes}
If an entire solution $0\leq V\leq1$ of \eqref{eq:kpp} satisfies $V(t,x)\leq H_{\mu}(t,x)$ with $H_{\mu}$ defined by \eqref{def-h}, then $V\equiv0$.
\end{proposition}
\begin{proof}
More generally, suppose $V\leq aH_{\mu}$ for a constant $a>0$. Remember that $\partial_t H_{\mu}-\Delta H_{\mu}=f'(0)H_{\mu}$ and $aH_{\mu}$ is a supersolution of \eqref{eq:kpp}.
For fixed $(t,x)$ and $n>\max\{0,-t\}$, the comparison principle gives
\[
 V(t,x)\leq \bigl(L_{t+n}[\min\{aH_{\mu}(-n,\cdot),1\}]\bigr)(x).
\]
Lemma~\ref{lem:critical-gaussian-half} gives
$V(t,x)\leq (a/2)H_{\mu}(t,x)$.
Iteration forces $V(t,x)=0$.
\end{proof}

Consequently, evolution from capped critical exponential data
converges to zero as the initial time tends to the remote past.

\begin{corollary}\label{cor:critical-field-pullback}
For every $a>0$, it holds
\begin{equation}\label{eq:general-critical-extinction}
 S_{t+T}[\min\{aH_{\mu}(-T,\cdot),1\}]\rightarrow0
 \qquad\hbox{in $C^{1,2}_{\mathrm{loc}}(\R\times\R^N)$ as }T\to\infty.
\end{equation}
\end{corollary}
\begin{proof}
Since $aH_{\mu}$ solves the linearized equation and
$f'(0)aH_{\mu}-f(aH_{\mu})\geq0$ on $\{aH_{\mu}<1\}$, the function
$\min\{aH_{\mu},1\}$ is a supersolution. The decreasing limit of $S_{t+T}[\min\{aH_{\mu}(-T,\cdot),1\}]$ exists and is an entire solution of \eqref{eq:kpp} which is less than
$aH_{\mu}$. Then, the limit is $0$ by Proposition~\ref{prop:critical-field-vanishes}.
\end{proof}

We next recall the solutions generated by finite families of traveling fronts
and establish their uniqueness between the corresponding barriers.

Write
\[
 X:=\bigl(\Sph\times[c_*,+\infty)\bigr)\cup\{\infty\},\qquad
 \lambda(c):=\frac{c-\sqrt{c^2-4f'(0)}}2\quad\text{for }c\geq c_*.
\]
The point $\infty$ represents the spatially homogeneous solution.
Let
\[
 \vartheta=\sum_{i=1}^{m}a_i\delta_{(\nu_i,c_i)}+a_0\delta_\infty
\]
be a finite positive atomic measure on $X$, where $m\geq0$,
$a_0\geq0$, and $a_i>0$, $\nu_i\in\Sph$, $c_i\geq c_*$ for
$i=1,\cdots,m$. Assume that the pairs $(\nu_i,c_i)$ are distinct.
Let $h_i$ be the phases corresponding to the weights $a_i$ in the
Hamel--Nadirashvili parameterization identified in
the proof of Theorem~\ref{thm:main-closure}, and set
\[
 v_i(t,x)=\phi_{c_i}(x\cdot\nu_i+c_it+h_i),
 \qquad i=1,\cdots,m.
\]
If $a_0>0$, also include the homogeneous solution
$v_0(t,x)=\xi_f(t+h_0)$, where $h_0$ corresponds to $a_0$.
All maxima and sums below range over $i=1,\cdots,m$, with the index
$0$ included precisely when $a_0>0$.
Distinct speeds in the same direction are allowed.
The maximum of the components is a bounded subsolution.
Concavity and $f(0)=0$ make their sum a supersolution wherever it
is below one; hence its minimum with one is a bounded
supersolution. The pullbacks from the maximum therefore increase,
remain below the capped sum, and have an entire limit $u_\vartheta$
satisfying
\[
 \max_i v_i\leq u_\vartheta\leq\min\{\textstyle\sum_i v_i,1\}.
\]
This is the finite-front construction of
\cite[Theorem~1.1 and Section~2]{HN2001}, namely
\begin{equation}\label{eq:transfer-packet-pullback}
 u_\vartheta(t,x)
 =\lim_{T\to\infty}\bigl(S_{t+T}[\max_i v_i(-T,\cdot)]\bigr)(x).
\end{equation}
We record the uniqueness property of this solution.

\begin{theorem}
\label{thm:packet}\label{lem:packet}
Any entire solution $U$ satisfying
\[
 \max_i v_i\leq U\leq\min\{\textstyle\sum_i v_i,1\}
\]
coincides with
$u_\vartheta$.
\end{theorem}

\begin{proof}
Let $U,\,V$ be two entire solutions satisfying these bounds.
Fix $(t,x)$ and choose $k$ such that $v_k(t,x)=\max_i v_i(t,x)$.
Suppressing the variables, we have $v_k\leq\sum_i v_i$ and $v_k\leq1$.
Moreover, $0\leq v_i\leq v_k$ implies $v_i\leq\sqrt{v_iv_k}$.
Consequently,
\[
 0\leq\min\{\textstyle\sum_i v_i,1\}-\max_i v_i
 \leq\sum_{i\ne k}v_i\leq\sum_{i\ne k}\sqrt{v_iv_k}
 \leq\sum_{i<j}\sqrt{v_iv_j}.
\]
Estimate \eqref{eq:general-cauchy-stability} gives,
for $T>-t$,
\[
 |U(t,x)-V(t,x)|
 \leq\sum_{i<j}\bigl(L_{T+t}[\sqrt{v_i(-T,\cdot)v_j(-T,\cdot)}]\bigr)(x).
\]
The normalized tails established in \eqref{eq:intro-tail}--\eqref{eq:intro-wave-slope},
together with $0<\phi_c<1$, give
\[
 \phi_c(s)\leq C_c e^{\lambda(c)s}\quad\text{for }c>c_*,
 \qquad
 \phi_*(s)\leq C(1+|s|)e^{\lambda_*s},
 \qquad\text{for }s\in\R.
\]
Substituting $s=y\cdot\nu_i-c_iT+h_i$ and using
$1+|s|\leq1+|y|+c_iT+|h_i|$, we obtain, for $T\geq1$ and $y\in\R^N$,
\[
 v_i(-T,y)=\phi_{c_i}(y\cdot\nu_i-c_iT+h_i)\leq K_i(1+T+|y|)
 e^{\lambda(c_i)y\cdot\nu_i-c_i\lambda(c_i)T}.
\]
Here $K_i$ absorbs the fixed phase factor $e^{\lambda(c_i)h_i}$
and is independent of $T$ and $y$. The polynomial factor is only needed for critical fronts.
If the homogeneous solution is present, its normalized tail
\eqref{eq:homogeneous-normalization} gives
$v_0(-T,y)=\xi_f(-T+h_0)\leq K_0e^{-f'(0)T}$.
Using $c_i\lambda(c_i)=f'(0)+\lambda(c_i)^2$ and the definition of $L_{t}$, we obtain, for $1\leq i<j\leq m$,
every fixed $R>0$, and all sufficiently large $T$,
\[
 \sup_{(t,x)\in[-R,R]\times\overline B_R}
 \bigl(L_{T+t}[\sqrt{v_i(-T,\cdot)v_j(-T,\cdot)}]\bigr)(x)\leq K_{i,j,R}(1+T)
 e^{-|\lambda(c_i)\nu_i-\lambda(c_j)\nu_j|^2T/4}.
\]
For a pair involving the homogeneous solution, the same computation gives
\[
 \sup_{(t,x)\in[-R,R]\times\overline B_R}
 \bigl(L_{T+t}[\sqrt{v_0(-T,\cdot)v_i(-T,\cdot)}]\bigr)(x)
 \leq K_{0,i,R}(1+T)e^{-\lambda(c_i)^2T/4}.
\]
Since $\lambda(c)$ is positive and strictly decreasing for $c\geq c_*$,
the distinct pairs $(\nu_i,c_i)$ give distinct vectors
$\lambda(c_i)\nu_i$. Letting $T\to\infty$ in above inequalities gives $U=V$.
\end{proof}

We now prove approximation by solutions associated with finite atomic measures
under a finiteness condition on the linear evolution from the remote past.

\begin{theorem}\label{thm:general-finite-lift}
Let $U$ be a classical entire solution of \eqref{eq:kpp} with
$0<U<1$ on $\R\times\R^N$, and suppose that
\begin{equation}\label{eq:general-finite-lift-hypothesis}
 \sup_{T>0}\bigl(L_T[U(-T,\cdot)]\bigr)(0)<\infty.
\end{equation}
Then there exists a sequence of finite positive measures on $X$ of the form
\begin{equation}\label{eq:finite-hn-atomic-measures}
 \vartheta_k
 =\sum_{i=1}^{m_k}a_{k,i}\delta_{(\nu_{k,i},c_{k,i})}
   +a_{k,0}\delta_\infty,
\end{equation}
where $m_k\geq0$ is an integer, $a_{k,0}\geq0$, and
\[
 a_{k,i}>0,\qquad \nu_{k,i}\in\Sph,\qquad c_{k,i}\geq c_*
 \qquad\text{for }i=1,\cdots,m_k,
\]
with the pairs $(\nu_{k,i},c_{k,i})$ pairwise distinct for each $k$,
such that 
\[
 u_{\vartheta_k}\rightarrow U
 \qquad\text{in }C^{1,2}_{\mathrm{loc}}(\R\times\R^N)
 \quad\text{as }k\to\infty.
\]
\end{theorem}
\begin{proof}
We give the construction and each limiting step.

\emph{Step 1: the associated linear solution and exact reconstruction.}
Put $H_T(t,\cdot)=L_{t+T}[U(-T,\cdot)]$. The equation gives
\[
 \partial_TH_T=L_{t+T}[(f'(0)U-f(U))(-T,\cdot)]\geq0,
 \qquad
 H_T-U=\int_{-T}^tL_{t-s}[(f'(0)U-f(U))(s,\cdot)]\,ds.
\]
To obtain a local bound for $H_T$, Theorem~\ref{thm:gkpp-ancient}
first gives
$U(-T,y+x)\leq e^{\lambda_*|x|}U(-T,y)$.
By positivity and translation invariance of the Gaussian convolution,
\begin{equation}\label{H_T-spatial}
 H_T(t,x)=\bigl(L_{t+T}[U(-T,\cdot+x)]\bigr)(0)\leq e^{\lambda_*|x|}\bigl(L_{t+T}[U(-T,\cdot)]\bigr)(0)
 =e^{\lambda_*|x|}H_T(t,0).
\end{equation}
Next, the estimate $0<\partial_t\ln U\leq2f'(0)$ gives
\[
 U(-T,y)\leq e^{2f'(0)|t|}U(-T-t,y).
\]
Applying the same positive operator $L_{T+t}$ to both sides yields
\[
 H_T(t,0)\leq e^{2f'(0)|t|}\bigl(L_{T+t}[U(-T-t,\cdot)]\bigr)(0)=e^{2f'(0)|t|}H_{T+t}(0,0).
\]
Thus, for every $R>0$ and $T>R$,
\[
 \sup_{(t,x)\in[-R,R]\times\overline B_R}H_T(t,x)
 \leq e^{(\lambda_*+2f'(0))R}
 \sup_{S>0}\bigl(L_S[U(-S,\cdot)]\bigr)(0)<\infty
\]
by \eqref{eq:general-finite-lift-hypothesis}. This bound is independent
of $T$. Linear interior estimates and
monotonicity give $H_T\to H$ in
$C^\infty_{\mathrm{loc}}(\R\times\R^N)$, with $H$ finite.
Applying the positive heat representation underlying
\eqref{eq:widder-input} to the positive heat solution
$e^{-f'(0)t}H(t,x)$, and scaling its spectral variable by $\lambda_*$, we obtain
a finite positive Borel measure $\mu$ on the dimensionless spectral space $\R^N$ such that
\[
 H(t,x)=\int_{\R^N}
 e^{\lambda_*x\cdot\xi+f'(0)(1+|\xi|^2)t}\,d\mu(\xi),
 \qquad \mu(\R^N)=H(0,0)<\infty.
\]
We next show that $\mu$ is supported in $\overline B_1$.
Passing to the limit in \eqref{H_T-spatial} gives
$H(0,se)\leq e^{\lambda_*s}H(0,0)$ for $s\geq0$ and every unit vector $e$.
For any $\varepsilon>0$, the integral representation therefore yields
\[
 e^{\lambda_*s(1+\varepsilon)}
 \mu\bigl(\{\xi:e\cdot\xi\geq1+\varepsilon\}\bigr)
 \leq H(0,se)\leq e^{\lambda_*s}H(0,0).
\]
Dividing by $e^{\lambda_*s(1+\varepsilon)}$ and letting $s\to\infty$
gives $\mu(\{\xi:e\cdot\xi\geq1+\varepsilon\})=0$.
 Taking a countable dense set of unit vectors $e$ and positive
rational $\varepsilon$, we conclude that
$\mu(\R^N\setminus\overline B_1)=0$, and the representation reduces to
\[
 H(t,x)=\int_{\overline B_1}
 e^{\lambda_*x\cdot\xi+f'(0)(1+|\xi|^2)t}\,d\mu(\xi),
 \qquad \mu(\overline B_1)=H(0,0)<\infty.
\]
In particular, $L_{t+T}[H(-T,\cdot)]=H(t,\cdot)$ for every
$T>\max\{0,-t\}$.

Since $H_t-\Delta H-f(H)=F_1(H)\geq0$ on $\{H<1\}$ and $f(1)=0$,
$\min\{H,1\}$ is a supersolution. Hence, the following limit exists and is an entire solution
\[
 V_H(t,\cdot):=\lim_{T\to\infty}
 S_{t+T}[\min\{H(-T,\cdot),1\}],
\]
with convergence in $C^{1,2}_{\mathrm{loc}}(\R\times\R^N)$.
Since $U\leq H$ and $U\leq1$, we have
$U(-T,\cdot)\leq\min\{H(-T,\cdot),1\}$. For
$T>\max\{0,-t\}$, the comparison principle and \eqref{eq:general-cauchy-stability} therefore give
\begin{align*}
 0\leq S_{t+T}[\min\{H(-T,\cdot),1\}]-U(t,\cdot) \leq L_{t+T}[H(-T,\cdot)-U(-T,\cdot)]=H(t,\cdot)-H_T(t,\cdot).
\end{align*}
Letting $T\to\infty$ and using $H_T\to H$ yields $V_H=U$.

\emph{Step 2: reconstruction inside the critical sphere and exclusion of boundary mass.}
For a finite positive Borel measure $\eta$ supported in $\overline B_r$,
where $0\leq r<1$, define
\begin{equation}\label{def-h-B}
 h(t,x):=\int_{\overline B_r}
 e^{\lambda_*x\cdot\xi+f'(0)(1+|\xi|^2)t}\,d\eta(\xi),
 \qquad (t,x)\in\R\times\R^N.
\end{equation}
Choose
\begin{equation}\label{eq:power-barrier-constants}
 1<p\leq2,\quad pr^2<1,\quad
 c=f'(0)(p-1)(1-pr^2)>0,\quad A\geq\max\{1,M/c\}.
\end{equation}
Then $(\partial_t-\Delta-f'(0))h^p\geq c h^p$ and, for
$T>\max\{0,-t\}$,
\begin{equation}\label{eq:general-power-decay}
 L_{t+T}[h(-T,\cdot)^p]\leq e^{-c(t+T)}h(t,\cdot)^p.
\end{equation}
To prove \eqref{eq:general-power-decay}, the case $\eta=0$ is immediate.
Otherwise $h>0$. Fix $(t,x)$ and $T>\max\{0,-t\}$, and define the
measure
\[
 d\pi_{t,x}(\xi)
 :=\frac{e^{\lambda_*x\cdot\xi+f'(0)(1+|\xi|^2)t}}{h(t,x)}\,d\eta(\xi),
 \qquad \pi_{t,x}(\overline B_r)=1.
\]
Since $p>1$, Jensen's inequality gives, for every $y\in\R^N$,
\begin{align*}
 \left(\frac{h(-T,y)}{h(t,x)}\right)^p
 &=\left(\int_{\overline B_r}
 e^{\lambda_*\xi\cdot(y-x)-f'(0)(1+|\xi|^2)(t+T)}
 \,d\pi_{t,x}(\xi)\right)^p\\
 &\leq\int_{\overline B_r}
 e^{p\lambda_*\xi\cdot(y-x)-pf'(0)(1+|\xi|^2)(t+T)}
 \,d\pi_{t,x}(\xi).
\end{align*}
Notice that
\[
 \frac1{(4\pi(t+T))^{N/2}}\int_{\R^N}
 e^{-\frac{|y-x|^2}{4(t+T)}+p\lambda_*\xi\cdot(y-x)}\,dy
 =e^{p^2f'(0)|\xi|^2(t+T)}.
\]
Thus, using $L_{t+T}=e^{f'(0)(t+T)}P_{t+T}$ and Tonelli's theorem, we obtain
\begin{align*}
 \frac{\bigl(L_{t+T}[h(-T,\cdot)^p]\bigr)(x)}{h(t,x)^p}
 &\leq\int_{\overline B_r}
 e^{f'(0)[1-p(1+|\xi|^2)+p^2|\xi|^2](t+T)}
 \,d\pi_{t,x}(\xi)\\
 &=\int_{\overline B_r}
 e^{-f'(0)(p-1)(1-p|\xi|^2)(t+T)}\,d\pi_{t,x}(\xi)\\
 &\leq e^{-c(t+T)}.
\end{align*}
This proves
\eqref{eq:general-power-decay}.

Define
\[
 \omega:=\max\{h-Ah^p,0\}.
\]
On $\{\omega>0\}$, we have $\omega=h-Ah^p$, $Ah^{p-1}<1$, and hence
$0<\omega\leq h<1$. Since $1<p\leq2$ and $Ac\geq M$, on this set
\[
 \omega_t-\Delta\omega-f(\omega)
 =-A(\partial_t-\Delta-f'(0))h^p+F_1(\omega)
 \leq-Ac h^p+M h^2\leq(M-Ac)h^p\leq0.
\]
Also $h_t-\Delta h-f(h)=F_1(h)\geq0$ on $\{h<1\}$. Thus, since $f(0)=f(1)=0$,
$\omega$ and $\min\{h,1\}$ are ordered 
subsolutions and supersolutions, respectively. The comparison principle gives, for
$T_2>T_1>\max\{0,-t\}$,
\begin{align*}
 \omega(t,x)
 &\leq \bigl(S_{t+T_2}[\min\{h(-T_2,\cdot),1\}]\bigr)(x)\\
 &\leq \bigl(S_{t+T_1}[\min\{h(-T_1,\cdot),1\}]\bigr)(x)
 \leq\min\{h(t,x),1\}.
\end{align*}
Monotonicity and parabolic estimates therefore yield an entire solution
\begin{equation}\label{eq:subcritical-capped-pullback}
 V_h(t,\cdot):=\lim_{T\to\infty}
 S_{t+T}[\min\{h(-T,\cdot),1\}],
 \qquad\text{in }C^{1,2}_{\mathrm{loc}}(\R\times\R^N),
\end{equation}
with $\omega\leq V_h\leq\min\{h,1\}$. In particular,
\[
 0\leq h-V_h\leq h-\omega=\min\{h,Ah^p\}\leq Ah^p.
\]

For uniqueness, let $W_1,W_2\in\mathcal E$ satisfy
$0\leq h-W_i\leq A_i h^p$ for some constants $A_i>0$.
For $T>\max\{0,-t\}$, \eqref{eq:general-cauchy-stability} and
\eqref{eq:general-power-decay} give
\begin{align*}
 |W_1(t,x)-W_2(t,x)|
 &\leq \bigl(L_{t+T}[|W_1(-T,\cdot)-W_2(-T,\cdot)|]\bigr)(x)\\
 &\leq(A_1+A_2)\bigl(L_{t+T}[h(-T,\cdot)^p]\bigr)(x)\\
 &\leq(A_1+A_2)e^{-c(t+T)}h(t,x)^p\rightarrow0,
 \qquad \hbox{as }T\to\infty.
\end{align*}
Hence $W_1=W_2$.

For continuity, let $\eta_k\rightharpoonup\eta$ be finite positive
measures on the same $\overline B_r$, and let $h_k,\,h$ be defined as \eqref{def-h-B}. Compactness of $\overline B_r$ gives
\[
 h_k\rightarrow h\quad\text{in }C^\infty_{\mathrm{loc}}(\R\times\R^N).
\]
The constants $p,\,c,\,A$ above are independent of $k$, so
\[
 0\leq V_{h_k}\leq1,\qquad 0\leq h_k-V_{h_k}\leq A h_k^p.
\]
Parabolic estimates make $(V_{h_k})$ precompact in
$C^{1,2}_{\mathrm{loc}}(\R\times\R^N)$. Every subsequential limit
$W\in\mathcal E$ satisfies $0\leq h-W\leq Ah^p$, so uniqueness
gives $W=V_h$. Consequently,
\[
 V_{h_k}\rightarrow V_h\quad\text{in }
 C^{1,2}_{\mathrm{loc}}(\R\times\R^N).
\]
If $\eta=0$, then $h=0$ and the same bound gives $W=V_h=0$.

Split $H=H_{\mathrm{int}}+H_{\mathrm{bd}}$ according to $|\xi|<1$ and $|\xi|=1$.
The sphere component is
\[
 H_{\mathrm{bd}}(t,x)=\int_{\partial B_1}
 e^{\lambda_*x\cdot\xi+2f'(0)t}\,d\mu(\xi).
\]
Since $\mu(\partial B_1)<\infty$, Proposition~\ref{prop:critical-field-vanishes}
and \eqref{eq:general-critical-extinction} give
\[
 V_{H_{\mathrm{bd}}}(t,\cdot)=\lim_{T\to\infty}
 S_{t+T}[\min\{H_{\mathrm{bd}}(-T,\cdot),1\}]=0,
\]
with convergence in $C^{1,2}_{\mathrm{loc}}(\R\times\R^N)$.
Order and Lemma~\ref{lem:pde-semigroup-inequalities} now give
$V_H=V_{H_{\mathrm{int}}}$.
Since $U(-T,\cdot)\leq H_{\mathrm{int}}(-T,\cdot)$, positivity of
$L_{t+T}$ gives, for $T>\max\{0,-t\}$,
\[
 H_T(t,\cdot)=L_{t+T}[U(-T,\cdot)]
 \leq L_{t+T}[H_{\mathrm{int}}(-T,\cdot)]
 =H_{\mathrm{int}}(t,\cdot).
\]
Letting $T\to\infty$ yields $H\leq H_{\mathrm{int}}$.
Thus $H_{\mathrm{bd}}=0$ and $\mu(\partial B_1)=0$.

\emph{Step 3: finite atomic approximation and identification with HN solutions.}
For $0<r<1$, set
\[
 H_r(t,x):=\int_{\overline B_r}
 e^{\lambda_*x\cdot\xi+f'(0)(1+|\xi|^2)t}\,d\mu(\xi).
\]
Since $H_r\leq H$, the comparison principle gives
$V_{H_r}\leq V_H=U$. For each fixed $r$, the construction above provides
constants $p_r>1$, $A_r>0$, and $c_r>0$ such that
$0\leq H_r-V_{H_r}\leq A_rH_r^{p_r}$ and
\eqref{eq:general-power-decay} holds with these constants.
Combining this bound with $U\leq H$, we obtain, for
$T>\max\{0,-t\}$,
\[
 \begin{aligned}
 0\leq U(-T,\cdot)-V_{H_r}(-T,\cdot)
 &\leq H(-T,\cdot)-V_{H_r}(-T,\cdot)\\
 &=H(-T,\cdot)-H_r(-T,\cdot)
   +H_r(-T,\cdot)-V_{H_r}(-T,\cdot)\\
 &\leq H(-T,\cdot)-H_r(-T,\cdot)
   +A_rH_r(-T,\cdot)^{p_r}.
 \end{aligned}
\]
Applying \eqref{eq:general-cauchy-stability} from time $-T$ to $t$, it follows from \eqref{eq:general-power-decay} that
\[
 \begin{aligned}
 0\leq U(t,x)-V_{H_r}(t,x)
 &\leq \bigl(L_{t+T}[U(-T,\cdot)-V_{H_r}(-T,\cdot)]\bigr)(x)\\
 &\leq H(t,x)-H_r(t,x)
   +A_r e^{-c_r(t+T)}H_r(t,x)^{p_r}.
 \end{aligned}
\]
First let $T\to\infty$ at fixed $r$ to obtain
$0\leq U-V_{H_r}\leq H-H_r$. For every $R>0$,
\[
 \begin{aligned}
 \sup_{(t,x)\in[-R,R]\times\overline B_R}|H(t,x)-H_r(t,x)|
 &=\sup_{(t,x)\in[-R,R]\times\overline B_R}
 \int_{\{r<|\xi|\leq1\}}
 e^{\lambda_*x\cdot\xi+f'(0)(1+|\xi|^2)t}\,d\mu(\xi)\\
 &\leq e^{(\lambda_*+2f'(0))R}
 \mu\bigl(\{r<|\xi|\leq1\}\bigr).
 \end{aligned}
\]
Since $\mu$ is finite, continuity from above gives
$\mu(\{r<|\xi|\leq1\})\rightarrow\mu(\partial B_1)=0$ as $r\nearrow1$.
Thus $H_r\rightarrow H$ locally uniformly in $\R\times\R^N$.
Parabolic estimates then give $V_{H_r}\to U$ in
$C^{1,2}_{\mathrm{loc}}(\R\times\R^N)$ as $r\nearrow1$.

Fix $r<1$ such that $\mu(\overline B_r)>0$. Choose finite atomic measures
\[
 \eta_n=\sum_{j=1}^{m_n}b_{n,j}\delta_{\xi_{n,j}}
 \rightharpoonup\mu|_{\overline B_r},
 \qquad b_{n,j}>0,\quad \xi_{n,j}\in\overline B_r,
\]
with distinct atoms and $\sum_j b_{n,j}=\mu(\overline B_r)$, and set
\[
 h_n(t,x)=\sum_{j=1}^{m_n}b_{n,j}
 e^{\lambda_*x\cdot\xi_{n,j}+f'(0)(1+|\xi_{n,j}|^2)t}.
\]
The continuity proved above gives
$V_{h_n}\rightarrow V_{H_r}$ in
$C^{1,2}_{\mathrm{loc}}(\R\times\R^N)$ as $n\to\infty$.

We show that each $V_{h_n}$ is of the form $u_\vartheta$ defined in
\eqref{eq:transfer-packet-pullback}, for some nonzero finite atomic
measure $\vartheta$ on $X$. For a single atom
$b\delta_\xi$, where $b>0$ and $0<|\xi|<1$, put
\[
 \nu=\frac{\xi}{|\xi|},\qquad
 c=\lambda_*\frac{1+|\xi|^2}{|\xi|}>c_*,\qquad
 z=x\cdot\nu+ct.
\]
Then $\lambda(c)=\lambda_*|\xi|$ and let
$h(t,x)=b e^{\lambda(c)z}$. Fix $a\in\R$ and $y\in\R^N$ with
$y\cdot\nu+ca=0$. Then
\[
 h(-T,x+y)
 =b e^{\lambda(c)(x\cdot\nu+y\cdot\nu-cT)}
 =h(-(T+a),x).
\]
Using spatial translation invariance of $S_\tau$ and the definition of
$V_h$, we obtain
\[
 \begin{aligned}
 V_h(t+a,x+y)
 &=\lim_{T\to\infty}\bigl(S_{t+a+T}[\min\{h(-T,\cdot+y),1\}]\bigr)(x)\\
 &=\lim_{T\to\infty}\bigl(S_{t+(T+a)}[\min\{h(-(T+a),\cdot),1\}]\bigr)(x)\\
 &=V_h(t,x).
 \end{aligned}
\]
Thus $V_h(t,x)=\psi(z)$ for some profile $\psi$.
The equation and the estimate $0\leq h-V_h\leq Ah^p$ imply
\[
 \psi''-c\psi'+f(\psi)=0,\qquad
 \psi(z)=b e^{\lambda(c)z}+O(e^{p\lambda(c)z})
 \quad\text{as }z\to-\infty.
\]
Thus $V_h$ is nonconstant and the strong maximum principle gives
$0<V_h<1$. Theorem~\ref{thm:gkpp-ancient} yields
$\psi'>0$ and $\psi(-\infty)=0$, $\psi(+\infty)=1$.
By uniqueness of the traveling fronts and the
normalization $\phi_c(z)\sim e^{\lambda(c)z}$ from
{\eqref{eq:intro-tail}--\eqref{eq:intro-wave-slope}},
\[
 V_h(t,x)=\phi_c\left(x\cdot\nu+ct+\frac{\ln b}{\lambda(c)}\right).
\]
For an atom $b\delta_0$, spatial translation invariance makes $V_h$
independent of $x$, and the same tail estimate gives
\[
 V_h(t,x)=\xi_f\left(t+\frac{\ln b}{f'(0)}\right).
\]

Now let $m\geq1$, choose $b_j>0$ and distinct
$\xi_j\in\overline B_r$ for $j=1,\cdots,m$, and set
\[
 h_j(t,x):=b_j e^{\lambda_*x\cdot\xi_j+f'(0)(1+|\xi_j|^2)t},
 \qquad h(t,x):=\sum_{j=1}^{m}h_j(t,x).
\]
For each $j$, let $v_j:=V_{h_j}$ be the entire solution constructed
from $h_j$ above. Then, $v_j$ is a
planar traveling front if $\xi_j\ne0$, and a spatially homogeneous
solution if $\xi_j=0$. Since $h_j\leq h$, the comparison principle gives $\max_{1\leq j\leq m}v_j\leq V_h$.
Choose $p,A,c$ as in \eqref{eq:power-barrier-constants}.
Since all $\xi_j$ lie in $\overline B_r$, the preceding construction
applies to every $h_j$ with the same constants and gives
$0\leq h_j-v_j\leq Ah_j^p$ for $j=1,\cdots,m$.
These constants are independent of the masses $b_j$ and the number of
atoms, and \eqref{eq:general-power-decay} also holds for $h=\sum_jh_j$
with the same $p$ and $c$. Consequently,
\[
 0\leq\min\{h,1\}-\min\{\textstyle\sum_jv_j,1\}
 \leq\sum_j(h_j-v_j)\leq A\sum_jh_j^p\leq Ah^p.
\]
Since $\min\{\sum_jv_j,1\}$ is a supersolution,
\eqref{eq:general-cauchy-stability} and \eqref{eq:general-power-decay} imply
\[
 \begin{aligned}
 \bigl(S_{t+T}[\min\{h(-T,\cdot),1\}]\bigr)(x)
 &\leq\min\{\textstyle\sum_jv_j(t,x),1\}
     +A \bigl(L_{t+T}[h(-T,\cdot)^p]\bigr)(x)\\
 &\leq\min\{\textstyle\sum_jv_j(t,x),1\}
     +Ae^{-c(t+T)}h(t,x)^p.
 \end{aligned}
\]
Letting $T\to\infty$ yields
\[
 \max_jv_j\leq V_h\leq\min\{\textstyle\sum_jv_j,1\}.
\]
Let $\vartheta$ be the finite positive atomic measure on $X$ corresponding
to $v_1,\cdots,v_m$, and let $u_\vartheta$ be defined by
\eqref{eq:transfer-packet-pullback}. Both $V_h$ and $u_\vartheta$ lie
between $\max_jv_j$ and $\min\{\sum_jv_j,1\}$. The uniqueness in
Theorem~\ref{thm:packet} therefore gives
\[
 V_h(t,x)=u_\vartheta(t,x)
 \qquad\text{for all }(t,x)\in\R\times\R^N.
\]
In particular, for each fixed $r$ and $n$, write
$V_{h_n}=u_{\vartheta_{r,n}}$, where $\vartheta_{r,n}$ is the corresponding
finite positive atomic measure on $X$.

Finally, choose $r_k\nearrow1$ with $\mu(\overline B_{r_k})>0$ such that
\[
 \|V_{H_{r_k}}-U\|_{C^{1,2}([-k,k]\times\overline B_k)}<\frac1{2k}.
\]
For each $k$, choose $n_k$ sufficiently large that
\[
 \|u_{\vartheta_{r_k,n_k}}-V_{H_{r_k}}\|_{C^{1,2}([-k,k]\times\overline B_k)}
 <\frac1{2k}.
\]
Setting $\vartheta_k=\vartheta_{r_k,n_k}$, we obtain
\[
 \|u_{\vartheta_k}-U\|_{C^{1,2}([-k,k]\times\overline B_k)}<\frac1k,
\]
which proves the asserted convergence.
\end{proof}

\subsection{Small-amplitude limits under a strict gradient bound}
\label{subsec:wave-amplitude-limits}
\label{subsec:small-amplitude-limits}

We first identify the action of dilation on traveling fronts and
the homogeneous solution and derive a scalar small-amplitude criterion.
We then use it to prove that a strict logarithmic gradient bound
gives a finite limit of $D_aU/a$, equal to the linear evolution limit
from the remote past. This verifies the finiteness condition required
in Subsection~\ref{subsec:linear-reconstruction}.

Throughout this subsection, $U$ is an entire solution with $0<U<1$.
We use the amplitude regularity and derivative convergence established
in Lemma~\ref{lem:amplitude-C1-degenerate}, and retain $F_1$ and $M$
from \eqref{eq:general-quadratic-error}.
The critical profile $\phi_*$ has the normalization
\eqref{eq:intro-critical-expansion}:
\[
 \phi_*(z)=e^{\lambda_*z}\bigl(-z+O(1)\bigr)\quad\text{as }z\to-\infty.
\]
Write $q(u)=\phi_*'(\phi_*^{-1}(u))$ for $0<u<1$. In particular,
$0<q(u)<\lambda_*u$ and $q(u)/u\to\lambda_*$ as $u\searrow0$.

The following lemma records the phase shifts produced by dilation
and the critical-profile inequalities used in the measure construction.

\begin{lemma}\label{lem:wave-tail-action}
For every $c\geq c_*$, the traveling front $\phi_c$ satisfies
\[
 0<\frac{\phi_c'(z)}{\phi_c(z)}<\lambda(c)
 \qquad\text{for }z\in\R,
 \qquad
 \frac{\phi_c'(z)}{\phi_c(z)}\rightarrow\lambda(c)
 \quad\text{as }z\rightarrow-\infty.
\]
For every $a>0$ and $\nu\in\Sph$, it holds
\begin{equation}\label{eq:wave-amplitude-action}
 D_a[\phi_c(x\cdot\nu+ct)]
 =\phi_c(x\cdot\nu+ct+\lambda(c)^{-1}\ln a).
\end{equation}
The critical profile $\phi_*$ additionally satisfies
$\phi_*''-\lambda_*\phi_*'\leq0$ on $\R$.
Consequently, the function defined by
$g(q)=\phi_*(\lambda_*^{-1}\ln q)$ for $q>0$ and $g(0)=0$
is increasing and concave on $[0,+\infty)$.
 For the spatially homogeneous orbit $\xi_f$, one similarly has
$D_a\xi_f(t)=\xi_f(t+f'(0)^{-1}\ln a)$.
\end{lemma}
\begin{proof}
The bounds on $\phi_c'/\phi_c$ and its limit at $-\infty$ follow from \eqref{eq:intro-tail}--\eqref{eq:intro-wave-slope}. Integrating the logarithmic derivative gives
\begin{equation}\label{eq:wave-fixed-shift}
 \frac{\ln\phi_c(z)}z\to\lambda(c),
 \qquad \frac{\phi_c(z+b)}{\phi_c(z)}\to e^{\lambda(c) b}
 \qquad\text{as }z\rightarrow-\infty,
\end{equation}
for every fixed $b$.

For $0<a<1$, we have $\lambda(c)^{-1}\ln a<0$. Integrating
$\phi_c'/\phi_c<\lambda(c)$ from $z+\lambda(c)^{-1}\ln a$ to $z$ gives
\[
 \ln\frac{\phi_c(z)}{\phi_c(z+\lambda(c)^{-1}\ln a)}
 =\int_{z+\lambda(c)^{-1}\ln a}^{z}
 \frac{\phi_c'(s)}{\phi_c(s)}\,ds
 <-\ln a.
\]
Thus $a\phi_c(z)<\phi_c(z+\lambda(c)^{-1}\ln a)$.
Comparison and \eqref{eq:cone-basic-bounds} yield
\[
 a\phi_c(z)\leq D_a\phi_c(z)\leq\phi_c(z+\lambda(c)^{-1}\ln a).
\]
Translation invariance of $D_a$ and the preceding bounds show that
$\psi:=D_a\phi_c$ satisfies
\[
 \psi''-c\psi'+f(\psi)=0,\qquad 0<\psi<1.
\]
Since $c>0$, applying Theorem~\ref{thm:gkpp-ancient} to the entire solution
$\psi(x\cdot\nu+ct)$ gives
\[
 \psi'>0,\qquad \lim_{z\to-\infty}\psi(z)=0,\qquad
 \lim_{z\to+\infty}\psi(z)=1.
\]
Uniqueness of traveling fronts therefore gives $\psi(z)=\phi_c(z+b)$
for some $b\in\R$. Dividing the preceding bounds by $\phi_c(z)$ and
letting $z\to-\infty$ in \eqref{eq:wave-fixed-shift}, we obtain
\[
 a\leq\lim_{z\to-\infty}\frac{\psi(z)}{\phi_c(z)}
 =e^{\lambda(c)b}
 \leq\lim_{z\to-\infty}
 \frac{\phi_c(z+\lambda(c)^{-1}\ln a)}{\phi_c(z)}=a.
\]
Thus $b=\lambda(c)^{-1}\ln a$, proving \eqref{eq:wave-amplitude-action}
for $0<a<1$. The case $a=1$ is immediate. For $a>1$, the case
$a^{-1}<1$ gives
$D_{1/a}\bigl[\phi_c(\cdot+\lambda(c)^{-1}\ln a)\bigr]=\phi_c$.
Applying $D_a$ and the group law in
Proposition~\ref{prop:pde-full-dilation}, we obtain $D_a\phi_c=\phi_c(\cdot+\lambda(c)^{-1}\ln a)$,
which proves \eqref{eq:wave-amplitude-action} for $a>1$.

Finally put $h=\phi_*''-\lambda_*\phi_*'=\lambda_*\phi_*'-f(\phi_*)$.
Differentiating
the wave equation gives
\[
 h'-\lambda_*h=(f'(0)-f'(\phi_*))\phi_*'\geq0.
\]
The profile and its first two derivatives are bounded, so
$e^{-\lambda_*z}h(z)\to0$ as $z\to+\infty$. Integration of
$(e^{-\lambda_*z}h)'\geq0$ gives $h\leq0$, hence
$\lambda_*\phi_*'\leq f(\phi_*)$.
Since $q^2g''(q)=\lambda_*^{-2}h(\lambda_*^{-1}\ln q)\le 0$,
$g$ is concave. 
For $\xi_f$, spatial invariance makes $D_a\xi_f$ another nonzero
homogeneous orbit, hence a time translate of $\xi_f$.
The inequalities
$a\xi_f(t)\leq D_a\xi_f(t)\leq\xi_f(t+f'(0)^{-1}\ln a)$ for $0<a<1$
follow from $f(u)/u\leq f'(0)$ by the same integration and comparison.
Its positive $e^{f'(0)t}$ tail determines the shift to be $f'(0)^{-1}\ln a$.
Again the group law extends the identity to all $a>0$.
\end{proof}

We next turn the amplitude differential inequality into a scalar
criterion for the limits at scales $a$ and $a\ln(1/a)$.
It will be applied to $G(a)=D_aU(t,x)$.

\begin{lemma}\label{lem:scalar-trace-degenerate}
\label{lem:small-amplitude-calculus}
 Let $G\in C^1((0,\infty);(0,1))$ be increasing and concave, and satisfy
$\displaystyle\lim_{a\searrow0}G(a)=0$ and $G-aG'\geq0$. Suppose that, for some
$0<\gamma\leq1$ and finite $M>0$,
\begin{equation}\label{eq:scalar-trace-distribution}
 \gamma a^2G''\leq aG'-G+M G^2
 \qquad\text{in distributions on }(0,\infty).
\end{equation}
If $\gamma=1$, there exists a finite nonnegative constant $A$ such that
\[
 \lim_{a\searrow0}\frac{G(a)}{a\ln(1/a)}=A.
\]
If $0<\gamma<1$, there exists a finite positive constant $A$ such that
\[
 \lim_{a\searrow0}\frac{G(a)}{a}=A.
\]
\end{lemma}
\begin{proof}
Set $a=e^{-\tau}$ and $\widetilde G(\tau)=e^\tau G(e^{-\tau})$. Then
\[
 \widetilde G'=G/a-G'\geq0,\qquad \widetilde G''=\widetilde G'+aG'',
\]
where the second identity holds in distributions. Since $G$ is
concave, $G'$ and $\widetilde G'$ are locally of bounded variation.
The hypothesis becomes
\begin{equation}\label{eq:small-amplitude-damping}
 \widetilde G''+(\gamma^{-1}-1)\widetilde G'\leq\frac M\gamma e^{-\tau}\widetilde G^2.
\end{equation}
Since $\widetilde G>0$ and $\gamma\leq1$,
\[
 \left(\frac{\widetilde G'}{\widetilde G}\right)'
 \leq\frac M\gamma G(e^{-\tau})-\left(\frac{\widetilde G'}{\widetilde G}\right)^2.
\]
Since $G(e^{-\tau})\to0$ as $\tau\to+\infty$, for every
$\varepsilon>0$ there is $\tau_\varepsilon$ such that
$(M/\gamma)G(e^{-\tau})\leq\varepsilon^2$ for all
$\tau\geq\tau_\varepsilon$. Comparison with
$\varepsilon+1/(\tau-\tau_\varepsilon)$ gives
\[
 0\leq\frac{\widetilde G'(\tau)}{\widetilde G(\tau)}
 \leq\varepsilon+\frac1{\tau-\tau_\varepsilon}
 \qquad\text{for }\tau>\tau_\varepsilon.
\]
The comparison remains valid in distributions: the singular part of
the derivative is nonpositive, and the positive integrating factor
for the difference preserves its sign. Consequently
$\widetilde G'/\widetilde G\to0$.
Choose $\tau_0$ so that $\widetilde G'/\widetilde G\leq1/4$ for $\tau\geq\tau_0$.
Integration and \eqref{eq:small-amplitude-damping} give, for some $C>0$,
\[
 \widetilde G(\tau)\leq \widetilde G(\tau_0)e^{(\tau-\tau_0)/4},\qquad
 \widetilde G''+(\gamma^{-1}-1)\widetilde G'\leq Ce^{-\tau/2}.
\]
Hence
\[
 \frac{d}{d\tau}\bigl(\widetilde G'+2Ce^{-\tau/2}\bigr)
 \leq-(\gamma^{-1}-1)\widetilde G'\leq0.
\]
The nonnegative function $\widetilde G'+2Ce^{-\tau/2}$ has a finite limit,
and therefore so does $\widetilde G'$.
If $\gamma=1$, write $A=\lim_{\tau\to\infty}\widetilde G'(\tau)\geq0$. Then
\[
 \frac{G(a)}{a\ln(1/a)}=\frac{\widetilde G(\tau)}\tau
 =\frac{\widetilde G(\tau_0)}\tau+
 \frac1\tau\int_{\tau_0}^{\tau}\widetilde G'(s)\,ds
 \rightarrow A.
\]
If $0<\gamma<1$, integration gives
\[
 (\gamma^{-1}-1)\int_{\tau_0}^{\tau}\widetilde G'(s)\,ds
 \leq \widetilde G'(\tau_0)+2Ce^{-\tau_0/2}.
\]
Thus $\widetilde G$ is positive, nondecreasing, and bounded. So, for some $A\in(0,\infty)$,
\[
 \frac{G(a)}a=\widetilde G(\tau)\rightarrow A
 \qquad\text{as }\tau\to\infty\ \text{or, equivalently, }a\searrow0.
\]
This completes the proof.
\end{proof}

We now give a sufficient condition for
\eqref{eq:general-finite-lift-hypothesis}: a strict logarithmic
gradient bound ensures that $D_aU/a$ has a finite limit, which
coincides with the linear evolution limit from the remote past.

\begin{theorem}\label{thm:gradient-gap-finite-lift}
If $|\nabla\ln U|\leq\lambda_*r$ for some $0\leq r<1$ everywhere, then
\begin{equation}\label{eq:gap-lift-equality}
 H(t,x):=\lim_{T\to\infty}\bigl(L_{t+T}[U(-T)]\bigr)(x)
 =\lim_{a\searrow0}\frac{D_aU(t,x)}a<\infty.
\end{equation}
Both limits hold in
$C^{1,2}_{\mathrm{loc}}(\mathbb R\times\mathbb R^N)$.
There is a finite positive measure $\mu$ supported in
$\overline B_r$ such that
\[
 H(t,x)=\int_{\overline B_r}
 e^{\lambda_*x\cdot\xi+f'(0)(1+|\xi|^2)t}\,\mu(d\xi).
\]
In particular, Theorem~\ref{thm:general-finite-lift} reconstructs
$U$ from this field and approximates it by finite HN solutions.
\end{theorem}
\begin{proof}
Set $\gamma=(1+r^2)/2$, so that $r^2\leq\gamma<1$ and
$\gamma>0$, including when $r=0$.
For $G(a)=D_aU(t,x)$, inequality
\eqref{eq:trace-F-inequality} and $f(G)/f'(0)\geq G-(M/f'(0))G^2$ give
\eqref{eq:scalar-trace-distribution} with this $\gamma$ and quadratic
constant $M/f'(0)$.
Lemma~\ref{lem:scalar-trace-degenerate} gives the finite positive
pointwise limit
\[
 H_0(t,x)=\lim_{a\searrow0}\frac{D_aU(t,x)}a.
\]
Set $Z_a=D_aU/a$. Applying Theorem~\ref{thm:gkpp-ancient} to
$D_aU$ and integrating its logarithmic bounds gives
\[
 e^{-2f'(0)|t|-\lambda_*|x|}Z_a(0,0)
 \leq Z_a(t,x)
 \leq e^{2f'(0)|t|+\lambda_*|x|}Z_a(0,0).
\]
Since $Z_a(0,0)\to H_0(0,0)<\infty$ as $a\to0$, the family
$Z_a$ is locally uniformly bounded for small $a$. Its equation is
\[
 (Z_a)_t-\Delta Z_a=f'(0)Z_a-\frac{F_1(aZ_a)}a,
 \qquad 0\leq\frac{F_1(aZ_a)}a\leq MaZ_a^2,
\]
interior parabolic estimates give convergence in
$C^{1,2}_{\mathrm{loc}}$ and
\[
 (H_0)_t=\Delta H_0+f'(0)H_0,\qquad
 |\nabla H_0|\leq\lambda_*rH_0.
\]
The gradient bound follows from Lemma~\ref{lem:gap-preservation}.
The representation \eqref{eq:widder-input}, together with
$|\nabla H_0|\leq\lambda_*rH_0$, gives a finite positive
representing measure supported in $\overline B_r$, by the
support argument in the proof of Theorem~\ref{thm:general-finite-lift}
with radius $1$ replaced by $r$.

For $t>-T$, let $H_T(t)=L_{t+T}[U(-T)]$. Since $F_1\geq0$,
\[
 \partial_TH_T(t)=L_{t+T}[F_1(U(-T))]\geq0.
\]
Also $D_aU\geq aU$ for $0<a\leq1$, hence $U\leq H_0$ and
\[
 H_T(t)\leq L_{t+T}[H_0(-T)]=H_0(t).
\]
The last equality follows either from the spectral representation
or from Gaussian integration. Consequently $H_T\nearrow H\leq H_0$.
Conversely, finite-time linear comparison gives
\[
 S_{t+T}[aU(-T)]\leq aL_{t+T}[U(-T)].
\]
First letting $T\to\infty$ and then $a\searrow0$ yields
$D_aU\leq aH$ and $H_0\leq H$, respectively. Thus $H=H_0$.
Local bounds and the linear equation also upgrade $H_T\to H$ to
$C^{1,2}_{\mathrm{loc}}$ convergence. This proves the finite linear
limit before applying Theorem~\ref{thm:general-finite-lift}.
\end{proof}
The strict logarithmic gradient bound above is sufficient but is not
assumed for a general entire solution. In
Section~\ref{sec:degenerate-spectral-decomposition}, we decompose
such a solution into a critical part and an increasing limit of
solutions satisfying strict logarithmic gradient bounds.
The latter are covered by the reconstruction established here.
For the critical part, Section~\ref{sec:critical-faithfulness}
uses the different normalization $D_aU/[a\ln(1/a)]$ to extract
a measure on the critical sphere. Failure of a strict gradient
bound alone does not imply that the original solution is purely critical.

\section{Translation parameters and critical--subcritical decomposition}
\label{sec:degenerate-spectral-decomposition}

Section~\ref{sec:pde-dilations} constructs the ordered vector space
$E$ and its complete principal ideals $E_U$.
Proposition~\ref{prop:cone-translation-multipliers} represents
space--time translations on $E_U$ as multiplication by
$e^{s\lambda(k)+\lambda_*h\cdot\xi(k)}$.
We now prove that these parameters satisfy
$\lambda(k)=f'(0)(1+|\xi(k)|^2)$ and use this relation to decompose
an entire solution into a critical part and an increasing limit
of parts with a strict logarithmic gradient bound.

\subsection{Identification of the translation parameters}

We show that the translation parameters satisfy
$\lambda=f'(0)(1+|\xi|^2)$ by localizing their continuous-function
representation and identifying the resulting limits as traveling
fronts or homogeneous solutions.

Fix a nonzero
$U\in\mathcal C$.  Propositions~\ref{prop:cone-principal-ideal}
and~\ref{prop:cone-translation-multipliers} give a compact Hausdorff
space $K_U$, an isometric lattice isomorphism
\[
 \mathcal J:E_U\rightarrow C(K_U),\qquad \mathcal J(U)=1,
\]
and continuous functions
\[
 \lambda:K_U\rightarrow[0,2f'(0)],\qquad
 \xi:K_U\rightarrow\overline B_1
\]
such that
\begin{equation}\label{eq:cutoff-spectral-input}
 \mathcal J(T_{s,h}Z)(k)
 =e^{s\lambda(k)+\lambda_*h\cdot\xi(k)}\mathcal J(Z)(k)
 \qquad\text{for }Z\in E_U.
\end{equation}
Only the bounds $0\leq\lambda\leq2f'(0)$ and $|\xi|\leq1$ are used
as input below.  In particular, the lower bound $\lambda\geq f'(0)$
is not assumed.

{
We now determine the relation imposed by the equation on the temporal
and spatial translation parameters.
}

\begin{theorem}\label{thm:translation-spectral-surface}
For every $k\in K_U$,
\begin{equation}\label{eq:translation-spectral-surface}
 \lambda(k)=f'(0)(1+|\xi(k)|^2).
\end{equation}
\end{theorem}

\begin{proof}
Fix $k_0\in K_U$ and write
$\lambda_0=\lambda(k_0)$ and $\xi_0=\xi(k_0)$.
For each $n\geq1$, choose $\psi_n\in C(K_U)$ such that
\[
 0\leq\psi_n\leq1,\qquad \psi_n(k_0)=1,\qquad
 \operatorname{supp}\psi_n
 \subset\{k:|\lambda(k)-\lambda_0|+|\xi(k)-\xi_0|<1/n\}.
\]
Set $Z_n=\mathcal J^{-1}\psi_n$. Since $0\leq Z_n\leq U<1$ and
$\psi_n(k_0)=1$, the strong maximum principle gives $0<Z_n<1$.
By Corollary~\ref{cor:cone-normalization}, for fixed $\theta\in(0,1)$
there is a unique $b_n>0$ such that
\[
 Y_n:=D_{b_n}Z_n,\qquad Y_n(0,0)=\theta.
\]
Moreover, $\mathcal J(Y_n)=b_n\psi_n$.
By \eqref{eq:amplitude-lipschitz}, for $0<a<b$,
\begin{equation}\label{eq:spectral-amplitude-lipschitz}
 0\leq D_bY-D_aY
 =D_{b/a}(D_aY)-D_aY
 \leq\left(\frac ba-1\right)D_aY
 \leq\frac{b-a}{a}.
\end{equation}
Interior parabolic estimates, \eqref{eq:spectral-amplitude-lipschitz},
and diagonal extraction give, along a subsequence,
\begin{equation}\label{eq:spectral-joint-compactness}
 D_aY_n(t,x)\rightarrow\mathscr A(a;t,x)
 \quad\text{locally uniformly on }
 (0,\infty)\times\mathbb R\times\mathbb R^N.
\end{equation}
For fixed $a>0$, the convergence is also in $C^{1,2}_{\loc}$
in $(t,x)$, and $\mathscr A(a;\cdot,\cdot)$ solves \eqref{eq:kpp}.
In particular,
\[
 Y:=\mathscr A(1;\cdot,\cdot),\qquad Y(0,0)=\theta,
\]
is an entire solution of \eqref{eq:kpp} satisfying $0<Y<1$.

Put $d_n=(|t|+\lambda_*|x|)/n$.  Applying
\eqref{eq:cutoff-spectral-input} on the support of $\psi_n$ and
then evaluating the corresponding order inequalities at $(0,0)$ gives
\[
 D_{a e^{\lambda_0t+\lambda_*\xi_0\cdot x-d_n}}Y_n(0,0)
 \leq D_aY_n(t,x)
 \leq D_{a e^{\lambda_0t+\lambda_*\xi_0\cdot x+d_n}}Y_n(0,0).
\]
By \eqref{eq:spectral-amplitude-lipschitz} and
\eqref{eq:spectral-joint-compactness},
\begin{equation}\label{eq:normalized-spectral-limit}
 \mathscr A(a;t,x)=F(ae^{\lambda_0t+\lambda_*\xi_0\cdot x}),
 \qquad F(a):=\mathscr A(a;0,0),\qquad F(1)=\theta.
\end{equation}
By Theorem~\ref{thm:gkpp-ancient}, $Y_t>0$, whereas
$\lambda_0=0$ would give $Y_t=0$. Thus $\lambda_0>0$, and
\[
 F(a)=Y\left(\frac{\ln a}{\lambda_0},0\right),
 \qquad F'(a)=
 \frac{Y_t((\ln a)/\lambda_0,0)}{\lambda_0a}>0,
 \qquad 0<F(a)<1.
\]
Passing to the limit in \eqref{eq:inverse-wave-upper}
and \eqref{eq:inverse-wave-lower} gives
\[
 F(a)\leq\phi_*(\lambda_*^{-1}\ln a+\phi_*^{-1}(\theta))\quad\text{for }0<a\leq1,
 \qquad
 F(a)\geq\phi_*(\lambda_*^{-1}\ln a+\phi_*^{-1}(\theta))\quad\text{for }a\geq1.
\]
In particular,
\begin{equation}\label{eq:spectral-F-endpoints}
 \lim_{a\searrow0}F(a)=0,\qquad
 \lim_{a\rightarrow\infty}F(a)=1,\qquad
 F(a)\leq C_\theta a\ln(1/a)\quad\text{for }0<a\leq a_\theta,
\end{equation}
for some $C_\theta>0$ and $a_\theta\in(0,1)$, by
\eqref{eq:intro-critical-tail}. We also have
\begin{equation}\label{eq:spectral-lower-semicontinuity}
 D_aY(0,0)\leq F(a)\qquad\text{for }0<a<1.
\end{equation}
Indeed, \eqref{eq:cone-stability} gives, for fixed $T>0$,
\[
 \left|S_T[aY_n(-T)]-S_T[aY(-T)]\right|(0)
 \leq ae^{f'(0)T}e^{T\Delta}|Y_n(-T)-Y(-T)|(0)\rightarrow0.
\]
For $0<a<1$, monotonicity in $T$ therefore gives
\[
 \begin{aligned}
 D_aY(0,0)
 &=\sup_{T>0}\lim_{n\rightarrow\infty}\bigl(S_T[aY_n(-T)]\bigr)(0)\\
 &\leq\liminf_{n\rightarrow\infty}D_aY_n(0,0)=F(a).
 \end{aligned}
\]

Suppose first that $\xi_0=0$.  Substitution of
\eqref{eq:normalized-spectral-limit} into the equation gives
\[
 \lambda_0aF'(a)=f(F(a)).
\]
Since $f(u)=f'(0)u+O(u^2)$ at zero, integration yields
\[
 F(a)=c_0a^{f'(0)/\lambda_0}(1+o(1))
 \quad\text{as }a\searrow0,\qquad c_0>0.
\]
Here $1/f(u)-1/(f'(0)u)=O(1)$ as $u\to0$.
By \eqref{eq:spectral-F-endpoints},
\[
 c_0a^{f'(0)/\lambda_0}(1+o(1))
 \leq C_\theta a\ln(1/a)
 \quad\Longrightarrow\quad \lambda_0\leq f'(0).
\]
On the other hand, the homogeneous-orbit amplitude action gives
\[
 D_aY(0,0)=Y\left(\frac{\ln a}{f'(0)},0\right)
 =F(a^{\lambda_0/f'(0)}).
\]
Using \eqref{eq:spectral-lower-semicontinuity} and $F'>0$,
\[
 F(a^{\lambda_0/f'(0)})\leq F(a)
 \ \Longrightarrow\ a^{\lambda_0/f'(0)}\leq a
 \ \Longrightarrow\ \lambda_0\geq f'(0)
 \qquad\text{for }0<a<1.
\]
Hence $\lambda_0=f'(0)$.

Suppose now that $\xi_0\ne0$.  Define
\[
 \nu_0=\frac{\xi_0}{|\xi_0|},\qquad
 c=\frac{\lambda_0}{\lambda_*|\xi_0|},\qquad
 \Psi(z)=F(e^{\lambda_*|\xi_0|z}).
\]
Then
\[
 Y(t,x)=\Psi(x\cdot\nu_0+ct),\qquad
 \Psi'>0,\qquad \Psi(-\infty)=0,\quad\Psi(+\infty)=1.
\]
Since $\Psi(z)=Y(0,z\nu_0)$, Lemma~\ref{lem:wave-tail-action}
gives $c\geq c_*$ and, with
\[
 r=\frac{c-\sqrt{c^2-4f'(0)}}2\in(0,\lambda_*],
\]
$\Psi'(z)/\Psi(z)\to r$ as $z\to-\infty$. Consequently,
\[
 \lim_{a\searrow0}\frac{\ln F(a)}{\ln a}
 =\frac r{\lambda_*|\xi_0|}.
\]
The upper bound in \eqref{eq:spectral-F-endpoints} implies
$r\geq\lambda_*|\xi_0|$.  Conversely, the exact wave amplitude action and
\eqref{eq:spectral-lower-semicontinuity} give
\[
 \Psi(r^{-1}\ln a)
 =D_aY(0,0)
 \leq F(a)=\Psi((\lambda_*|\xi_0|)^{-1}\ln a)
 \qquad\text{for }0<a<1.
\]
Since $\Psi$ is strictly increasing and $\ln a<0$,
$r\leq\lambda_*|\xi_0|$.  Thus $r=\lambda_*|\xi_0|$, and
$r^2-cr+f'(0)=0$ gives
\[
 \lambda_0=c\lambda_*|\xi_0|=f'(0)(1+|\xi_0|^2).
\]
This proves \eqref{eq:translation-spectral-surface} for every
$k_0\in K_U$.
\end{proof}

\begin{remark}\label{rem:spectral-normalized-uniqueness}
For fixed $(\lambda_0,\xi_0)$ and $\theta$, every subsequential
normalized limit in the preceding proof is the same.  If
$\xi_0\ne0$, the direction, speed, and value $\Psi(0)=\theta$
determine the wave uniquely by Lemma~\ref{lem:wave-tail-action}.
If $\xi_0=0$, the scalar initial-value problem
$Y_t=f(Y)$, $Y(0)=\theta$, gives uniqueness.  Consequently the
whole normalized sequence $Y_n$ converges in
$C^{1,2}_{\mathrm{loc}}$ to this solution.  No convergence
uniform down to amplitude zero is asserted or needed.
\end{remark}

\subsection{The critical--subcritical decomposition}

The parameter relation allows us to separate the region $|\xi|<1$
from the critical sphere. We approximate the former by solutions with
strict logarithmic gradient bounds and show that the remaining
solution satisfies the critical time--dilation identity.

\begin{proposition}
\label{prop:cutoff-critical-subcritical-decomposition}
For every $U\in\mathcal C$ there exist $A,C\in\mathcal C$ and
an increasing sequence $A_n\in\mathcal C$ such that
\begin{equation}\label{eq:cutoff-direct-decomposition}
 U=A\oplus C,\qquad
 A_n\nearrow A,\qquad A_n\oplus C\nearrow U
 \quad\text{in }C^{1,2}_{\mathrm{loc}}(\mathbb R\times\mathbb R^N).
\end{equation}
Every nonzero $A_n$ satisfies
\begin{equation}\label{eq:cutoff-direct-gradient-gap}
 |\nabla\ln A_n|\leq\lambda_*r_n,\qquad
 r_n:=1-\frac1n<1\qquad\text{for }n\geq3
\end{equation}
and has a finite associated linear field with measure supported in
$\overline B_{r_n}$.  The remaining solution $C$ satisfies
\begin{equation}\label{eq:cutoff-critical-covariance}
 C(t+s,x)=D_{e^{2f'(0)s}}C(t,x)
 \qquad\text{for every }s\in\mathbb R, (t,x)\in\mathbb R\times\mathbb R^N.
\end{equation}
The zero solution is allowed for $A$, $C$, and $A_n$.
\end{proposition}

\begin{proof}
The case $U=0$ is immediate.  For $U\ne0$, use the representation
above and define, for $n\geq3$,
\[
 \chi_n(q)=\min\{1,\max\{0,n(1-q)-1\}\},\qquad 0\leq q\leq1,
 \qquad
 A_n=\mathcal J^{-1}(\chi_n(|\xi|)).
\]
The functions $\chi_n$ increase with $n$, vanish for
$q\geq r_n$, and equal one for $q\leq1-2/n$.  Thus
\[
 0\leq A_n\leq A_{n+1}\leq U.
\]
For any $h\in\mathbb R^N$, multiplication in $C(K_U)$ gives
\[
 e^{-\lambda_*r_n|h|}\chi_n(|\xi|)
 \leq e^{\lambda_*h\cdot\xi}\chi_n(|\xi|)
 \leq e^{\lambda_*r_n|h|}\chi_n(|\xi|).
\]
Returning to entire solutions and then using the elementary dilation
bounds yields
\[
 \begin{aligned}
 e^{-\lambda_*r_n|h|}A_n(t,x)
 &\leq D_{e^{-\lambda_*r_n|h|}}A_n(t,x)
 \leq A_n(t,x+h)\\
 &\leq D_{e^{\lambda_*r_n|h|}}A_n(t,x)
 \leq e^{\lambda_*r_n|h|}A_n(t,x).
 \end{aligned}
\]
If $A_n\ne0$, it is strictly positive.  Taking logarithms and
difference quotients proves \eqref{eq:cutoff-direct-gradient-gap}.
This argument uses no amplitude differentiability.
The finite-linear-lift theorem for a strict logarithmic gradient
gap, Theorem~\ref{thm:gradient-gap-finite-lift}, gives the asserted
measure supported in $\overline B_{r_n}$.

Bounded monotone parabolic compactness gives
\[
 A_n\nearrow A\quad\text{in }C^{1,2}_{\mathrm{loc}},
 \qquad 0\leq A\leq U,\qquad A\in\mathcal C\cap E_U.
\]
This $A$ is also the supremum of the $A_n$ in $E$: any vector
upper bound is positive because it dominates the nonnegative
$A_n$, and its representing entire solution bounds the pointwise
limit $A$.  Nonlinear subtraction therefore gives a unique
$C\in\mathcal C$ such that
\[
 C=U\ominus A,\qquad U=A\oplus C.
\]
Here $\ominus$ is subtraction in the ordered vector space, not
ordinary subtraction of physical functions.  Moreover $0\leq C\leq U$,
so $C\in E_U$.  Put $g=\mathcal J(C)$.  Linearity and order give
\[
 0\leq g=1-\mathcal J(A)
 \leq1-\chi_n(|\xi|)\qquad\text{for every }n\geq3.
\]
At each $k$ with $|\xi(k)|<1$, the right side vanishes for all
sufficiently large $n$.  Consequently
\[
 \operatorname{supp}g\subset\{k:|\xi(k)|=1\}.
\]
By Theorem~\ref{thm:translation-spectral-surface}, $\lambda=2f'(0)$
on this set.  Hence, for every $s\in\mathbb R$,
\[
 \mathcal J(T_{s,0}C)=e^{s\lambda}g=e^{2f'(0)s}g
 =\mathcal J(D_{e^{2f'(0)s}}C).
\]
Injectivity of $\mathcal J$ proves
\eqref{eq:cutoff-critical-covariance}.  Finally, increasing
continuity of superposition gives
\[
 A_n\oplus C\nearrow A\oplus C=U;
\]
bounded parabolic compactness upgrades this convergence to
$C^{1,2}_{\mathrm{loc}}$, completing
\eqref{eq:cutoff-direct-decomposition}.
\end{proof}

In particular, this construction does not require a prior critical
projection, continuity of that projection, or preservation of the
critical trace under a decreasing limit.  It produces the critical
covariant remainder directly from the spectral support of
$\mathcal J(U\ominus A)$.
\section{Construction and identification of critical entire solutions}
\label{sec:critical-measure-solutions}

We construct entire solutions from arbitrary finite positive measures
on the critical sphere and prove that every critical component obtained
in Section~\ref{sec:degenerate-spectral-decomposition} belongs to this
family. This establishes Theorem~\ref{thm:borel-main} and identifies
the part not covered by the strict gradient bound.

Throughout this section the fixed reaction satisfies
\eqref{eq:general-kpp-assumptions}, with $f'(0)>0$, and uses the
concave extension fixed after \eqref{eq:general-kpp-assumptions}.
We retain $F_1$ from \eqref{eq:reaction-defects}, write $S_t$ for
the Cauchy solution operator, and set
\[
 \mathcal F[v]=v_t-\Delta v-f(v),\qquad
 M_f=\tfrac12\max_{0\leq u\leq1}|f''(u)|.
\]
Concavity gives $f'(u)\leq f'(0)$ for $u\geq0$, and Taylor's formula gives
$0\leq F_1(u)\leq M_fu^2$ for $0\leq u\leq1$.

For a finite positive Borel measure $\mu$ on $\mathbb S^{N-1}$,
recall $L_\mu$ and $z_\mu$ from \eqref{eq:critical-measure-phase}.
We first construct a radially symmetric solution. It provides a uniform
upper bound for the subsequent construction with arbitrary finite
Borel measures on $\mathbb S^{N-1}$. Dimension one is treated below.

Fix one translate of the increasing critical wave, satisfying
\[
 \phi_*''-c_*\phi_*'+f(\phi_*)=0,\qquad
 \phi_*(-\infty)=0,\qquad \phi_*(+\infty)=1.
\]
The wave estimates in
Subsection~\ref{subsec:wave-amplitude-limits}, together with
\eqref{eq:intro-critical-expansion}, give
\begin{equation}\label{eq:radial-wave-estimates}
 \phi_*(s)\sim(-s)e^{\lambda_*s}\quad(s\to-\infty),\qquad
 |\phi_*^{(j)}(s)|\leq C(1+|s|)e^{\lambda_*s}
 \quad\text{for }s\leq-1,\ j=0,1,2,
\end{equation}
We also need $\phi_*''-\lambda_*\phi_*'\leq0$.
For completeness, if $h=\phi_*''-\lambda_*\phi_*'=\lambda_*\phi_*'-f(\phi_*)$, then
\[
 h'-\lambda_*h=(f'(0)-f'(\phi_*))\phi_*'\geq0.
\]
The boundedness of the profile and its first two derivatives gives
$e^{-\lambda_*s}h(s)\to0$ at $+\infty$. Integration of
$(e^{-\lambda_*s}h)'\geq0$ therefore gives $h\leq0$ on $\mathbb R$.
In particular, for
\[
 g(q)=\phi_*(\lambda_*^{-1}\ln q)\quad\text{for }q>0,\qquad g(0)=0,
\]
we have $g''(q)=\lambda_*^{-2}q^{-2}h(\lambda_*^{-1}\ln q)\leq0$. The critical tail ensures
continuity at zero, so $g$ is concave on $[0,\infty)$.

\subsection{The radially symmetric solution}
\label{sec:radial-reference-complete}

We construct the radially symmetric solution that will bound the
solutions generated by general critical-sphere measures. The main
step is to build a supersolution that decays on compact sets in
the remote past and dominates the critical-profile subsolution.

Throughout this subsection $N\geq2$.
Let $\sigma$ be the normalized surface area measure on $\mathbb S^{N-1}$, so that
\[
 \int_{\mathbb S^{N-1}}d\sigma(\nu)=\sigma(\mathbb S^{N-1})=1.
\]
{
The following theorem constructs the radially symmetric entire
solution that will provide an upper bound for solutions associated
with general measures.
}

\begin{theorem}
\label{thm:radial-critical-solution}
There exists a radially symmetric entire solution
$u_\sigma=U^{\rm rad}$ satisfying
\[
 0<u_\sigma<1,\qquad \partial_tu_\sigma>0,
 \qquad \phi_*(z_\sigma)\leq u_\sigma,
\]
and, for every $M>0$,
\[
 \lim_{t\to-\infty}\sup_{|x|\leq M}u_\sigma(t,x)=0.
\]
It is minimal above $\phi_*(z_\sigma)$: if $W\in\mathcal E$ and
$W\geq\phi_*(z_\sigma)$, then $W\geq u_\sigma$.
For every $b>0$, the minimal entire solution above
$\phi_*(z_{b\sigma})$ is
\[
 u_{b\sigma}(t,x)
 =u_\sigma\left(t+\frac{\ln b}{2f'(0)},x\right).
\]
\end{theorem}

\begin{proof}
Set $\mathcal F[v]:=v_t-\Delta v-f(v)$.
We use $L_\sigma$ and $z_\sigma$ defined in \eqref{eq:critical-measure-phase}.
Since the surface area measure $\sigma$ is invariant under orthogonal
transformations, $L_\sigma$ is radially symmetric, with
$L_\sigma(x)=L_\sigma(|x|e_1)$, where $e_1=(1,0,\cdots,0)$.
We also write $z_\sigma(t,r)=z_\sigma(t,re_1)$ for $r\geq0$.
The function $L_\sigma$ is smooth, positive, and solves
$\Delta L_\sigma=f'(0)L_\sigma$. Thus
\begin{equation}\label{eq:radial-log-identity}
 \partial_{rr} z_\sigma+\lambda_*(\partial_r z_\sigma)^2+\frac{N-1}{r}\partial_r z_\sigma=\lambda_*
 \qquad\text{for }t\in\R,\ r>0.
\end{equation}
Differentiation under the integral gives
\[
\begin{aligned}
 \partial_r z_\sigma(t,r)&=\frac{\displaystyle\int_{\mathbb S^{N-1}}\nu_1e^{\lambda_*r\nu_1}\,\sigma(d\nu)}{L_\sigma(re_1)},\\
 \partial_{rr} z_\sigma(t,r)&=\frac{\displaystyle\lambda_*\iint_{\mathbb S^{N-1}\times\mathbb S^{N-1}}(\nu_1-\eta_1)^2
 e^{\lambda_*r(\nu_1+\eta_1)}\,\sigma(d\nu)\sigma(d\eta)}
 {2L_\sigma(re_1)^2}.
\end{aligned}
\]
Clearly,
\begin{equation}\label{eq:radial-log-signs}
 0<\partial_r z_\sigma(t,r)<1,\qquad \partial_{rr} z_\sigma(t,r)\geq0
 \qquad\text{for }t\in\R,\ r>0.
\end{equation}

Here are the elementary bounds that will be used below:
\begin{equation}\label{eq:radial-laplace-bounds}
 c(1+r)^{-\frac{N-1}{2}}e^{\lambda_*r}\leq L_\sigma(re_1)
 \leq C(1+r)^{-\frac{N-1}{2}}e^{\lambda_*r}
 \qquad\text{for }r\geq0,
\end{equation}
and
\begin{equation}\label{eq:radial-log-derivative-bound}
 0\leq1-\partial_r z_\sigma(t,r)\leq\frac Cr
 \qquad\text{for }t\in\R,\ r\geq1.
\end{equation}
Indeed, the integral representation of the modified Bessel function
of the first kind gives
\[
 L_\sigma(re_1)=\Gamma\!\left(\frac N2\right)
 \left(\frac{2}{\lambda_*r}\right)^{\frac N2-1}
 I_{\frac N2-1}(\lambda_*r)
 \qquad\text{for }r>0.
\]
Its differentiation formula and expansion yield
\[
\begin{aligned}
 L_\sigma(re_1)
 &=\frac{2^{\frac N2-1}\Gamma(N/2)}{\sqrt{2\pi}}
   (\lambda_*r)^{-\frac{N-1}{2}}e^{\lambda_*r}
   \bigl(1+O(r^{-1})\bigr),\\
 \partial_r z_\sigma(t,r)&=\frac{I_{\frac N2}(\lambda_*r)}{I_{\frac N2-1}(\lambda_*r)}
 =1-\frac{N-1}{2\lambda_*r}+O(r^{-2}),
\end{aligned}
\qquad\text{as }r\to\infty.
\]
Together with \eqref{eq:radial-log-signs} and positivity and continuity
on bounded intervals, these imply \eqref{eq:radial-laplace-bounds}
and \eqref{eq:radial-log-derivative-bound}.

For $t=-T<0$, define $R(t)>0$ by
\begin{equation}\label{eq:radial-radius-definition}
 z_\sigma(t,R(t))=0.
\end{equation}
Existence and uniqueness follow from $z_\sigma(t,0)=c_*t<0$, $\partial_r z_\sigma(t,r)>0$
for $r>0$, and $z_\sigma(t,r)\to\infty$ as $r\to\infty$.
Since $\partial_t z_\sigma=c_*$, \eqref{eq:radial-laplace-bounds},
\eqref{eq:radial-log-derivative-bound}, and \eqref{eq:radial-radius-definition} imply
\begin{equation}\label{eq:radial-radius-estimates}
 R(-T)=c_*T+\frac{N-1}{c_*}\ln T+O(1),\qquad
 R'(t)=-\frac{c_*}{\partial_r z_\sigma(t,R(t))},\qquad
 c_*\leq-R'(t)\leq C
\end{equation}
for all sufficiently large $T$.

Define
\[
 \underline u(t,x):=\phi_*(z_\sigma(t,x)).
\]
The function $z_\sigma$ is smooth even at $x=0$.
If $H_0=e^{2f'(0)t}L_\sigma(x)$, then
$(H_0)_t-\Delta H_0=f'(0)H_0$ and
$|\nabla H_0|\leq\lambda_*H_0$.
Equivalently,
\[
 \partial_t z_\sigma-\Delta z_\sigma=\lambda_*(1+|\nabla z_\sigma|^2),\qquad |\nabla z_\sigma|\leq1.
\]
The wave equation therefore yields
\begin{equation}\label{eq:radial-global-subsolution}
 \mathcal F[\underline u]=(1-|\nabla z_\sigma|^2)
 \bigl(\phi_*''(z_\sigma)-\lambda_*\phi_*'(z_\sigma)\bigr)\leq0.
\end{equation}
Thus $\underline u$ is a positive bounded global subsolution.

We now construct a bounded supersolution for all sufficiently negative
times. Put
\[
 d(t,x):=|x|-R(t),\qquad
 L(t):=\frac{R(t)}{(\ln R(t))^2},\qquad
 \eta(t):=\frac{D}{\ln R(t)},
\]
where $D>0$ will be chosen. The function $d$ is smooth when $x\ne0$.
For $d\geq-3L(t)$ and large $R(t)$, one has $|x|\geq R(t)/2$, and
\[
 |\nabla d|=1,\qquad d_t=-R'(t),\qquad
 \Delta d=\frac{N-1}{R(t)+d}.
\]
At $d=0$, \eqref{eq:radial-log-identity} gives the exact inequality
\begin{align*}
 -R'(t)-\frac{N-1}{R(t)}-c_*
 &=\frac{c_*}{\partial_r z_\sigma(t,R(t))}-\frac{N-1}{R(t)}-c_*\\
 &=\frac{\lambda_*(1-\partial_r z_\sigma(t,R(t)))^2+\partial_{rr} z_\sigma(t,R(t))}{\partial_r z_\sigma(t,R(t))}\geq0.
\end{align*}
For $-3L(t)\leq d\leq0$, using $d_t=-R'(t)$ and
$\Delta d=(N-1)/(R(t)+d)$, we obtain
\begin{align*}
 d_t-\Delta d-c_*
 &=\left(-R'(t)-\frac{N-1}{R(t)}-c_*\right)
   +(N-1)\left(\frac1{R(t)}-\frac1{R(t)+d}\right)\\
 &\geq\frac{(N-1)d}{R(t)(R(t)+d)}
 \geq-\frac{CL(t)}{R(t)^2},
\end{align*}
where the last inequality uses $|d|\leq3L(t)$ and $R(t)+d\geq R(t)/2$
for all sufficiently negative $t$.
For $d\geq0$, the same identity gives $d_t-\Delta d-c_*\geq0$.
On the other hand,
\[
 \eta'(t)=\frac{-DR'(t)}{R(t)(\ln R(t))^2}
 \geq\frac{c_*D}{R(t)(\ln R(t))^2}.
\]
Since $L(t)/R(t)^2=1/(R(t)(\ln R(t))^2)$,
the shifted wave
\[
 v(t,x):=\phi_*(d(t,x)+\eta(t))
\]
satisfies
\begin{equation}\label{eq:radial-outer-profile}
 \mathcal F[v]
 =\phi_*'(d+\eta(t))(d_t+\eta'(t)-\Delta d-c_*)\geq0
 \qquad\text{for }d\geq-3L(t),
\end{equation}
by taking $D$ sufficiently large.

The supersolution inequality above is only established for
$d\geq-3L(t)$, and $d=|x|-R(t)$ is not smooth at the origin.
We therefore cut off $v$ in the deep interior.
Fix a smooth function $0\leq\chi\leq1$ with
$\chi(s)=0$ for $s\leq-2$ and $\chi(s)=1$ for $s\geq-1$.
Define
\[
 \Psi(t,x):=\chi(d(t,x)/L(t))\phi_*(d(t,x)+\eta(t)).
\]
It agrees with $v$ when $d\geq-L(t)$, and it
vanishes when $d\leq-2L(t)$.

On the strip $-2L(t)\leq d\leq-L(t)$, one has
$d_t=O(1)$, $\Delta d=O(R(t)^{-1})$, and
$L'(t)=O((\ln R(t))^{-2})$. The chain rule gives
\[
 \left|\nabla\!\left[\chi\!\left(\frac d{L(t)}\right)\right]\right|
 \leq\frac C{L(t)},\qquad
 \left|(\partial_t-\Delta)\!\left[\chi\!\left(\frac d{L(t)}\right)\right]\right|
 \leq\frac C{L(t)}.
\]
The product identity is
\begin{align*}
 \mathcal F[\Psi]
 &=\chi\!\left(\frac d{L(t)}\right)\mathcal F[v]
 +v(\partial_t-\Delta)\!\left[\chi\!\left(\frac d{L(t)}\right)\right]\\
 &\quad-2\nabla\!\left[\chi\!\left(\frac d{L(t)}\right)\right]\cdot\nabla v
 +\chi\!\left(\frac d{L(t)}\right)f(v)
 -f\!\left(\chi\!\left(\frac d{L(t)}\right)v\right).
\end{align*}
The first term is nonnegative by
\eqref{eq:radial-outer-profile}. The last two terms satisfy
\[
 \chi\!\left(\frac d{L(t)}\right)f(v)
 -f\!\left(\chi\!\left(\frac d{L(t)}\right)v\right)
 =-\chi\!\left(\frac d{L(t)}\right)F_1(v)
 +F_1\!\left(\chi\!\left(\frac d{L(t)}\right)v\right)
 \geq-M_fv^2.
\]
Lemma~\ref{lem:wave-tail-action} and  \eqref{eq:intro-critical-expansion}, together with $\eta(t)\to0$ and
$L(t)\to\infty$ as $t\to-\infty$, bound $v$ and $|\nabla v|$ on this strip by
$C(1+L(t))e^{\lambda_*d}$. Consequently
\begin{equation}\label{eq:radial-cutoff-error}
 \mathcal F[\Psi]\geq-C(1+L(t))e^{\lambda_*d}
 \qquad\text{for }-2L(t)\leq d\leq-L(t).
\end{equation}
Outside this strip, $\mathcal F[\Psi]\geq0$.

The cutoff introduces an error on
$-2L(t)\leq d(t,x)\leq-L(t)$, as estimated in
\eqref{eq:radial-cutoff-error}. To compensate for this error, we add a correction term $H$ to $\Psi$.
Since the wave profile is close to zero in this region, we construct
$H$ from solutions of the linearized equation. Fix $p>0$ and define
\[
 K(t,x):=\int_0^1 e^{f'(0)(1+\ell^2)t}L_\sigma(\ell x)\,d\ell,
 \qquad H(t,x):=R(t)^{-p}K(t,x).
\]
For each $\ell\in[0,1]$, the function
$e^{f'(0)(1+\ell^2)t}L_\sigma(\ell x)$ satisfies
\[
 (\partial_t-\Delta-f'(0))
 \left[e^{f'(0)(1+\ell^2)t}L_\sigma(\ell x)\right]=0,
\]
because $\Delta L_\sigma(\ell x)=f'(0)\ell^2L_\sigma(\ell x)$.
Differentiating under the integral gives $K_t-\Delta K=f'(0)K$.
Multiplying $K$ by $R(t)^{-p}$ gives
\begin{equation}\label{eq:radial-correction-equation}
 H_t-\Delta H-f'(0)H=\beta(t)H,\qquad
 \beta(t):=-p\frac{R'(t)}{R(t)},\qquad
 \frac cT\leq\beta(-T)\leq\frac CT.
\end{equation}

The correction term $H$ is small on every fixed compact set. For every fixed \(M>0\) and $|x|\leq M$,
one has $L_\sigma(\ell x)\leq e^{\lambda_*M}$, and therefore
\begin{equation}\label{eq:radial-correction-compact}
 \sup_{|x|\leq M}K(-T,x)
 \leq C_Me^{-f'(0)T}\int_0^1e^{-f'(0)T\ell^2}\,d\ell
 \leq C_M T^{-1/2}e^{-f'(0)T}.
\end{equation}

We next show that the correction term $H$ dominates both the subsolution and
the cutoff error in the required interior region. If $d(-T,x)\leq-L(-T)$,
then, by \eqref{eq:radial-radius-estimates}, for a fixed $c_0>0$ and
all sufficiently large $T$,
\begin{equation}\label{eq:radial-interior-gap}
 0\leq r:=|x|\leq R(-T)-L(-T)
 \leq c_*T-\frac{c_0T}{(\ln T)^2}.
\end{equation}
In particular, $r\leq c_*T-\sqrt T$. Put
$\ell_0=r/(c_*T)$. The interval
$[\ell_0,\ell_0+1/(2\lambda_*\sqrt T)]$ lies in $[0,1]$. The lower
bound in \eqref{eq:radial-laplace-bounds} gives
\begin{align*}
 K(-T,x)
 &\geq c e^{-f'(0)T+r^2/(4T)}
 \int_0^1(1+\lambda_*\ell r)^{-\frac{N-1}{2}}e^{-f'(0)T(\ell-\ell_0)^2}\,d\ell\\
 &\geq cT^{-1/2}(1+r^2/T)^{-\frac{N-1}{2}}e^{-f'(0)T+r^2/(4T)}.
\end{align*}
For $\ell\in[\ell_0,\ell_0+1/(2\lambda_*\sqrt T)]$,
$f'(0)T(\ell-\ell_0)^2\leq1/4$ and
\[
 1+\lambda_*\ell r\leq1+\frac{r^2}{2T}
                  +\frac{r}{2\sqrt T}
 \leq C(1+r^2/T).
\]
Meanwhile,
\[
 e^{\lambda_*z_\sigma(-T,x)}=e^{-2f'(0)T}L_\sigma(re_1)
 \leq C(1+r)^{-\frac{N-1}{2}}e^{-2f'(0)T+\lambda_*r}.
\]
Since $R(-T)\asymp T$ as $T\to\infty$\footnote{The notation
$A\asymp B$ means that $C^{-1}B\leq A\leq CB$ for some constant
$C\geq1$ independent of the asymptotic variables under consideration.},
it follows that throughout
\eqref{eq:radial-interior-gap},
\begin{align}
 \frac{H(-T,x)}{T e^{\lambda_*z_\sigma(-T,x)}}
 &\geq cT^{-p-3/2}
 \left(\frac{1+r}{1+r^2/T}\right)^{\frac{N-1}{2}}
 \exp\left(\frac{(c_*T-r)^2}{4T}\right),
 \label{eq:radial-correction-ratio}\\
 \frac{\beta(-T)H(-T,x)}{T e^{\lambda_*z_\sigma(-T,x)}}
 &\geq cT^{-p-5/2}
 \left(\frac{1+r}{1+r^2/T}\right)^{\frac{N-1}{2}}
 \exp\left(\frac{(c_*T-r)^2}{4T}\right).
 \label{eq:radial-residual-ratio}
\end{align}
Here $r\leq c_*T$ implies
$(1+r)/(1+r^2/T)\geq\min\{1,c_*^{-1}\}>0$.
Since $c_*T-r\geq c_0T/(\ln T)^2$ by
\eqref{eq:radial-interior-gap}, we also have
\[
 \frac{(c_*T-r)^2}{4T}\geq\frac{cT}{(\ln T)^4}.
\]
Both right sides in \eqref{eq:radial-correction-ratio}--
\eqref{eq:radial-residual-ratio} consequently tend to infinity as
$T\to\infty$, uniformly for $d(-T,x)\leq-L(-T)$.

On $d(-T,x)\leq-L(-T)$, one has $-c_*T\leq z_\sigma(-T,x)\leq-1$ for large $T$.
 The critical wave-tail estimate
\eqref{eq:intro-critical-expansion} gives $\underline u\leq CT e^{\lambda_*z_\sigma}$, so
\eqref{eq:radial-correction-ratio} proves
\begin{equation}\label{eq:radial-interior-order}
 H\geq\underline u\qquad\text{for }d\leq-L(t).
\end{equation}
For $-2L(t)\leq d\leq0$, the identity
\[
 z_\sigma(t,x)-d(t,x)=\int_{R(t)+d(t,x)}^{R(t)}(1-\partial_r z_\sigma(t,s))\,ds
\]
and \eqref{eq:radial-log-derivative-bound} give
\begin{equation}\label{eq:radial-phase-distance}
 0\leq z_\sigma-d\leq C|d|/R(t)\leq C/(\ln R(t))^2.
\end{equation}
Since $e^{\lambda_*d}\leq e^{\lambda_*z_\sigma}$ and
$1+L(t)\leq C|t|$, estimates \eqref{eq:radial-residual-ratio} and
\eqref{eq:radial-cutoff-error} give, for $T_0$ sufficiently large,
\begin{equation}\label{eq:radial-cutoff-absorbed}
 \mathcal F[\Psi]+\beta H\geq0
 \qquad\text{for }-2L(t)\leq d\leq-L(t),\ t\leq-T_0.
\end{equation}

Define, for $t\leq-T_0$,
\[
 \overline u(t,x):=\min\{1,\Psi(t,x)+H(t,x)\}.
\]
This is a bounded supersolution. Indeed, writing $W=\Psi+H$,
on the set $\{W<1\}$ we have
\begin{align*}
 \mathcal F[W]
 &=\mathcal F[\Psi]+\beta H+f'(0)H+f(\Psi)-f(\Psi+H)\\
 &\geq\mathcal F[\Psi]+\beta H\geq0.
\end{align*}
The first inequality follows by integrating $f'\leq f'(0)$ between
$\Psi$ and $\Psi+H$. The second follows from
\eqref{eq:radial-cutoff-absorbed} on the cutoff strip, from
\eqref{eq:radial-outer-profile} where $d\geq-L$, and from $\Psi=0$
where $d\leq-2L$.

Here is a distributional verification of the cap. Put
$\mathscr L=\partial_t-\Delta$. For a smooth function $v$, convex
smooth approximations of $v_+=\max\{v,0\}$ give
\[
 \mathscr L(v_+)\leq\mathbf1_{\{v>0\}}\mathscr Lv
\]
in distributions. More explicitly, choose increasing convex
$\theta_\varepsilon$ which vanish for $s\leq0$, have derivatives
bounded by one, and converge to $s_+$. Then
\[
 \mathscr L\theta_\varepsilon(v)
 =\theta_\varepsilon'(v)\mathscr Lv
  -\theta_\varepsilon''(v)|\nabla v|^2
 \leq\theta_\varepsilon'(v)\mathscr Lv,
\]
and pass to the distributional limit. Applied to $v=1-W$, this gives
\[
 \mathscr L\overline u
 =-\mathscr L(1-W)_+
 \geq\mathbf1_{\{W<1\}}\mathscr LW
 \geq\mathbf1_{\{W<1\}}f(W)=f(\overline u).
\]
The last equality uses $f(1)=0$. This argument is local before the
cap, so the spatial growth of $H$ causes no difficulty. Afterwards,
$0\leq\overline u\leq1$, and all comparisons are between bounded
functions.

We verify $\underline u\leq\overline u$ everywhere. On $d\leq-L(t)$ this follows
from \eqref{eq:radial-interior-order}. On $-L(t)\leq d\leq0$,
\eqref{eq:radial-phase-distance} and
$\eta(t)=D/\ln R(t)$ give $z_\sigma\leq d+\eta(t)$ for large $R(t)$, where $\Psi=v$.
On $d\geq0$, integration of $\partial_r z_\sigma(t,r)\leq1$ gives $z_\sigma\leq d$.
Thus $\underline u\leq\Psi$ on $d\geq-L(t)$, and we obtain
\begin{equation}\label{eq:radial-barrier-order}
 0<\underline u\leq\overline u\leq1\qquad\text{for }t\leq-T_0.
\end{equation}

Both $\underline u$ and $\overline u$ are radially symmetric in $x$.
For each integer $n>T_0$, let $u_n$ be the Cauchy solution
with initial data $u_n(-n,x)=\underline u(-n,x)$.
Rotational invariance of the equation and uniqueness of the Cauchy
problem imply that each $u_n$ is radially symmetric in $x$.
The parabolic comparison principle gives
\[
 \underline u\leq u_n\leq\overline u\leq1
 \quad\text{for }-n\leq t\leq-T_0,\qquad
 0\leq u_n\leq1\quad\text{for }t\geq-n.
\]

At time $-n$, the subsolution gives
$u_{n+1}(-n,\cdot)\geq\underline u(-n,\cdot)=u_n(-n,\cdot)$.
Hence $u_{n+1}(t,x)\geq u_n(t,x)$ for $t\geq-n$ and $x\in\R^N$. Boundedness and
interior parabolic estimates give
\[
 u_n\nearrow U^{\rm rad}
 \quad\text{in }C^{1,2}_{\mathrm{loc}}(\R\times\R^N)
 \quad\text{as }n\to\infty.
\]
The limit $U^{\rm rad}$ is therefore a radially symmetric entire solution and satisfies
$\underline u\leq U^{\rm rad}\leq\overline u$ for $t\leq-T_0$.

For each fixed $M$, the function $\Psi(-T,x)$ vanishes on
$|x|\leq M$ once $T$ is sufficiently large, because $R(-T)-2L(-T)\to\infty$.
Consequently \eqref{eq:radial-correction-compact} gives
\begin{equation}\label{eq:radial-strict-past-bound}
 \sup_{|x|\leq M}U^{\rm rad}(-T,x)
 \leq C_M T^{-p-1/2}e^{-f'(0)T}\rightarrow0.
\end{equation}
The lower bound by $\underline u$ proves positivity. The past bound excludes
the constant solution $1$. More explicitly, if $U^{\rm rad}$ took
the value $1$ at a finite space--time point, the strong parabolic
maximum principle for $1-U^{\rm rad}$ would force it to equal $1$
at all earlier times, contradicting
\eqref{eq:radial-strict-past-bound}. Indeed, $w=1-U^{\rm rad}$ satisfies
$w_t-\Delta w=c(t,x)w$ with
$c=-f(U^{\rm rad})/(1-U^{\rm rad})$, whose continuous extension
at $U^{\rm rad}=1$ is $f'(1)$ and is bounded. Thus
\begin{equation}\label{eq:radial-strict-reference}
 0<U^{\rm rad}<1,
 \qquad U^{\rm rad}(t,0)\rightarrow0\quad(t\to-\infty).
\end{equation}

By the comparison principle, $U^{\rm rad}$ is the minimal
entire solution above $\underline u=\phi_*(z_\sigma)$.
We denote this solution associated with $\sigma$ by $u_\sigma:=U^{\rm rad}$.

Finally, $U^{\rm rad}$ is strictly increasing in time. For $h>0$,
\[
 u_n(-n+h,\cdot)\geq \underline u(-n+h,\cdot)\geq \underline u(-n,\cdot).
\]
Comparing $u_n(t+h,\cdot)$ with $u_n(t,\cdot)$ and taking the limit
gives $U^{\rm rad}(t+h,x)\geq U^{\rm rad}(t,x)$.
Thus $v=\partial_tU^{\rm rad}\geq0$ and
\[
 v_t-\Delta v=f'(U^{\rm rad})v.
\]
If $v$ vanished at a finite point, the strong maximum principle
would make it vanish at all earlier times. The solution would then
be time-independent on that earlier interval, and its past limit
would force it to be zero there, contradicting positivity.
Therefore $\partial_tU^{\rm rad}>0$ everywhere.

For the measure $b\sigma$, where $b>0$, the corresponding minimal
entire solution is a time translate of $u_\sigma$. Indeed,
\[
 \phi_*(z_{b\sigma}(t,x))
 =\phi_*(z_\sigma(t,x)+\lambda_*^{-1}\ln b)
 =\underline u\left(t+\frac{\ln b}{2f'(0)},x\right).
\]
Thus $U^{\rm rad}(t+\ln b/(2f'(0)),x)$ is an entire solution above
$\phi_*(z_{b\sigma}(t,x))$. If $W\in\mathcal E$
satisfies $W\geq\phi_*(z_{b\sigma})$, then
\[
 W\left(t-\frac{\ln b}{2f'(0)},x\right)
 \geq\underline u(t,x).
\]
Minimality of $U^{\rm rad}$ and time translation therefore give
\[
 U^{\rm rad}\left(t+\frac{\ln b}{2f'(0)},x\right)\leq W(t,x).
\]
Hence this time translate is the minimal entire solution
above $\phi_*(z_{b\sigma})$, which we denote by
\begin{equation}\label{eq:radial-mass-time-scaling}
 u_{b\sigma}(t,x)
 :=u_\sigma\left(t+\frac{\ln b}{2f'(0)},x\right)
 =U^{\rm rad}\left(t+\frac{\ln b}{2f'(0)},x\right).
\end{equation}
This completes the proof.
\end{proof}

\subsection{Entire solutions from critical-sphere measures}
\label{sec:canonical-critical-construction}

We now construct a canonical entire solution for each finite positive
measure on the critical sphere. Spherical averaging provides a radial
upper bound, and comparison relates scaling of the measure to dilation
and time translation of the solution.

Let $g(q)=\phi_*(\lambda_*^{-1}\ln q)$ for $q>0$, with $g(0)=0$.
Its monotonicity and concavity follow from Lemma~\ref{lem:wave-tail-action}.
Let $\sigma$ be the normalized surface measure on $\mathbb S^{N-1}$. For $N=1$, it means $(\delta_{+1}+\delta_{-1})/2$.
For a finite positive Borel measure $\mu$, recall
$L_\mu$ and $z_\mu$ from \eqref{eq:critical-measure-phase} and set
\begin{equation}\label{eq:canonical-critical-linear-field}
 H_\mu(t,x):=e^{2f'(0)t}L_\mu(x)=e^{\lambda_*z_\mu(t,x)}.
\end{equation}

{
We use the radial solution to ensure that the construction for an
arbitrary finite positive measure stays strictly below one. We also
identify the action of dilation on the measure.
}

\begin{theorem}\label{thm:canonical-critical-solutions}
For every finite positive Borel measure $\mu$ on $\mathbb S^{N-1}$,
there exists an entire solution $u_\mu$ of \eqref{eq:kpp} satisfying
\begin{equation}\label{eq:canonical-critical-radial-upper}
 0<\phi_*(z_\mu(t,x))\leq u_\mu(t,x)
 \leq U^{\rm rad}\!\left(\frac{z_\mu(t,x)}{c_*},0\right)<1
\end{equation}
for every $(t,x)\in\R\times\R^N$, and
\[
 u_\mu(t,\cdot)\rightarrow0
 \qquad\text{locally uniformly in }\R^N
 \quad\text{as }t\to-\infty.
\]
Moreover,
\begin{equation}\label{eq:canonical-critical-dilation}
 (D_bu_\mu)(t,x)=u_{b\mu}(t,x)
 =u_\mu\!\left(t+\frac{\ln b}{2f'(0)},x\right)
 \qquad\text{for }b>0.
\end{equation}
\end{theorem}

\begin{proof}
These functions and all their derivatives are finite on compact
space--time sets. They satisfy $H_t-\Delta H=f'(0)H$ and
$|\nabla H|\leq\lambda_*H$. Direct substitution therefore gives
\[
 (\partial_t-\Delta)g(H)-f(g(H))
 =g''(H)(f'(0)H^2-|\nabla H|^2)\leq0.
\]
Thus $g(H_\mu)$ is a bounded subsolution. For each fixed $(t,x)$,
$\bigl(S_{t+T}[g(H_\mu(-T,\cdot))]\bigr)(x)$ is nondecreasing in
$T>\max\{0,-t\}$ and bounded above by $1$. Hence
\begin{equation}\label{eq:canonical-critical-pullback}
 u_\mu(t,x):=\lim_{T\to\infty}\bigl(S_{t+T}[g(H_\mu(-T,\cdot))]\bigr)(x)
\end{equation}
defines an entire solution of \eqref{eq:kpp}. The convergence holds in
$C^{1,2}_{\mathrm{loc}}(\R\times\R^N)$ by parabolic
estimates. We set $u_0:=0$. If an entire solution $Z$ lies above $g(H_\mu)$, comparison
from time $-T$ and passage to the limit show that $Z\geq u_\mu$.
This is the minimality property used below.

For $N=1$, write $\mu=\alpha\delta_{+1}+\beta\delta_{-1}$, where
$\alpha,\beta\geq0$ and $\alpha+\beta>0$, and set
\[
 v_+(t,x)=\phi_*(x+c_*t+\lambda_*^{-1}\ln\alpha),\qquad
 v_-(t,x)=\phi_*(-x+c_*t+\lambda_*^{-1}\ln\beta),
\]
for $\alpha>0$ and $\beta>0$, respectively, with the conventions
$v_+\equiv0$ if $\alpha=0$ and $v_-\equiv0$ if $\beta=0$.
Monotonicity and
concavity of $g$, together with $g(0)=0$, give
$\max\{v_+,v_-\}\leq g(H_\mu)\leq\min\{v_++v_-,1\}$.
Comparison in \eqref{eq:canonical-critical-pullback} therefore yields
\[
 \max\{v_+,v_-\}\leq u_\mu\leq\min\{v_++v_-,1\}.
\]
By Theorem~\ref{thm:packet}, $u_\mu$ coincides with the corresponding
entire solution of Hamel and Nadirashvili
\cite[Theorem~1.1]{HN2001}. In particular, $0<u_\mu<1$.
For $\mu=\sigma=(\delta_{+1}+\delta_{-1})/2$, symmetry of the initial data gives the even solution
$U^{\rm rad}=u_\sigma$.
The comparison and strong maximum principle argument above gives
$\partial_tU^{\rm rad}>0$; minimality also gives
\eqref{eq:radial-mass-time-scaling} in this dimension.

We next use the radially symmetric solution $U^{\rm rad}$ to prove
that $u_\mu(t,x)<1$ for every finite positive Borel measure $\mu$
and every $(t,x)\in\R\times\R^N$.
For $b>0$, definition \eqref{eq:canonical-critical-linear-field} gives
\[
 H_{b\mu}(t,x)=bH_\mu(t,x)
 =H_\mu\!\left(t+\frac{\ln b}{2f'(0)},x\right).
\]
Substituting this identity into \eqref{eq:canonical-critical-pullback}, we obtain
\begin{equation}\label{eq:canonical-critical-time-scaling}
 \begin{aligned}
 u_{b\mu}(t,x)
 &=\lim_{T\to\infty}\bigl(S_{t+T}[g\!\left(H_\mu\!\left(-T+\frac{\ln b}{2f'(0)},\cdot\right)\right)]\bigr)(x)\\
 &=u_\mu\!\left(t+\frac{\ln b}{2f'(0)},x\right).
 \end{aligned}
\end{equation}

Let $M=\mu(\mathbb S^{N-1})$, and fix $T$. Fubini's theorem and
the radial symmetry of $L_\sigma$ give
\begin{equation}\label{eq:canonical-spherical-laplace-average}
 \begin{aligned}
 \int_{\mathbb S^{N-1}}L_\mu(|x|\xi)\,d\sigma(\xi)
 &=\int_{\mathbb S^{N-1}}
   \left(\int_{\mathbb S^{N-1}}
   e^{\lambda_*|x|\xi\cdot\nu}\,d\sigma(\xi)\right)d\mu(\nu)\\
 &=\int_{\mathbb S^{N-1}}L_\sigma(|x|\nu)\,d\mu(\nu)
 =M L_\sigma(x).
 \end{aligned}
\end{equation}
Concavity of $g$ then implies
\[
 \int_{\mathbb S^{N-1}}g(H_\mu(-T,|x|\xi))\,d\sigma(\xi)
 \leq g(H_{M\sigma}(-T,x)).
\]
Set
\[
 v_{\mu,T}(t,x):=\bigl(S_{t+T}[g(H_\mu(-T,\cdot))]\bigr)(x)
 \qquad\text{for }t\geq-T,
\]
and define its spherical average by
\[
 W_T(t,x):=\int_{\mathbb S^{N-1}}
 v_{\mu,T}(t,|x|\xi)\,d\sigma(\xi).
\]
Since spherical averaging commutes with the Laplacian, concavity of $f$ gives
\[
 (W_T)_t-\Delta W_T
 =\int_{\mathbb S^{N-1}}f(v_{\mu,T}(t,|x|\xi))\,d\sigma(\xi)
 \leq f(W_T).
\]
Thus it is a subsolution.
Comparison gives $W_T\leq v_{M\sigma,T}$. At the origin
$W_T(t,0)=v_{\mu,T}(t,0)$. Sending $T\to\infty$ gives
\begin{equation}\label{eq:canonical-origin-comparison}
 u_\mu(t,0)\leq u_{M\sigma}(t,0)
 \qquad\text{for }t\in\R.
\end{equation}

For a general point $x$, define the finite measure
$d\mu_x(\nu)=e^{\lambda_*x\cdot\nu}d\mu(\nu)$. Then
$L_{\mu_x}(y)=L_\mu(x+y)$ and its total mass is $L_\mu(x)$.
Uniqueness gives
$u_{\mu_x}(t,y)=u_\mu(t,x+y)$. Applying
\eqref{eq:canonical-origin-comparison} to $\mu_x$, whose total mass
is $L_\mu(x)$, gives
\[
 u_\mu(t,x)=u_{\mu_x}(t,0)\leq u_{L_\mu(x)\sigma}(t,0).
\]
By \eqref{eq:canonical-critical-time-scaling} with $b=L_\mu(x)$
and $u_\sigma=U^{\rm rad}$,
\[
 u_{L_\mu(x)\sigma}(t,0)
 =U^{\rm rad}\!\left(t+\frac{\ln L_\mu(x)}{2f'(0)},0\right)
 =U^{\rm rad}\!\left(\frac{z_\mu(t,x)}{c_*},0\right).
\]
Together with the lower bound from \eqref{eq:canonical-critical-pullback}
and $U^{\rm rad}<1$, this proves \eqref{eq:canonical-critical-radial-upper}.
The same bound and the past decay of $U^{\rm rad}$ give
$u_\mu(t,\cdot)\rightarrow0$ locally uniformly in $\R^N$ as $t\to-\infty$.

We next prove that $D_bu_\mu=u_{b\mu}$ for every $b>0$, using
the dilation group established in Proposition~\ref{prop:pde-full-dilation}
and the comparison estimates below. Thus dilation of $u_\mu$
corresponds to multiplication of its defining measure by $b$.
The wave equation, the tail estimate \eqref{eq:intro-critical-expansion},
and Lemma~\ref{lem:wave-tail-action} give
\[
 \frac{d^2}{ds^2}\bigl(e^{-\lambda_*s}\phi_*(s)\bigr)
 =e^{-\lambda_*s}\bigl(f'(0)\phi_*(s)-f(\phi_*(s))\bigr)\geq0,
 \qquad
 -1\leq\frac{d}{ds}\bigl(e^{-\lambda_*s}\phi_*(s)\bigr)\leq0.
\]
For $b>0$ and $q>0$, the preceding derivative bound gives
\begin{equation}\label{eq:canonical-critical-profile-error}
 \begin{aligned}
 |b g(q)-g(bq)|
=bq\left|\int_{\lambda_*^{-1}\ln q}^{\lambda_*^{-1}\ln(bq)}
 \frac{d}{dr}\bigl(e^{-\lambda_*r}\phi_*(r)\bigr)\,dr\right|\leq \frac b{\lambda_*}q|\ln b|.
 \end{aligned}
\end{equation}
The left side is also bounded by $b+1$, hence by
$\min\{\frac b{\lambda_*}|\ln b|H_\mu,b+1\}$ by taking $q=H_\mu$.

Fix $b>0$, $b\ne1$. Applying \eqref{eq:cone-stability} and
\eqref{eq:canonical-critical-profile-error}, and using
$S_t[Bh]\leq B S_t[h]$ for $B\geq1$, we obtain
\[
\begin{aligned}
 &\left|S_{t+T}[b g(H_\mu(-T,\cdot))]
       -S_{t+T}[g(bH_\mu(-T,\cdot))]\right|\\
 &\quad\leq S_{t+T}[\min\left\{
       \frac b{\lambda_*}|\ln b|H_\mu(-T,\cdot),b+1\right\}]\\
 &\quad\leq(b+1)S_{t+T}[\min\left\{
       \frac{b|\ln b|}{\lambda_*(b+1)}H_\mu(-T,\cdot),1\right\}]
 \rightarrow0
 \qquad\text{as }T\to\infty,
\end{aligned}
\]
locally uniformly, by \eqref{eq:general-critical-extinction} with
$\eta=\mu$. Since $bH_\mu=H_{b\mu}$,
\eqref{eq:canonical-critical-pullback} and interior parabolic estimates give
\begin{equation}\label{eq:canonical-scaled-seed-convergence}
 S_{t+T}[b g(H_\mu(-T,\cdot))]
 \rightarrow u_{b\mu}(t,\cdot)
 \qquad\text{in }C^{1,2}_{\mathrm{loc}}(\R\times\R^N)
 \quad\text{as }T\to\infty.
\end{equation}
For $b=1$, this is precisely \eqref{eq:canonical-critical-pullback}.

For $0<b<1$, concavity of $g$ gives
\[
 bS_{t+T}[g(H_\mu(-T,\cdot))]
 \leq S_{t+T}[b g(H_\mu(-T,\cdot))]
 \leq S_{t+T}[g(bH_\mu(-T,\cdot))].
\]
Letting $T\to\infty$ yields $bu_\mu\leq u_{b\mu}$.
Therefore,
\[
 S_{t+T}[bu_\mu(-T,\cdot)]\leq u_{b\mu}(t,\cdot),
\]
and hence $D_bu_\mu\leq u_{b\mu}$. Conversely,
$u_\mu\geq g(H_\mu)$ implies
\[
 S_{t+T}[bu_\mu(-T,\cdot)]
 \geq S_{t+T}[b g(H_\mu(-T,\cdot))].
\]
Letting $T\to\infty$, the definition of $D_b$ and
\eqref{eq:canonical-scaled-seed-convergence} give $D_bu_\mu\geq u_{b\mu}$.
Consequently, $D_bu_\mu=u_{b\mu}$.
For $b>1$, applying this identity with parameter $1/b$ and measure $b\mu$ gives $D_{1/b}u_{b\mu}=u_{(1/b)(b\mu)}=u_\mu$.
Applying $D_b$ and using Proposition~\ref{prop:pde-full-dilation} yields
$D_bu_\mu=u_{b\mu}$. Together with
\eqref{eq:canonical-critical-time-scaling}, this proves
\eqref{eq:canonical-critical-dilation}.
\end{proof}

\subsection{The critical trace and its nonvanishing}
\label{sec:critical-faithfulness}

We extract a measure on the critical sphere from the small-amplitude
limit and prove that this measure is nonzero for each nonzero solution
satisfying the critical time--dilation identity. A spherical averaging
inequality connects the trace to the radial solution constructed above.

Throughout this section the reaction satisfies
\eqref{eq:general-kpp-assumptions}, with $f'(0)>0$.
All entire solutions considered here solve $U_t=\Delta U+f(U)$.
Let $\phi_*$ be the increasing speed-$c_*$ profile fixed earlier, with
\[
 \phi_*(s)\sim(-s)e^{\lambda_*s}\quad(s\to-\infty).
\]
For a finite nonnegative measure $\mu$ on $\mathbb S^{N-1}$, set
\[
 L_\mu(x)=\int_{\mathbb S^{N-1}}e^{\lambda_*x\cdot\nu}\,\mu(d\nu),
 \qquad H_\mu(t,x)=e^{2f'(0)t}L_\mu(x).
\]
For $\mu\ne0$, write $z_\mu=\lambda_*^{-1}\ln H_\mu$. Let $\sigma$ be normalized
surface measure and let $\ell(R)=L_\sigma(Re_1)$; when $N=1$,
$\sigma=(\delta_{-1}+\delta_1)/2$ and $\ell(R)=\cosh(\lambda_*R)$.

The preceding construction gives the minimal entire solution $u_\mu$
above $\phi_*(z_\mu)$ in every dimension, and its radial reference
$U^{\rm rad}=u_\sigma$. In particular,
\begin{gather*}
 \phi_*(z_\mu(t,x))\leq u_\mu(t,x)
 \leq U^{\rm rad}(z_\mu(t,x)/c_*,0),\\
 D_bu_\mu=u_{b\mu},\qquad
 u_{b\mu}(t,x)=u_\mu(t+\tfrac{\ln b}{2f'(0)},x),\qquad
 U^{\rm rad}(t,0)\rightarrow0\quad(t\to-\infty).
\end{gather*}
We also use the finite critical-field estimate already proved:
for every finite angular measure $\eta$, every $C>0$, and every fixed
$(t,x)$,
\begin{equation}\label{eq:faithfulness-critical-error}
 \lim_{T\to\infty}\bigl(S_{t+T}[\min\{Ce^{-2f'(0)T}L_\eta,1\}]\bigr)(x)=0,
\end{equation}
with local $C^{1,2}$ convergence. All data in this formula belong
to $[0,1]$.

The ordered vector space constructed above has $D_a$ as positive
scalar multiplication and $\oplus$ as addition on its positive cone;
its order is pointwise order between physical solutions. For each
entire solution $U$ of \eqref{eq:kpp} with $0<U<1$,
its principal ideal has a representation
$J:E_U\to C(K_U)$ with $J(U)=1$. In this representation translations
are multiplication by $e^{\lambda_*h\cdot\xi+s\lambda}$, where
$\lambda=f'(0)(1+|\xi|^2)$ and $|\xi|\leq1$. These facts were established
before the critical classification and therefore are available here.

A critical solution $U_*$ is an entire solution of \eqref{eq:kpp}
with $0\leq U_*<1$
satisfying
\begin{equation}\label{eq:critical-solution}
 D_aU_*(t,x)=U_*\left(t+\frac{\ln a}{2f'(0)},x\right)
 \qquad\text{for }a>0.
\end{equation}
We first obtain a critical-sphere measure from the small-amplitude
limit of an entire solution. This limit exists for every $0<U<1$;
we will subsequently prove that it is nonzero for each nonzero
solution satisfying \eqref{eq:critical-solution}.
In the following proof, $F_1(u)=f'(0)u-f(u)$ and $M$ are as in
\eqref{eq:general-quadratic-error}.

\begin{theorem}\label{thm:pde-critical-trace}
For every entire solution $U$ of \eqref{eq:kpp} with $0<U<1$,
there is a unique finite
nonnegative measure $\mu_*$ on $\mathbb S^{N-1}$ such that
\begin{equation}\label{eq:pde-critical-trace}
 \frac{D_aU(t,x)}{a\ln(1/a)}
 \rightarrow H_*[U](t,x)
 =e^{2f'(0)t}\int_{\mathbb S^{N-1}}e^{\lambda_*x\cdot\nu}\,\mu_*(d\nu)
 \quad\text{in }C^{1,2}_{\mathrm{loc}}(\mathbb R\times\mathbb R^N)
 \quad\text{as }a\searrow0.
\end{equation}
The measure may be zero. If it is nonzero, then
\begin{equation}\label{eq:pde-trace-lower-barrier}
 U(t,x)\geq\phi_*\left(\frac1{\lambda_*}\ln\bigl(\lambda_*H_*[U](t,x)\bigr)\right)
 \quad\text{for all }(t,x)\in\mathbb R\times\mathbb R^N.
\end{equation}
\end{theorem}
\begin{proof}
Fix $(t,x)$ and set $G(a)=D_aU(t,x)$. The amplitude properties,
Lemma~\ref{lem:amplitude-C1-degenerate}, and
\eqref{eq:trace-F-inequality} with $\gamma=1$ imply all the
hypotheses of Lemma~\ref{lem:scalar-trace-degenerate}, with its quadratic
constant $M/f'(0)$, because $f(G)/f'(0)\geq G-(M/f'(0))G^2$.
Thus the normalized amplitude has a finite
pointwise limit $H_*[U]$.

For $0<a<1$, put $b_a=a\ln(1/a)$ and $Z_a=D_aU/b_a$.
The ancient estimates applied to $D_aU$ give
\[
 e^{-2f'(0)|t|-\lambda_*|x|}Z_a(0,0)\leq Z_a(t,x)
 \leq e^{2f'(0)|t|+\lambda_*|x|}Z_a(0,0).
\]
The pointwise limit at $(0,0)$ consequently bounds $Z_a$ on each
compact set for small $a$. Its equation is
\[
 (Z_a)_t-\Delta Z_a
 =f'(0)Z_a-\frac{F_1(b_aZ_a)}{b_a},\qquad
 0\leq\frac{F_1(b_aZ_a)}{b_a}\leq Mb_aZ_a^2.
\]
Interior parabolic estimates and the pointwise limit yield
convergence in $C^{1,2}_{\mathrm{loc}}$ to a solution of
\[
 (H_*)_t=\Delta H_*+f'(0)H_*,\qquad |\nabla H_*|\leq\lambda_*H_*.
\]
The two-sided comparison shows that $H_*$ is either identically
zero or positive everywhere. In the positive case, the positive
ancient heat representation used in Section~\ref{sec:general-finite-lift}
gives
\begin{equation}\label{eq:trace-preliminary-spectrum}
 H_*(t,x)=\int_{\overline B_1}
 e^{\lambda_*x\cdot\xi+f'(0)(1+|\xi|^2)t}\,\mu(d\xi).
\end{equation}
Indeed the representing measure has mass $H_*(0,0)<\infty$.
The bound $H_*(0,se)\leq e^{\lambda_*s}H_*(0,0)$ for every unit $e$ and
$s\geq0$ excludes mass where $e\cdot\xi>1$. A countable dense
set of unit vectors then excludes all mass outside $\overline B_1$.
Uniqueness follows from uniqueness of the positive heat
representation, or equivalently uniqueness of the Laplace transform.

Divide \eqref{eq:inverse-wave-upper} by $a\ln(1/a)$ and use the
critical tail of $\phi_*$. This gives
\[
 H_*(t,x)\leq\frac1{\lambda_*}\exp\bigl(\lambda_*\phi_*^{-1}(U(t,x))\bigr),
\]
and hence \eqref{eq:pde-trace-lower-barrier}. In particular, the fixed
critical wave $\phi_*(x\cdot\nu+c_*t)$ has trace measure
$\lambda_*^{-1}\delta_\nu$.

It remains to exclude interior mass in
\eqref{eq:trace-preliminary-spectrum}. If $\mu(B_1)>0$, choose
$r\in[0,1)$ with $\mu(\overline B_r)>0$ and set
\[
 h(t,x)=\int_{\overline B_r}
 e^{\lambda_*x\cdot\xi+f'(0)(1+|\xi|^2)t}\,\mu(d\xi).
\]
Then $0<h\leq H_*$ and
$U\geq\phi_*(\lambda_*^{-1}\ln(\lambda_*h))$.
For every $m>0$, the critical tail yields a number $d_m>0$ such that
\[
 \phi_*(\lambda_*^{-1}\ln(\lambda_*s))\geq ms
 \quad\text{for }0<s\leq d_m.
\]
Choose $1<p\leq2$ with $pr^2<1$ and set
$c=f'(0)(p-1)(1-pr^2)>0$. Choose $B_m$ so large that
\[
 B_m\geq\max\{1,M/c\},\qquad
 B_m^{-1/(p-1)}/m\leq d_m.
\]
The lower barrier from the finite linear reconstruction is
\[
 \ell_m=(mh-B_m(mh)^p)_+.
\]
On its positive set $mh<1$ and $h\leq d_m$, so that
$\ell_m\leq mh\leq U$. Its subsolution inequality follows from
\[
 (\partial_t-\Delta-f'(0))(mh)^p\geq c(mh)^p,
 \qquad
 (\partial_t-\Delta)\ell_m-f(\ell_m)
 \leq-B_mc(mh)^p+M(mh)^2\leq0
\]
on that set; taking the maximum with zero preserves the
subsolution inequality.

Let $V_{mh}^f$ be the capped pullback reconstructed in
Theorem~\ref{thm:general-finite-lift}. Since
\[
 0\leq\min\{mh,1\}-\ell_m\leq B_m(mh)^p,
\]
comparison, \eqref{eq:general-cauchy-stability}, and
\eqref{eq:general-power-decay} give
\[
 \begin{split}
 0&\leq V_{mh}^f(t,x)-\bigl(S_{t+T}[\ell_m(-T)]\bigr)(x)\\
 &\leq B_m e^{-c(t+T)}(mh(t,x))^p\rightarrow0
 \quad\text{as }T\to\infty.
 \end{split}
\]
Because $\ell_m\leq U$, it follows that $V_{mh}^f\leq U$ for
every $m>0$. On the other hand, for $s\geq0$,
\[
 h(t+s,x)\leq e^{f'(0)(1+r^2)s}h(t,x),\qquad
 V_h^f(t+s,x)\leq V_{mh}^f(t,x)\leq U(t,x)
 \quad\text{for }m\geq e^{f'(0)(1+r^2)s}.
\]
The first inequality is immediate from the support of $\mu$; the
second uses order and time-translation covariance of the capped
pullback. Since $V_h^f$ is an entire solution of \eqref{eq:kpp}
with $0<V_h^f<1$,
$V_h^f(t+s,x)\to1$ as $s\to\infty$. This contradicts $U<1$.
Thus $\mu(B_1)=0$, proving the sphere representation in
\eqref{eq:pde-critical-trace}.
\end{proof}

{
For solutions satisfying the time--dilation identity, the amplitude
trace becomes a precise past-time asymptotic profile.
}

\begin{corollary}\label{cor:trace-pure-critical}
Suppose that
\[
 D_aU(t,x)=U\left(t+\frac{\ln a}{2f'(0)},x\right)
 \quad\text{for all }a>0.
\]
Then
\[
 \Delta U\geq\frac1{4f'(0)}U_{tt},\qquad
 \frac{U(t,\cdot)}{(-2f'(0)t)e^{2f'(0)t}}
 \rightarrow\kappa
 \quad\text{in }C^2_{\mathrm{loc}}(\mathbb R^N)
 \quad\text{as }t\to-\infty,
\]
where
$\kappa(x)=\int_{\mathbb S^{N-1}}e^{\lambda_*x\cdot\nu}\,\mu_*(d\nu)$
and $\Delta\kappa=f'(0)\kappa$.
\end{corollary}
\begin{proof}
The covariance makes $G(a)=D_aU(t,x)$ twice differentiable in $a$.
At $a=1$,
\[
 G'(1)=\frac1{2f'(0)}U_t,\qquad
 G''(1)=\frac1{4f'(0)^2}U_{tt}-\frac1{2f'(0)}U_t.
\]
Using \eqref{eq:trace-F-inequality} with $\gamma=1$ gives
$U_{tt}/(4f'(0))\leq U_t-f(U)=\Delta U$.
Apply \eqref{eq:pde-critical-trace} at reference time zero with
$a=e^{2f'(0)t}$ to obtain the asserted limit.
\end{proof}

To prove that the critical trace does not vanish, we first compare
the spherical average of a critical solution with its value at a
later time. The comparison follows from the translation representation
and the nonlinear averaging inequality.

\begin{lemma}
\label{lem:faithfulness-spherical-mean}
If $U_*$ is a nonzero critical solution in the sense of
\eqref{eq:critical-solution}, then
\begin{equation}\label{eq:faithfulness-mean-bound}
 \int_{\mathbb S^{N-1}}U_*(t,R\nu)\,\sigma(d\nu)
 \leq U_*(t+\tfrac{\ln L_\sigma(Re_1)}{2f'(0)},0)
 \qquad\text{for }R\geq0,\ t\in\mathbb R.
\end{equation}
\end{lemma}
\begin{proof}
Work in the principal ideal $E_{U_*}$ with its representation
$J:E_{U_*}\to C(K_{U_*})$ and $J(U_*)=1$.
Critical covariance gives $T_{s,0}U_*=D_{e^{2f'(0)s}}U_*$.
Applying $J$ shows $e^{s\lambda(k)}=e^{2f'(0)s}$ for every
$k\in K_{U_*}$, so $\lambda=2f'(0)$.
The spectral surface identity therefore implies $|\xi(k)|=1$
on all of $K_{U_*}$. Define
\[
 f_R(\zeta)=\int_{\mathbb S^{N-1}}
                e^{\lambda_*R\nu\cdot\zeta}\,\sigma(d\nu)
 \qquad\text{for }\zeta\in[-1,1]^N.
\]
Rotational invariance of $\sigma$ gives
$f_R(\xi(k))=L_\sigma(Re_1)$ on $K_{U_*}$.

For each $m\geq1$, choose $\nu_{m,1},\cdots,\nu_{m,k_m}\in\mathbb S^{N-1}$
such that $\mathbb S^{N-1}\subset\bigcup_{j=1}^{k_m}B_{1/m}(\nu_{m,j})$.
Set
\[
 A_{m,j}:=\bigl(B_{1/m}(\nu_{m,j})\cap\mathbb S^{N-1}\bigr)
 \setminus\bigcup_{i<j}B_{1/m}(\nu_{m,i})
 \qquad\text{for }j=1,\cdots,k_m.
\]
These sets form a Borel partition of the sphere, and
$|\nu-\nu_{m,j}|<1/m$ for $\nu\in A_{m,j}$.
Let $\theta_{m,j}:=\sigma(A_{m,j})\geq0$, so that
$\sum_{j=1}^{k_m}\theta_{m,j}=1$. Terms with $\theta_{m,j}=0$
are omitted from the sums below. Define
\[
 f_{R,m}(\zeta)=\sum_{j=1}^{k_m}\theta_{m,j}
                      e^{\lambda_*R\nu_{m,j}\cdot\zeta},\qquad
 Q_m=\bigoplus_{j=1}^{k_m}D_{\theta_{m,j}}(T_{0,R\nu_{m,j}}U_*).
\]
Thus $Q_m\geq0$.
Moreover, $f_{R,m}\rightarrow f_R$ uniformly on $[-1,1]^N$
as $m\rightarrow\infty$.
The translation multiplier formula gives
$J(Q_m)(k)=f_{R,m}(\xi(k))$. Since
$J(D_{L_\sigma(Re_1)}U_*)=L_\sigma(Re_1)$, we have
\[
 \|Q_m\ominus D_{L_\sigma(Re_1)}U_*\|_{U_*}
 \leq\|f_{R,m}-f_R\|_{C([-1,1]^N)}\rightarrow0.
\]
The subtraction inside this norm is subtraction in the ordered space.
Apply \eqref{eq:faithfulness-physical-jensen} to the finite average
$Q_m$ at $(t,0)$. It gives
\[
 Q_m(t,0)\geq\sum_{j=1}^{k_m}\theta_{m,j}U_*(t,R\nu_{m,j}).
\]
The left side tends to $D_{L_\sigma(Re_1)}U_*(t,0)$ as $m\rightarrow\infty$
by Lemma~\ref{lem:cone-unit-norm-local}, while
\[
 \sum_{j=1}^{k_m}\theta_{m,j}U_*(t,R\nu_{m,j})
 \rightarrow\int_{\mathbb S^{N-1}}U_*(t,R\nu)\,\sigma(d\nu)
 \qquad\text{as }m\rightarrow\infty.
\]
By \eqref{eq:critical-solution} with $a=L_\sigma(Re_1)$, one has
\[
 D_{L_\sigma(Re_1)}U_*(t,0)
 =U_*\left(t+\frac{\ln L_\sigma(Re_1)}{2f'(0)},0\right).
\]
This proves \eqref{eq:faithfulness-mean-bound}.
\end{proof}

We next use \eqref{eq:faithfulness-mean-bound} to prove that
$H_*[U_*]\not\equiv0$ for every entire solution $0<U_*<1$ satisfying
\eqref{eq:critical-solution}, where $H_*[U_*]$ is the small-amplitude
limit defined in \eqref{eq:pde-critical-trace}.
We use the solutions $U^{\rm rad}=u_\sigma$ and $u_\mu$ from
Theorems~\ref{thm:radial-critical-solution}
and~\ref{thm:canonical-critical-solutions}.

{
The lower bound from the trace and the radial upper bound now
identify every nonzero critical solution with its canonical
measure-generated solution.
}

\begin{theorem}
\label{thm:faithful-critical-trace}
Let $U_*$ be a nonzero critical solution in the sense of
\eqref{eq:critical-solution}.
Let $\mu_*:=\mu_*[U_*]$ be the measure from
Theorem~\ref{thm:pde-critical-trace}. With $L_{\mu_*}$ defined by
\eqref{eq:critical-measure-phase}, we have
\[
 H_*[U_*](t,x)=e^{2f'(0)t}L_{\mu_*}(x).
\]
Then $\mu_*\ne0$, and
\begin{equation}\label{eq:faithfulness-universal-upper}
 U_*(t,x)\leq U^{\rm rad}\left(t+\frac{\ln\bigl(\lambda_*L_{\mu_*}(x)\bigr)}{2f'(0)},0\right).
\end{equation}
In particular, for any entire solution $0\leq U_*<1$ satisfying
\eqref{eq:critical-solution},
\[
 H_*[U_*]\equiv0\quad\Rightarrow\quad U_*\equiv0.
\]
\end{theorem}
\begin{proof}
Applying Corollary~\ref{cor:trace-pure-critical} to $U_*$ at $x=0$ gives
\begin{equation}\label{eq:faithfulness-origin-asymptotic}
 \frac{U_*(s,0)}{(-2f'(0)s)e^{2f'(0)s}}\rightarrow L_{\mu_*}(0)
 \qquad\text{as }s\rightarrow-\infty.
\end{equation}
Fix any number $b>\lambda_*L_{\mu_*}(0)$. The critical tail asymptotics
\eqref{eq:intro-critical-expansion}
and \eqref{eq:faithfulness-origin-asymptotic} imply
$U_*(z/c_*,0)\leq\phi_*(z+\lambda_*^{-1}\ln b)$ for all sufficiently negative $z$.
Since $U_*$ is bounded, there exists $C_b>0$ such that
\begin{equation}\label{eq:faithfulness-scalar-envelope}
 U_*(z/c_*,0)\leq\phi_*(z+\lambda_*^{-1}\ln b)+C_be^{\lambda_*z}
 \qquad\hbox{for every }z\in\mathbb R.
\end{equation}

Define the spherical average
\[
 W(t,x)=\int_{\mathbb S^{N-1}}U_*(t,|x|\xi)\,d\sigma(\xi).
\]
It is a $C^{1,2}$ radially symmetric function and $W(t,0)=U_*(t,0)$.
By concavity of $f$, we have
\[
 W_t-\Delta W=\int_{\mathbb S^{N-1}}f(U_*(t,|x|\xi))\,d\sigma(\xi)\leq f(W),
\]
so $W$ is a subsolution. By
\eqref{eq:faithfulness-mean-bound}, one has
\[
 W(-T,x)\leq U_*\left(-T+\frac{\ln L_\sigma(x)}{2f'(0)},0\right).
\]
Apply \eqref{eq:faithfulness-scalar-envelope} with
$z=-c_*T+\lambda_*^{-1}\ln L_\sigma(x)$. Since $W\leq1$, the result can be written as
\[
 W(-T,x)\leq
 \phi_*\bigl(-c_*T+\lambda_*^{-1}\ln(bL_\sigma(x))\bigr)
 +\min\{C_be^{-2f'(0)T}L_\sigma(x),1\}.
\]
The comparison principle and semigroup subadditivity give at any fixed $(t,0)$ with $t>-T$,
\begin{align*}
 U_*(t,0)=W(t,0)
 &\leq \bigl(S_{t+T}[\phi_*(z_{b\sigma}(-T,\cdot))]\bigr)(0)\\
 &\quad+\bigl(S_{t+T}[\min\{C_be^{-2f'(0)T}L_\sigma,1\}]\bigr)(0).
\end{align*}
Let $T\to\infty$. The first term converges to
$u_{b\sigma}(t,0)$ by its definition. The second tends to zero by
\eqref{eq:general-critical-extinction} with $\mu=\sigma$. Thus
\[
 U_*(t,0)\leq u_{b\sigma}(t,0)
 =U^{\rm rad}(t+\tfrac{\ln b}{2f'(0)},0)
 \qquad\hbox{for every }b>\lambda_*L_{\mu_*}(0).
\]
If $L_{\mu_*}(0)=0$, letting $b\searrow0$ contradicts $U_*(t,0)>0$,
since $U^{\rm rad}(s,0)\to0$ as $s\to-\infty$.
Hence $L_{\mu_*}(0)>0$. Letting $b\searrow\lambda_*L_{\mu_*}(0)$ gives the
bound \eqref{eq:faithfulness-universal-upper} at $x=0$.

Applying the same argument to $U_*^x(t,y)=U_*(t,x+y)$ yields
\eqref{eq:faithfulness-universal-upper} at every $(t,x)$.
\end{proof}

\subsection{Identification and weak continuity}
\label{sec:critical-canonical-identification}

We use the critical trace to identify each critical solution with
the canonical solution associated with its measure, and then prove
weak continuity of this family. The argument starts by determining
the exact coefficient in the past asymptotics of the radial solution.

We next determine the precise asymptotic behaviour of
$U^{\rm rad}(t,0)$ as $t\rightarrow-\infty$.
We claim that
\begin{equation}\label{eq:faithfulness-radial-exact-tail}
 U^{\rm rad}(s,0)\sim(-c_*s)e^{2f'(0)s}\qquad \hbox{for }s\to-\infty.
\end{equation}
Let $\mu_*:=\mu_*[u_\sigma]$ be the measure from
Theorem~\ref{thm:pde-critical-trace}, so that
\[
 \lim_{a\searrow0}\frac{D_a u_\sigma(t,x)}{a\ln(1/a)}
 =H_*[u_\sigma](t,x)=e^{2f'(0)t}L_{\mu_*}(x).
\]
The radial symmetry of the solution passes to its limit.
Taking the spherical average and applying
\eqref{eq:canonical-spherical-laplace-average} with $\mu=\mu_*$ gives
\[
 L_{\mu_*}(x)=\int_{\mathbb S^{N-1}}L_{\mu_*}(|x|\xi)\,d\sigma(\xi)
 =\mu_*(\mathbb S^{N-1})L_\sigma(x)
 =L_{\mu_*}(0)L_\sigma(x).
\]
By \eqref{eq:canonical-critical-dilation},
$D_a u_\sigma(0,0)=u_\sigma(\tfrac{\ln a}{2f'(0)},0)$, so
\[
 \lim_{s\to-\infty}
 \frac{u_\sigma(s,0)}{(-2f'(0)s)e^{2f'(0)s}}=L_{\mu_*}(0).
\]
Since $L_\sigma(0)=1$, Theorem~\ref{thm:radial-critical-solution} gives
$u_\sigma(s,0)\geq\phi_*(c_*s)$, hence $L_{\mu_*}(0)\geq1/\lambda_*$.
By \eqref{eq:pde-trace-lower-barrier} and
$L_{\mu_*}(x)=L_{\mu_*}(0)L_\sigma(x)$,
\[
 u_\sigma(t,x)\geq
 \phi_*\left(\lambda_*^{-1}\ln\bigl(\lambda_*H_*[u_\sigma](t,x)\bigr)\right)
 =\phi_*\left(z_{\lambda_*L_{\mu_*}(0)\sigma}(t,x)\right).
\]
If $L_{\mu_*}(0)>1/\lambda_*$, minimality and
\eqref{eq:radial-mass-time-scaling} therefore give
\[
 u_\sigma\geq u_{\lambda_*L_{\mu_*}(0)\sigma}
 =u_\sigma(t+\tfrac{\ln(\lambda_*L_{\mu_*}(0))}{2f'(0)},x)>u_\sigma,
\]
contradicting strict time monotonicity. Thus $L_{\mu_*}(0)=1/\lambda_*$.

With the radial tail coefficient fixed, the lower critical profile
and the radial upper bound determine the same entire solution.
The following uniqueness statement will identify the critical component.

\begin{lemma}
\label{lem:faithfulness-envelope-absorption}
Let $\mu$ be a finite positive Borel measure on $\mathbb S^{N-1}$
and let $Z$ be an entire solution of \eqref{eq:kpp} satisfying
\[
 \phi_*(z_\mu(t,x))\leq Z(t,x)
 \leq U^{\rm rad}(z_\mu(t,x)/c_*,0).
\]
Then $Z=u_\mu$.
\end{lemma}
\begin{proof}
Minimality of the canonical entire solution gives $Z\geq u_\mu$.
Fix $b>1$. By \eqref{eq:faithfulness-radial-exact-tail}, there is
$C_b<\infty$ such that, for all $z\in \R$,
\[
 U^{\rm rad}(z/c_*,0)\leq\phi_*(z+\lambda_*^{-1}\ln b)+C_be^{\lambda_*z}.
\]
Indeed, \eqref{eq:faithfulness-radial-exact-tail} and
\eqref{eq:intro-critical-expansion} give
\[
 \lim_{z\to-\infty}
 \frac{\phi_*(z+\lambda_*^{-1}\ln b)}{U^{\rm rad}(z/c_*,0)}=b>1.
\]
Thus there exists $z_b\in\R$ such that
$U^{\rm rad}(z/c_*,0)\leq\phi_*(z+\lambda_*^{-1}\ln b)$ for $z\leq z_b$.
For $z\geq z_b$, choosing $C_b=e^{-\lambda_*z_b}$ gives
$U^{\rm rad}(z/c_*,0)\leq1\leq C_be^{\lambda_*z}$.
Since $Z\leq1$, it follows that
\[
 Z(-T,x)\leq\phi_*(z_{b\mu}(-T,x))
 +\min\{C_b e^{-2f'(0)T}L_\mu(x),1\}.
\]
Applying the comparison principle and semigroup subadditivity, letting
$T\to\infty$, and using
\eqref{eq:general-critical-extinction} give $Z\leq u_{b\mu}$.
Finally, letting $b\searrow1$ and using
\eqref{eq:canonical-critical-dilation} and continuity gives
$Z\leq u_\mu$. This proves equality. 
\end{proof}

\begin{theorem}
\label{thm:faithfulness-critical-classification}
Every nonzero critical solution $U_*$ in the sense of
\eqref{eq:critical-solution} has the form
\[
 U_*=u_\mu,\qquad \mu=\lambda_*\mu_*[U_*],
\]
where $\mu$ is a  finite positive Borel measure on $\mathbb S^{N-1}$.
\end{theorem}
\begin{proof}
Theorem~\ref{thm:faithful-critical-trace} gives $\mu_*[U_*]\ne0$. The lower bound in Theorem~\ref{thm:pde-critical-trace} reads
\[
 U_*\geq\phi_*\left(\lambda_*^{-1}\ln\bigl(\lambda_*H_*[U_*]\bigr)\right)=\phi_*(z_\mu).
\]
The universal radial upper bound
\eqref{eq:faithfulness-universal-upper} reads
$U_*\leq U^{\rm rad}(z_\mu/c_*,0)$. Hence
Lemma~\ref{lem:faithfulness-envelope-absorption} gives $U_*=u_\mu$.
\end{proof}

{
This identification also gives continuity under weak convergence of
measures, allowing approximation by finite atomic measures.
}

\begin{proposition}
\label{prop:faithfulness-weak-continuity}
If finite nonnegative measures $\mu_n$ on $\mathbb S^{N-1}$ converge
weakly to a finite measure $\mu$, then
\[
 u_{\mu_n}\rightarrow u_\mu
 \quad\text{in }C^{1,2}_{\mathrm{loc}}(\R\times\R^N).
\]
Here $u_0=0$. In particular, every critical solution $U_*$ is a
local $C^{1,2}$ limit of canonical solutions associated with nonzero
finite atomic angular measures. The corresponding finite HN parameters
are identified in the proof of Theorem~\ref{thm:main-closure}.
\end{proposition}
\begin{proof}
Weak convergence on the compact sphere gives
\[
 \sup_n\mu_n(\mathbb S^{N-1})<\infty,\qquad
 L_{\mu_n}\rightarrow L_\mu
 \quad\text{in }C^\infty_{\mathrm{loc}}(\R^N).
\]
If $\mu=0$, then for every $R>0$ and every $n$ with $\mu_n\ne0$,
\[
 \sup_{\substack{|t|\leq R\\|x|\leq R}}u_{\mu_n}(t,x)
 \leq U^{\rm rad}\left(
 R+\frac{\ln\bigl(e^{\lambda_*R}\mu_n(\mathbb S^{N-1})\bigr)}
 {2f'(0)},0\right)\rightarrow0.
\]
Here we used \eqref{eq:canonical-critical-radial-upper} and
$L_{\mu_n}(x)\leq e^{\lambda_*R}\mu_n(\mathbb S^{N-1})$ for $|x|\leq R$.
Since $u_0=0$, parabolic estimates yield the asserted convergence.

If $\mu\ne0$, then $L_\mu>0$ and
\[
 z_{\mu_n}\rightarrow z_\mu
 \quad\text{in }C^\infty_{\mathrm{loc}}(\R\times\R^N).
\]
By parabolic compactness, one has, up to extraction of a subsequence,
\[
 u_{\mu_{n_j}}\rightarrow Z
 \quad\text{in }C^{1,2}_{\mathrm{loc}}(\R\times\R^N),\qquad
 \phi_*(z_\mu)\leq Z\leq U^{\rm rad}(z_\mu/c_*,0).
\]
Lemma~\ref{lem:faithfulness-envelope-absorption} gives $Z=u_\mu$,
so the whole sequence converges.

For $U_*\ne0$, Theorem~\ref{thm:faithfulness-critical-classification}
gives $U_*=u_\mu$. By finite Borel partitions of the sphere, choose finite atomic measures
\[
 \mu_k=\sum_{i=1}^{m_k}b_{k,i}\delta_{\nu_{k,i}}
 \rightharpoonup\mu,\qquad
 b_{k,i}>0,\quad \nu_{k,i}\in\mathbb S^{N-1}.
\]
Then $u_{\mu_k}\to U_*$ in $C^{1,2}_{\mathrm{loc}}$.
For $U_*=0$, take $\mu_k=k^{-1}\delta_\nu$ with fixed
$\nu\in\mathbb S^{N-1}$.
\end{proof}

\begin{proof}[Proof of Theorem~\ref{thm:borel-main}]
Existence and \eqref{eq:borel-main-squeeze} follow from
Theorem~\ref{thm:canonical-critical-solutions}.
Theorem~\ref{thm:gkpp-ancient} gives strict time monotonicity and
the limits at $t=\pm\infty$; interior parabolic estimates upgrade
these limits to \eqref{eq:borel-main-extinction}.
For $\alpha\in(0,1)$, let $s_\alpha$ be the unique real number with
$U^{\rm rad}(s_\alpha,0)=\alpha$. The two bounds in
\eqref{eq:canonical-critical-radial-upper} imply
\[
 u_\mu(t,x)=\alpha
 \quad\Rightarrow\quad
 c_*s_\alpha\leq z_\mu(t,x)\leq\phi_*^{-1}(\alpha).
\]
Taking $C_\alpha=1+\max\{|c_*s_\alpha|,|\phi_*^{-1}(\alpha)|\}$
proves \eqref{eq:borel-main-level-strip}.

By \eqref{eq:intro-critical-tail},
\eqref{eq:faithfulness-radial-exact-tail}, and the same two bounds,
\[
 \frac{u_\mu(t,x)}{(-c_*t)e^{2f'(0)t}}
 \rightarrow L_\mu(x)
 \qquad\text{as }t\to-\infty
 \quad\text{for every }x\in\R^N.
\]
Thus $u_\mu=u_\nu$ implies $L_\mu=L_\nu$, and uniqueness of the
Laplace transform gives $\mu=\nu$.
Weak continuity follows from
Proposition~\ref{prop:faithfulness-weak-continuity}, completing the proof.
\end{proof}

\section{The Hamel--Nadirashvili closure theorem}
\label{sec:final-native-closure}
\label{sec:final-original-closure}

We combine the decomposition and the critical measure identification
to prove Theorem~\ref{thm:main-closure}. We first control mixed
approximation at a fixed strict gradient bound, identify the finite
atomic constructions as HN solutions, and then choose a diagonal
sequence converging to the original entire solution.

The reaction remains the fixed function satisfying
\eqref{eq:general-kpp-assumptions}, with $f'(0)>0$.
We retain $S_\tau$ for the nonlinear Cauchy semigroup and
$L_\tau=e^{f'(0)\tau}P_\tau$ for the positive linear semigroup.
Write
\[
 \mathcal Q_R=[-R,R]\times\overline B_R(0),\qquad
 \|v\|_{2,1;R}=\sup_{\mathcal Q_R}
 \left(|v|+|v_t|+\sum_i|v_{x_i}|+\sum_{i,j}|v_{x_ix_j}|\right).
\]

\subsection{Mixed approximation at a fixed spectral radius}

For a fixed $r<1$, we prove that weak approximation of the critical
measure and of the measure supported in $\overline B_r$ passes to
their nonlinear superposition. This permits simultaneous finite
atomic approximation of the two components.

For a nonnegative finite measure $\eta$ on $\mathbb S^{N-1}$, let
$u_{\eta}$ be the entire solution given by
Theorem~\ref{thm:canonical-critical-solutions} when $\eta\ne0$, and set $u_0=0$.
We use $L_\eta$ and $z_\eta$ from \eqref{eq:critical-measure-phase}
and the construction \eqref{eq:canonical-critical-pullback}.
Theorem~\ref{thm:canonical-critical-solutions} and the ancient estimates give
\begin{align}
 u_{\eta}(t,x)
 &\leq U^{\rm rad}\left(t+\frac{\ln L_\eta(x)}{2f'(0)},0\right),
                                                        \label{eq:final-critical-radial}\\
 D_au_{\eta}(t,x)
 &=u_{\eta}(t+\tfrac{\ln a}{2f'(0)},x),\qquad
 |\nabla\ln u_{\eta}|\leq\lambda_*                 \label{eq:final-critical-covariance}
\end{align}
for $\eta\ne0$.
For a finite nonnegative measure
$\zeta$ supported in $\overline B_r$, where $0\leq r<1$, put
\[
 H_\zeta(t,x)=\int_{\overline B_r}
 e^{\lambda_*x\cdot\xi+f'(0)(1+|\xi|^2)t}\,\zeta(d\xi).
\]
Let $V_{H_\zeta}$ be the entire solution reconstructed in
Theorem~\ref{thm:general-finite-lift}, with $V_{H_0}=0$.
That theorem gives $0\leq V_{H_\zeta}\leq\min\{H_\zeta,1\}$ and continuity into
$C^{1,2}_{\mathrm{loc}}(\R\times\R^N)$ under weak convergence of
measures on this fixed ball.

We first show that the maximum and capped sum of the two components
have the same pullback limit. The quantitative bound is uniform for
bounded measure masses at a fixed $r<1$, which will allow passage
to weak limits of both measures.

\begin{lemma}
\label{lem:final-uniform-mixed-heat}
Let $0\leq r<1$, and let $\eta,\zeta$ be finite nonnegative Borel measures
on $\mathbb S^{N-1}$ and $\overline B_r$, respectively. Set
\[
 \underline u=\max\{u_{\eta},V_{H_\zeta}\},\qquad \overline u=\min\{u_{\eta}+V_{H_\zeta},1\}.
\]
For $R\geq1$ and $T\geq2R+2$,
\begin{equation}\label{eq:final-uniform-mixed-gap}
 \sup_{(t,x)\in[-R,R]\times\overline B_R(0)}
 \bigl(L_{T+t}[\overline u(-T,\cdot)-\underline u(-T,\cdot)]\bigr)(x)
 \leq K_R T^{(N+1)/2}
       \exp\left(-\frac{f'(0)(1-r)^2T}{4}\right).
\end{equation}
where $K_R$ depends only on $f,N,r,R$, the fixed wave normalization, and upper bounds for
$\eta(\mathbb S^{N-1})$ and $\zeta(\overline B_r)$.
Moreover, as $T\to\infty$,
\[
 \begin{aligned}
 S_{t+T}[\underline u(-T,\cdot)]&\nearrow u_\eta\oplus V_{H_\zeta},\\
 S_{t+T}[\overline u(-T,\cdot)]&\searrow u_\eta\oplus V_{H_\zeta}
 \end{aligned}
 \quad\text{in }C^{1,2}_{\mathrm{loc}}(\R\times\R^N).
\]
Furthermore,
\[
 W\in\mathcal E,\quad\underline u\leq W\leq\overline u
 \quad\Longrightarrow\quad W=u_\eta\oplus V_{H_\zeta}.
\]
\end{lemma}
\begin{proof}
If either $\eta$ or $\zeta$ is zero, then $\underline u=\overline u$ and all conclusions are
immediate. Otherwise, by \eqref{eq:final-critical-radial},
\[
 u_{\eta}(0,0)\leq
 U^{\rm rad}\left(\frac{\ln\max\{\eta(\mathbb S^{N-1}),1\}}{2f'(0)},0\right)
 =:\theta<1.
\]
By \eqref{eq:inverse-wave-upper} and
\eqref{eq:final-critical-covariance}, one has
\[
 u_{\eta}(-T,0)=D_{e^{-2f'(0)T}}u_{\eta}(0,0)
 \leq\phi_*\left(-c_*T+\phi_*^{-1}(\theta)\right)
 \leq K(1+T)e^{-2f'(0)T}\qquad \hbox{for }T\geq1.
\]
Here $K$ depends only on $\theta$ and $\phi_*$.
By $|\nabla\ln u_{\eta}|\leq\lambda_*$, we have
\begin{equation}\label{eq:final-critical-global-past}
 u_{\eta}(-T,y)\leq K(1+T)e^{-2f'(0)T+\lambda_*|y|}
 \qquad\text{for }y\in\mathbb R^N,\ T\geq1.
\end{equation}
By $V_{H_\zeta}\leq H_\zeta$ and the definition of $H_\zeta$,
\begin{equation}\label{eq:final-subcritical-global-past}
 \begin{aligned}
 V_{H_\zeta}(-T,y)&\leq H_\zeta(-T,y)
 =\int_{\overline B_r}e^{\lambda_*y\cdot\xi-f'(0)(1+|\xi|^2)T}\,d\zeta(\xi)\\
 &\leq\zeta(\overline B_r)
 \sup_{0\leq a\leq r}\exp\{\lambda_*a|y|-f'(0)(1+a^2)T\}.
 \end{aligned}
\end{equation}

Since $0\leq u_\eta,V_{H_\zeta}\leq1$,
\[
 0\leq\overline u-\underline u
 \leq\min\{u_\eta,V_{H_\zeta}\}\leq\sqrt{u_\eta V_{H_\zeta}}.
\]
For $|t|,|x|\leq R$ and $T\geq2R+2$, polar integration with
$\rho=|y-x|$, $|y|\leq R+\rho$, and
\eqref{eq:final-critical-global-past}--\eqref{eq:final-subcritical-global-past} give
\[
 \begin{aligned}
 &\bigl(L_{T+t}[(\overline u-\underline u)(-T,\cdot)]\bigr)(x)\\
 &\quad\leq \bigl(L_{T+t}[\sqrt{u_\eta(-T,\cdot)V_{H_\zeta}(-T,\cdot)}]\bigr)(x)\\
 &\quad\leq K_R(1+T)^{1/2}(T+t)^{-N/2}
 \int_0^\infty\rho^{N-1}\sup_{0\leq a\leq r}e^{E_a(\rho)}\,d\rho,
 \end{aligned}
\]
where, with $I_{T+t}=[\lambda_*(T+t),\lambda_*(1+r)(T+t)]$,
\begin{align*}
 E_a(\rho)
 &=-\frac{\rho^2}{4(T+t)}
       +\frac{\lambda_*(1+a)\rho}{2}
       -\frac{f'(0)(3+a^2)T}{2}+f'(0)(T+t)\\
 &=-\frac{(\rho-\lambda_*(1+a)(T+t))^2}{4(T+t)}
       -\frac{f'(0)(1-a)^2T}{4}
       +f'(0)t\left(1+\frac{(1+a)^2}{4}\right)\\
 &\leq-\frac{\operatorname{dist}(\rho,I_{T+t})^2}{4(T+t)}
       -\frac{f'(0)(1-r)^2T}{4}+2f'(0)R.
\end{align*}
Since $\operatorname{dist}(\rho,I_{T+t})\geq\rho/2$ for
$\rho\geq4\lambda_*(T+t)$,
\begin{equation}\label{eq:final-radial-gaussian-integral}
 \begin{aligned}
 &\int_0^\infty\rho^{N-1}
 \exp\left(-\frac{\operatorname{dist}(\rho,I_{T+t})^2}
                     {4(T+t)}\right)d\rho\\
 &\quad\leq\frac{(4\lambda_*(T+t))^N}{N}
 +\int_0^\infty\rho^{N-1}e^{-\rho^2/(16(T+t))}\,d\rho\\
 &\quad\leq K_{N,f}(T+t)^N.
 \end{aligned}
\end{equation}
Together with $T/2\leq T+t\leq3T/2$, these estimates prove
\eqref{eq:final-uniform-mixed-gap}.

The functions $\underline u,\overline u$ are respectively a 
subsolution and supersolution, since
\[
 f(u_\eta)+f(V_{H_\zeta})-f(u_\eta+V_{H_\zeta})
 \geq0
 \qquad\text{by concavity of }f\text{ and }f(0)=0.
\]
The comparison principle gives
\[
 \begin{aligned}
 0&\leq S_{T+t}[\overline u(-T,\cdot)]-S_{T+t}[\underline u(-T,\cdot)]\\
 &\leq L_{T+t}[(\overline u-\underline u)(-T,\cdot)]
 \rightarrow0
 \end{aligned}
\]
locally uniformly as $T\to\infty$. Since $\underline u\leq u_\eta\oplus V_{H_\zeta}\leq\overline u$,
comparison traps this entire solution between the two evolutions.
Their difference tends to zero. By monotonicity and parabolic estimates,
\[
 \begin{aligned}
 S_{T+t}[\underline u(-T,\cdot)]&\nearrow u_\eta\oplus V_{H_\zeta},\\
 S_{T+t}[\overline u(-T,\cdot)]&\searrow u_\eta\oplus V_{H_\zeta}
 \end{aligned}
 \quad\text{in }C^{1,2}_{\mathrm{loc}}(\R\times\R^N)
 \quad\text{as }T\to\infty.
\]
For $W\in\mathcal E$ with $\underline u\leq W\leq\overline u$,
\[
 S_{T+t}[\underline u(-T,\cdot)]\leq W(t,\cdot)
 \leq S_{T+t}[\overline u(-T,\cdot)].
\]
Letting $T\to\infty$ yields $W=u_\eta\oplus V_{H_\zeta}$.
\end{proof}

{
At a fixed strict gradient gap, the preceding comparison estimates
allow weak approximation of the critical and subcritical measures
simultaneously.
}

\begin{proposition}
\label{prop:final-mixed-continuity}
Fix $0\leq r<1$. Suppose $\eta_n\rightharpoonup\eta$ on
$\mathbb S^{N-1}$ and $\zeta_n\rightharpoonup\zeta$ on
$\overline B_r$, all measures being finite and nonnegative. Then
\[
 u_{\eta_n}\oplus V_{H_{\zeta_n}}
 \rightarrow u_{\eta}\oplus V_{H_\zeta}
 \quad\hbox{in }C^{1,2}_{\mathrm{loc}}(\R\times\R^N).
\]
\end{proposition}
\begin{proof}
By Proposition~\ref{prop:faithfulness-weak-continuity} and
Theorem~\ref{thm:general-finite-lift},
\[
 u_{\eta_n}\rightarrow u_\eta,\qquad
 V_{H_{\zeta_n}}\rightarrow V_{H_\zeta}
 \quad\text{in }C^{1,2}_{\mathrm{loc}}(\R\times\R^N).
\]
Since $0\leq u_{\eta_n}\oplus V_{H_{\zeta_n}}\leq1$,
parabolic compactness gives, up to extraction of a subsequence,
\[
 u_{\eta_{n_j}}\oplus V_{H_{\zeta_{n_j}}}\rightarrow W
 \quad\text{in }C^{1,2}_{\mathrm{loc}}(\R\times\R^N),
 \qquad W\in\mathcal E.
\]
Passing to the limit in
\[
 \max\{u_{\eta_n},V_{H_{\zeta_n}}\}
 \leq u_{\eta_n}\oplus V_{H_{\zeta_n}}
 \leq\min\{u_{\eta_n}+V_{H_{\zeta_n}},1\}
\]
yields
\[
 \max\{u_\eta,V_{H_\zeta}\}\leq W
 \leq\min\{u_\eta+V_{H_\zeta},1\}.
\]
Lemma~\ref{lem:final-uniform-mixed-heat} gives
$W=u_\eta\oplus V_{H_\zeta}$. Thus the whole sequence converges.
\end{proof}

\subsection{Proof of the closure theorem}
\label{subsec:final-original-parameters}
\label{subsec:final-native-parameters}

We now apply the preceding results to an arbitrary entire solution.
After fixing one approximation of its noncritical component, we
approximate both measures by finite atomic measures and then choose
a diagonal sequence to recover the full solution.

\begin{proof}[Proof of Theorem~\ref{thm:main-closure}]
For $U\equiv0$ or $U\equiv1$, use
\[
 \xi_f(t-j)\rightarrow0,\qquad \xi_f(t+j)\rightarrow1
 \quad\text{in }C^{1,2}_{\mathrm{loc}}(\R\times\R^N).
\]
By the strong maximum principle, it remains to consider $0<U<1$.

Proposition~\ref{prop:cutoff-critical-subcritical-decomposition} gives $U_*,Z,A_j\in\mathcal C$ with
\[
 U=Z\oplus U_*,\qquad
 A_j\nearrow Z\quad\text{in }C^{1,2}_{\mathrm{loc}}(\R\times\R^N),
\]
\[
 |\nabla\ln A_j|\leq\lambda_*r_j,\quad 0\leq r_j<1
 \quad\text{if }A_j\ne0,
\]
\[
 D_aU_*(t,x)=U_*\left(t+\frac{\ln a}{2f'(0)},x\right)
 \quad\text{for }a>0.
\]
By Theorems~\ref{thm:gradient-gap-finite-lift},
\ref{thm:general-finite-lift}, and
\ref{thm:faithfulness-critical-classification}, there are finite nonnegative
Borel measures $\zeta_j,\eta$ such that
\[
 A_j=V_{H_{\zeta_j}},\qquad
 \operatorname{supp}\zeta_j\subset\overline B_{r_j},\qquad
 U_*=u_\eta,\qquad \operatorname{supp}\eta\subset\mathbb S^{N-1}.
\]
For $A_j=0$, take $\zeta_j=0$ and $r_j=0$; for $U_*=0$, take $\eta=0$.
Corollary~\ref{cor:cone-increasing-superposition} yields
\begin{equation}\label{eq:final-component-addition-limit}
 U_j:=A_j\oplus U_*\nearrow U
 \quad\text{in }C^{1,2}_{\mathrm{loc}}(\R\times\R^N).
\end{equation}
Since $U(0,0)>0$, we may discard finitely many terms so that $U_j\ne0$.

For each fixed $j$, choose finite atomic measures satisfying
\[
 \begin{aligned}
 \zeta_{j,n}&\rightharpoonup\zeta_j,
 &\zeta_{j,n}(\overline B_{r_j})&=\zeta_j(\overline B_{r_j}),\\
 \eta_n&\rightharpoonup\eta,
 &\eta_n(\mathbb S^{N-1})&=\eta(\mathbb S^{N-1}),\\
 \operatorname{supp}\zeta_{j,n}&\subset\overline B_{r_j},
 &\operatorname{supp}\eta_n&\subset\mathbb S^{N-1}.
 \end{aligned}
\]
These are obtained by placing each cell's mass at one point of a finite
Borel partition of diameter at most $1/n$; zero measures remain zero.
Proposition~\ref{prop:final-mixed-continuity} gives
\[
 F_{j,n}:=V_{H_{\zeta_{j,n}}}\oplus u_{\eta_n}
 \rightarrow U_j
 \quad\text{in }C^{1,2}_{\mathrm{loc}}(\R\times\R^N)
 \quad\text{as }n\to\infty.
\]
We now identify each $F_{j,n}$ as a finite HN solution. Fix $(j,n)$
and temporarily write $\zeta=\zeta_{j,n}$ and $\eta=\eta_n$.
Let
\[
 \zeta=a_0\delta_0+\sum_{i=1}^{p}a_i\delta_{\xi_i},
 \qquad \eta=\sum_{\ell=1}^{k}b_j\delta_{\nu_j},
\]
where $a_0\geq0$, $a_i,b_j>0$, $0<|\xi_i|<1$, $|\nu_j|=1$,
and the atoms within each sum are distinct. Assume
$a_0+\sum_i a_i+\sum_jb_j>0$.
Set
\[
 \rho_i:=|\xi_i|,\qquad e_i=\frac{\xi_i}{|\xi_i|},\qquad
 c_i=\lambda_*(|\xi_i|+|\xi_i|^{-1})>c_*,
 \qquad h_i=\frac{\ln a_i}{\lambda_*|\xi_i|}.
\]
Use the normalized homogeneous orbit $\xi_f$ from
\eqref{eq:homogeneous-normalization}. The corresponding profiles are
\[
 v_i(t,x)=\phi_{c_i}(x\cdot e_i+c_it+h_i),\qquad
 v_0(t)=\xi_f\left(t+\frac{\ln a_0}{f'(0)}\right)
 \quad\text{if }a_0>0,
\]
and
\begin{equation}\label{eq:final-original-critical-phase}
 w_j(t,x)=\phi_*\left(x\cdot\nu_j+c_*t+\lambda_*^{-1}\ln b_j\right).
\end{equation}
Here
\[
 \phi_{c_i}(s)\sim e^{\lambda_*|\xi_i|s},\qquad
 \xi_f(s)\sim e^{f'(0)s}\qquad\text{as }s\to-\infty.
\]
In all maxima and sums below, $v_0$ is included only if $a_0>0$,
and absent groups are omitted.

By Theorem~\ref{thm:general-finite-lift},
\[
 \max_i v_i\leq V_{H_\zeta}\leq\min\left\{\sum_i v_i,1\right\}.
\]
For $k\geq1$, the subadditivity of
$g(q)=\phi_*(\lambda_*^{-1}\ln q)$, with $g(0)=0$, gives
\[
 \max_j w_j\leq g(H_\eta)\leq\sum_jw_j.
\]
Comparison in \eqref{eq:canonical-critical-pullback} therefore yields
\[
 \max_j w_j\leq u_\eta\leq\min\left\{\sum_jw_j,1\right\}.
\]
The definition of $\oplus$ and comparison imply
\[
 \max\{\max_i v_i,\max_j w_j\}
 \leq V_{H_\zeta}\oplus u_\eta
 \leq\min\left\{\sum_i v_i+\sum_jw_j,1\right\}.
\]
By Theorem~\ref{thm:packet}, $V_{H_\zeta}\oplus u_\eta$ is the HN solution
associated with the profiles $v_i$ and $w_j$, including $v_0$ when $a_0>0$.
The associated speed--phase pairs are
\[
 \left(c_i,\frac{\ln a_i}{\lambda_*|\xi_i|}\right),\qquad
 \left(c_*,\lambda_*^{-1}\ln b_j\right),
\]
with homogeneous time shift $\ln a_0/f'(0)$ if $a_0>0$.

Here is also the precise original measure corresponding to this
solution. Our normalized critical profile $\phi_*$ agrees with the
physical-variable normalization of \cite{HN2001}. Thus the original
critical phases are $\lambda_*^{-1}\ln b_j$. The original paper assigns phase
$c_*\ln m_j$ to a critical atom of mass $m_j$ and uses
$u_{\alpha\mathfrak m}(t,x)=u_{\mathfrak m}(t+\ln\alpha,x)$;
see \cite[Theorem~1.2(i),(iii) and Section~3.1]{HN2001}.
Consequently the original critical mass is
$m_j=b_j^{1/(2f'(0))}$. This is distinct from the factor
$1/\lambda_*$ in the small-amplitude trace of the canonical angular family.

If $a_0+\sum_i a_i>0$, let $M>0$ be the unique solution of
\begin{equation}\label{eq:final-native-mass-equation}
 \frac{a_0}{M^{f'(0)}}
 +\sum_{i=1}^{p}\frac{a_i}{M^{f'(0)(1+\rho_i^2)}}=1.
\end{equation}
The left side is continuous and strictly decreasing from $+\infty$
to $0$, so $M$ exists uniquely. Define
\begin{equation}\label{eq:final-native-measure}
 \mathfrak m=
 \sum_{\ell=1}^{k}b_j^{1/(2f'(0))}\,\delta_{(\nu_j,c_*)}
 +\sum_{i=1}^{p}a_iM^{1-f'(0)(1+\rho_i^2)}\delta_{(e_i,c_i)}
 +a_0M^{1-f'(0)}\delta_{\infty}.
\end{equation}
When the noncritical list is empty, omit $M$ and the last two terms.
By \eqref{eq:final-native-mass-equation},
\[
 \begin{aligned}
 \mathfrak m\bigl(X\setminus(\mathbb S^{N-1}\times\{c_*\})\bigr)
 &=\sum_{i=1}^p a_iM^{1-f'(0)(1+\rho_i^2)}
   +a_0M^{1-f'(0)}=M,\\
 \mathfrak m(\{(e_i,c_i)\})M^{\lambda(c_i)c_i-1}&=a_i,
 \qquad \mathfrak m(\{\infty\})M^{f'(0)-1}=a_0.
 \end{aligned}
\]
Here $\lambda(c_i)c_i=f'(0)(1+\rho_i^2)$. Thus the parameter
conversion in \cite[Section~3.9, Lemma~3.7]{HN2001} and uniqueness
in Theorem~\ref{thm:packet} give
\[
 V_{H_\zeta}\oplus u_\eta=u_{\mathfrak m},\qquad
 \mathfrak m\in\mathcal M.
\]

The masses entering Lemma~\ref{lem:final-uniform-mixed-heat} are
\[
 \eta(\mathbb S^{N-1})=\sum_jb_j,\qquad
 \zeta(\overline B_1)=a_0+\sum_i a_i,
\]
not the original critical mass
$\sum_j b_j^{1/(2f'(0))}$. Each $\mathfrak m$ is finite and belongs
to $\mathcal M$; no uniform bound on $\mathfrak m(X)$ is required
along the approximating sequence.

Thus every $F_{j,n}$ is a finite HN solution. Choose $n(j)$ such that
\[
 \|F_{j,n(j)}-U_j\|_{C^{1,2}([-j,j]\times\overline B_j(0))}<1/j,
 \qquad F_j:=F_{j,n(j)}.
\]
For every fixed $R\geq1$ and $j\geq R$,
\[
 \begin{aligned}
 \|F_j-U\|_{C^{1,2}([-R,R]\times\overline B_R(0))}
 &\leq\|F_j-U_j\|_{C^{1,2}([-j,j]\times\overline B_j(0))}\\
 &\quad+\|U_j-U\|_{C^{1,2}([-R,R]\times\overline B_R(0))}\\
 &<1/j+\|U_j-U\|_{C^{1,2}([-R,R]\times\overline B_R(0))}
 \rightarrow0.
 \end{aligned}
\]
By \eqref{eq:final-native-measure}, $F_j=u_{\vartheta_j}$ for a nonzero
finite sum of Dirac measures $\vartheta_j$ on
$X=\mathbb S^{N-1}\times[c_*,\infty)\cup\{\infty\}$.
Thus $u_{\vartheta_j}\rightarrow U$ in $C^{1,2}_{\mathrm{loc}}$,
and hence in the topology of \cite{HN2001}.

The reverse inclusion follows from parabolic regularity and passage
to the limit in \eqref{eq:kpp}.
\end{proof}


\begin{thebibliography}{99}
\small

\bibitem{AgmonNirenberg1963}
S.~Agmon and L.~Nirenberg,
\newblock Properties of solutions of ordinary differential equations in a Banach space,
\newblock \emph{Comm. Pure Appl. Math.} \textbf{16} (1963), 121--239.

\bibitem{AronsonWeinberger1978}
D.~G.~Aronson and H.~F.~Weinberger,
\newblock Multidimensional nonlinear diffusions arising in population genetics,
\newblock \emph{Adv. Math.} \textbf{30} (1978), 33--76.

\bibitem{BarenblattZeldovich1971}
G.~I.~Barenblatt and Ya.~B.~Zel'dovich,
\newblock Intermediate asymptotics in mathematical physics,
\newblock \emph{Uspekhi Mat. Nauk} \textbf{26} (1971), 115--129;
English transl., \emph{Russian Math. Surveys} \textbf{26} (1971), 45--61.

\bibitem{BerestyckiHamel2002}
H.~Berestycki and F.~Hamel,
\newblock Front propagation in periodic excitable media,
\newblock \emph{Comm. Pure Appl. Math.} \textbf{55} (2002), 949--1032.

\bibitem{BerestyckiHamel2007}
H.~Berestycki and F.~Hamel,
\newblock Generalized travelling waves for reaction-diffusion equations,
\newblock \emph{Contemp. Math.} \textbf{446} (2007), 101--124.

\bibitem{BerestyckiHamel2012}
H.~Berestycki and F.~Hamel,
\newblock Generalized transition waves and their properties,
\newblock \emph{Comm. Pure Appl. Math.} \textbf{65} (2012), 592--648.

\bibitem{BerestyckiNirenberg1992}
H.~Berestycki and L.~Nirenberg,
\newblock Travelling fronts in cylinders,
\newblock \emph{Ann. Inst. H. Poincar\'e Anal. Non Lin\'eaire} \textbf{9} (1992), 497--572.

\bibitem{BoutillonHamelRoques2026}
N.~Boutillon, F.~Hamel, and L.~Roques,
\newblock Periodic KPP equations: new insights into persistence, spreading, and the role of advection,
\newblock \emph{J. Differential Equations} \textbf{463} (2026), Paper No.~114174.

\bibitem{Bramson1983}
M.~Bramson,
\newblock Convergence of solutions of the Kolmogorov equation to travelling waves,
\newblock American Mathematical Society, Providence, RI, 1983.

\bibitem{ChenGuo2005}
X.~Chen and J.-S.~Guo,
\newblock Existence and uniqueness of entire solutions for a reaction--diffusion equation,
\newblock \emph{J. Differential Equations} \textbf{212} (2005), 62--84.

\bibitem{ChenGuoNinomiya2006}
X.~Chen, J.-S.~Guo, and H.~Ninomiya,
\newblock Entire solutions of reaction--diffusion equations with balanced bistable nonlinearities,
\newblock \emph{Proc. R. Soc. Edinburgh Sect. A} \textbf{136} (2006), 1207--1237.

\bibitem{ChenGuoNinomiyaYao2018}
Y.-Y.~Chen, J.-S.~Guo, H.~Ninomiya, and C.-H.~Yao,
\newblock Entire solutions originating from monotone fronts to the Allen--Cahn equation,
\newblock \emph{Physica D} \textbf{378--379} (2018), 1--19.

\bibitem{CoddingtonLevinson1955}
E.~A.~Coddington and N.~Levinson,
\newblock Theory of Ordinary Differential Equations,
\newblock McGraw--Hill, New York, 1955.

\bibitem{DengDucrot2023}
L.~Deng and A.~Ducrot,
\newblock Existence of multi-dimensional pulsating fronts for KPP equations: a new formulation approach,
\newblock \emph{Calc. Var. Partial Differential Equations} \textbf{62} (2023), Paper No.~134.

\bibitem{Fife1979}
P.~C.~Fife,
\newblock Mathematical aspects of reacting and diffusing systems,
\newblock Springer Science \& Business Media, 2013.

\bibitem{Fisher1937}
R.~A.~Fisher,
\newblock The wave of advance of advantageous genes,
\newblock \emph{Ann. Eugenics} \textbf{7} (1937), 355--369.

\bibitem{Freidlin1995}
M.~Freidlin,
\newblock Wave-front propagation for KPP-type equations,
\newblock \emph{Surveys in Applied Mathematics}, vol.~2, Springer US, Boston, MA, 1995, 1--62.

\bibitem{GuoMorita2005}
J.-S.~Guo and Y.~Morita,
\newblock Entire solutions of reaction--diffusion equations and an application to discrete diffusive equations,
\newblock \emph{Discrete Contin. Dyn. Syst.} \textbf{12} (2005), 193--212.

\bibitem{GuoWang2025}
H.~Guo and K.~Wang,
\newblock Some new bistable transition fronts with changing shape,
\newblock \emph{Math. Ann.} \textbf{392} (2025), 3797--3850.

\bibitem{GuoWangZhou}
H.~Guo, K.~Wang, and M.~Zhou,
\newblock Critical endpoint asymptotics for radial entire solutions of the Fisher--KPP equation,
\newblock in preparation.

\bibitem{HadelerRothe1975}
K.~P.~Hadeler and F.~Rothe,
\newblock Travelling fronts in nonlinear diffusion equations,
\newblock \emph{J. Math. Biol.} \textbf{2} (1975), 251--263.

\bibitem{HamelNadin2022}
F.~Hamel and G.~Nadin,
\newblock Diameters of the level sets for reaction-diffusion equations in nonperiodic slowly varying media,
\newblock \emph{Proc. Amer. Math. Soc.} \textbf{150} (2022), 3549--3564.

\bibitem{HamelRoques2011}
F.~Hamel and L.~Roques,
\newblock Uniqueness and stability properties of monostable pulsating fronts,
\newblock \emph{J. Eur. Math. Soc.} \textbf{13} (2011), 345--390.

\bibitem{HamelRossi2015}
F.~Hamel and L.~Rossi,
\newblock Admissible speeds of transition fronts for non-autonomous
monostable equations,
\newblock \emph{SIAM J. Math. Anal.} \textbf{47} (2015), 3342--3392.

\bibitem{HamelRossi2016}
F.~Hamel and L.~Rossi,
\newblock Transition fronts for the Fisher--KPP equation,
\newblock \emph{Trans. Amer. Math. Soc.} \textbf{368} (2016), 8675--8713.

\bibitem{HamelZhang2025}
F.~Hamel and M.~Zhang,
\newblock KPP transition fronts in a one-dimensional two-patch habitat,
\newblock \emph{Math. Med. Biol.} \textbf{42} (2025), 71--97.

\bibitem{HN1999}
F.~Hamel and N.~Nadirashvili,
\newblock Entire solutions of the KPP equation,
\newblock \emph{Comm. Pure Appl. Math.} \textbf{52} (1999), 1255--1276.

\bibitem{HN2001}
F.~Hamel and N.~Nadirashvili,
\newblock Travelling fronts and entire solutions of the Fisher--KPP equation in $\mathbb R^N$,
\newblock \emph{Arch. Ration. Mech. Anal.} \textbf{157} (2001), 91--163.

\bibitem{HuangShen2009}
J.~Huang and W.~Shen,
\newblock Speeds of spread and propagation for KPP models in time almost
and space periodic media,
\newblock \emph{SIAM J. Appl. Dyn. Syst.} \textbf{8} (2009), 790--821.

\bibitem{Kakutani1941}
S.~Kakutani,
\newblock Concrete representation of abstract $(M)$-spaces (A characterization of the space of continuous functions),
\newblock \emph{Ann. of Math. (2)} \textbf{42} (1941), 994--1024.

\bibitem{Kametaka1976}
Y.~Kametaka,
\newblock On the nonlinear diffusion equation of Kolmogorov--Petrovskii--Piskunov type,
\newblock \emph{Osaka J. Math.} \textbf{13} (1976), 11--66.

\bibitem{Kanel1961}
Ya.~I.~Kanel',
\newblock Some problems involving burning-theory equations,
\newblock \emph{Soviet Math. Dokl.} \textbf{2} (1961), 48--51.

\bibitem{KPP1937}
A.~N.~Kolmogorov, I.~G.~Petrovskii, and N.~S.~Piskunov,
\newblock A study of the diffusion equation with increase in the amount of substance, and its application to a biological problem,
\newblock \emph{Bull. Moscow Univ. Math. Mech.} \textbf{1} (1937), 1--25.

\bibitem{LiangZhao2007}
X.~Liang and X.~Q.~Zhao,
\newblock Asymptotic speeds of spread and traveling waves for monotone
semiflows with applications,
\newblock \emph{Comm. Pure Appl. Math.} \textbf{60} (2007), 1--40.

\bibitem{LiangZhao2010}
X.~Liang and X.~Q.~Zhao,
\newblock Spreading speeds and traveling waves for abstract monostable
evolution systems,
\newblock \emph{J. Funct. Anal.} \textbf{259} (2010), 857--903.

\bibitem{LiangZhou2022}
X.~Liang and T.~Zhou,
\newblock Propagation of KPP equations with advection in one-dimensional almost periodic media and its symmetry,
\newblock \emph{Adv. Math.} \textbf{407} (2022), Paper No.~108568.

\bibitem{LinZhang2019}
F.~Lin and Q.~S.~Zhang,
\newblock On ancient solutions of the heat equation,
\newblock \emph{Comm. Pure Appl. Math.} \textbf{72} (2019), 2006--2028.

\bibitem{Murray1989}
J.~D.~Murray,
\newblock Mathematical Biology,
\newblock Springer-Verlag, Berlin, 1989.

\bibitem{Nadin2009}
G.~Nadin,
\newblock Travelling fronts in space-time periodic media,
\newblock \emph{J. Math. Pures Appl.} \textbf{92} (2009), 232--262.

\bibitem{NadinRossi2017}
G.~Nadin and L.~Rossi,
\newblock Generalized transition fronts for one-dimensional almost periodic Fisher--KPP equations,
\newblock \emph{Arch. Ration. Mech. Anal.} \textbf{223} (2017), 1239--1267.

\bibitem{NinomiyaTaniguchi2025}
H.~Ninomiya and M.~Taniguchi,
\newblock Traveling front solutions of dimension $n$ generate entire solutions of dimension $(n-1)$ in reaction--diffusion equations as the speeds go to infinity,
\newblock \emph{Arch. Ration. Mech. Anal.} \textbf{249} (2025), Paper No.~13.

\bibitem{NolenRuddXin2005}
J.~Nolen, M.~Rudd, and J.~Xin,
\newblock Existence of KPP fronts in spatially-temporally periodic advection
and variational principle for propagation speeds,
\newblock \emph{Dyn. Partial Differ. Equ.} \textbf{2} (2005), 1--24.

\bibitem{Pazy1967}
A.~Pazy,
\newblock Asymptotic expansions of solutions of ordinary differential equations in Hilbert space,
\newblock \emph{Arch. Ration. Mech. Anal.} \textbf{24} (1967), 193--218.

\bibitem{Rothe1981}
F.~Rothe,
\newblock Convergence to pushed fronts,
\newblock \emph{Rocky Mountain J. Math.} \textbf{11} (1981), 617--633.

\bibitem{SalakoShen2019}
R.~B.~Salako and W.~Shen,
\newblock Long-time behavior of random and nonautonomous Fisher--KPP equations. Part II. Transition fronts,
\newblock \emph{Stoch. Dyn.} \textbf{19} (2019), Paper No.~1950046.

\bibitem{Shen2001}
W.~Shen,
\newblock Dynamical systems and traveling waves in almost periodic structures,
\newblock \emph{J. Differential Equations} \textbf{169} (2001), 493--548.

\bibitem{Shen2004}
W.~Shen,
\newblock Traveling waves in diffusive random media,
\newblock \emph{J. Dynam. Differential Equations} \textbf{16} (2004), 1011--1060.

\bibitem{Shen2010}
W.~Shen,
\newblock Variational principle for spreading speeds and generalized
propagating speeds in time almost periodic and space periodic KPP models,
\newblock \emph{Trans. Amer. Math. Soc.} \textbf{362} (2010), 5125--5168.

\bibitem{Shen2011a}
W.~Shen,
\newblock Existence, uniqueness, and stability of generalized traveling waves
in time-dependent monostable equations,
\newblock \emph{J. Dynam. Differential Equations} \textbf{23} (2011), 1--44.

\bibitem{Shen2011b}
W.~Shen,
\newblock Existence of generalized traveling waves in time-recurrent
and space-periodic monostable equations,
\newblock \emph{J. Appl. Anal. Comput.} \textbf{1} (2011), 69--93.

\bibitem{ShigesadaKawasakiTeramoto1986}
N.~Shigesada, K.~Kawasaki, and E.~Teramoto,
\newblock Traveling periodic waves in heterogeneous environments,
\newblock \emph{Theor. Popul. Biol.} \textbf{30} (1986), 143--160.

\bibitem{Stokes1976}
A.~N.~Stokes,
\newblock On two types of moving front in quasilinear diffusion,
\newblock \emph{Math. Biosci.} \textbf{31} (1976), 307--315.

\bibitem{Taniguchi2025}
M.~Taniguchi,
\newblock Polyhedral entire solutions in reaction--diffusion equations,
\newblock \emph{J. Differential Equations} \textbf{429} (2025), 529--565.

\bibitem{Uchiyama1978}
K.~Uchiyama,
\newblock The behavior of solutions of some semilinear diffusion equation for large time,
\newblock \emph{J. Math. Kyoto Univ.} \textbf{18} (1978), 453--508.

\bibitem{Weinberger2002}
H.~F.~Weinberger,
\newblock On spreading speeds and traveling waves for growth and migration
models in a periodic habitat,
\newblock \emph{J. Math. Biol.} \textbf{45} (2002), 511--548.

\bibitem{Widder1975}
D.~V.~Widder,
\newblock The Heat Equation,
\newblock Pure and Applied Mathematics, Vol.~67, Academic Press, New York--London, 1975.

\bibitem{Yagisita2003}
H.~Yagisita,
\newblock Backward global solutions characterizing annihilation dynamics of travelling fronts,
\newblock \emph{Publ. Res. Inst. Math. Sci.} \textbf{39} (2003), 117--164.

\bibitem{Zlatos2012}
A.~Zlato\v{s},
\newblock Transition fronts in inhomogeneous Fisher--KPP reaction-diffusion equations,
\newblock \emph{J. Math. Pures Appl.} \textbf{98} (2012), 89--102.
\end{thebibliography}
\end{document}